\documentclass{amsart}
\usepackage{amsmath,amssymb,amsthm,geometry,graphics,color,mdwlist,mathrsfs,tikz}
\usepackage[all]{xy}

\DeclareMathOperator*{\colim}{colim}

\DeclareMathOperator{\Hom}{Hom}

\DeclareMathOperator{\RHom}{\mathrm{R}\!\Hom}
\DeclareMathOperator{\cHom}{\mathcal{H}\hspace{-0.3pt}om}
\DeclareMathOperator{\End}{End}

\DeclareMathOperator{\REnd}{\mathrm{R}\!\End}

\def\lotimes{\otimes^\mathrm{L}}

\DeclareMathOperator{\rT}{\mathrm{T}}

\DeclareMathOperator{\Ext}{Ext}

\DeclareMathOperator{\id}{inj.dim}

\DeclareMathOperator{\gd}{gl.dim}
\DeclareMathOperator{\gldim}{gl.dim}

\DeclareMathOperator{\cone}{cone}

\DeclareMathOperator{\md}{mod}
\renewcommand{\mod}{\md}
\DeclareMathOperator{\Mod}{Mod}

\DeclareMathOperator{\proj}{proj}

\DeclareMathOperator{\thick}{thick}
\DeclareMathOperator{\per}{per}
\DeclareMathOperator{\pvd}{pvd}

\DeclareMathOperator{\ind}{ind}
\DeclareMathOperator{\add}{add}

\DeclareMathOperator{\silt}{\mathrm{silt}}

\DeclareMathOperator{\ctilt}{\mathrm{ctilt}}
\newcommand\ct[1]{#1\text{-\!}\ctilt}

\DeclareMathOperator{\soc}{soc}

\newcommand\recollement[3]{\xymatrix{{#1}\ar[r]&{#2}\ar[r]\ar@/_7pt/[l]\ar@/^7pt/[l]&{#3}\ar@/_7pt/[l]\ar@/^7pt/[l] }}
\newcommand\recollementwithmaps[9]{\xymatrix{{#1}\ar[r]|-{#8}&{#2}\ar[r]|-{#5}\ar@/_7pt/[l]_-{#7}\ar@/^7pt/[l]^-{#9}&{#3}\ar@/_7pt/[l]_-{#4}\ar@/^7pt/[l]^-{#6} }}
\def\A{\mathcal{A}}
\def\B{\mathcal{B}}
\def\C{\mathcal{C}}
\def\D{\mathcal{D}}

\def\F{\mathcal{F}}

\def\K{\mathcal{K}}
\def\M{\mathcal{M}}
\def\N{\mathcal{N}}

\def\P{\mathcal{P}}

\def\T{\mathcal{T}}
\def\U{\mathcal{U}}
\def\V{\mathcal{V}}
\def\W{\mathcal{W}}
\def\X{\mathcal{X}}
\def\Y{\mathcal{Y}}
\def\cZ{\mathcal{Z}}

\def\rL{\mathrm{L}}

\def\sA{\underline{A}}
\def\G{\Gamma}
\def\Ga{\Gamma}

\def\e{\varepsilon}

\def\La{\Lambda}

\def\sPi{\underline{\Pi}}

\def\al{\alpha}

\def\be{\beta}

\def\la{\lambda}

\def\si{\sigma}

\def\Z{\mathbb{Z}}
\def\ZZ{\mathbb{Z}}

\def\fd{\mathrm{fd}}
\def\op{\mathrm{op}}
\def\dg{\mathrm{dg}}
\def\lsimeq{\rotatebox{90}{$\simeq$}}
\def\rsimeq{\rotatebox{-90}{$\simeq$}}
\def\xsimeq{\xrightarrow{\simeq}}
\def\ysimeq{\xleftarrow{\simeq}}

\def\vsubset{\rotatebox{90}{$\subset$}}

\newtheorem{Thm}{Theorem}[section]
\newtheorem{Lem}[Thm]{Lemma}
\newtheorem{Prop}[Thm]{Proposition}
\newtheorem{Cor}[Thm]{Corollary}

\newtheorem{Prop-Def}[Thm]{Proposition-Definition}
\newtheorem{Thm-Def}[Thm]{Theorem-Definition}
\newtheorem{Cj}[Thm]{Conjecture}

\theoremstyle{definition}
\newtheorem{Def}[Thm]{Definition}
\newtheorem{Ex}[Thm]{Example}

\newtheorem{Qs}[Thm]{Question}

\newtheorem{Pb}[Thm]{Problem}

\theoremstyle{remark}
\newtheorem{Rem}[Thm]{Remark}

\newcounter{step}

\def\disoplus{\displaystyle\bigoplus}

\makeatletter

\@addtoreset{equation}{section}
\makeatother

\def\sPi{\underline{\Pi}}
\def\sA{\underline{A}}
\def\Lotimes{\otimes^{\mathrm{L}}}
\def\bb{b}
\def\dctilt{d\text{-\!}\ctilt}

\def\Q{\mathcal{Q}}

\def\TT{\mathbb{T}}
\DeclareMathOperator{\sect}{sect}
\author{Norihiro Hanihara}
\address{Faculty of Mathematics, Kyushu University, 744 Motooka, Nishi-ku, Fukuoka, 819-0395, Japan}
\email{hanihara@math.kyushu-u.ac.jp}

\author{Osamu Iyama}
\address{Graduate School of Mathematical Sciences, University of Tokyo, 3-8-1 Komaba Meguro-ku Tokyo 153-8914, Japan}
\email{iyama@ms.u-tokyo.ac.jp}
\thanks{The first author is supported by JSPS KAKENHI Grant Numbers JP22KJ0737 and JP25K17233. The second author is supported by JSPS Grant-in-Aid for Scientific Research (B) 22H01113, (B) 23K22384.
The first author is supported by Kavli Institute for the Physics and Mathematics of the Universe (WPI), The University of Tokyo, as an affiliate member.}
\subjclass{18G80, 18G35, 16G10, 16E35}
\keywords{silting object, $d$-silting object, cluster tilting object, cluster category, ($\F$-)liftable Calabi-Yau dg algebra}

\begin{document}
\title{Silting correspondences and Calabi-Yau dg algebras}

\begin{abstract}
This paper is devoted to studying two important classes of objects in triangulated categories, that is, silting objects and $d$-cluster tilting objects ($d\ge1$), and their correspondences.
%We study the correspondences between silting objects in the derived categories and cluster tilting objects in the cluster category.

%Thanks to the maps above,  We prove this result independently in the setting of triangulated categories.

First, we introduce the notion of \emph{$d$-silting objects} as a generalization tilting objects whose endomorphism algebras have global dimension at most $d$. For a smooth dg algebra $A$ and its $(d+1)$-Calabi-Yau completion $\Pi$, we show that the functor $-\lotimes_A\Pi$ gives an embedding from the poset ${\silt}^dA$ of $d$-silting objects of $A$ to the poset $\silt\Pi$ of silting objects of $\Pi$. 
Moreover, under the assumption that $H^0\Pi$ is finite dimensional, this functor identifies the Hasse quiver of $\silt^dA$ as a full subquiver of the Hasse quiver of $\silt\Pi$. In this case, we also prove that each $P\in\silt^dA$ gives a $d$-cluster tilting subcategory of $\per A$ as the $\nu[-d]$-orbit of $P$.
%For example, if $A$ is an $e$-Calabi-Yau dg algebra with $e\le d$, then we obtain a bijection $-\lotimes_A\Pi:\silt A=\silt^dA\simeq\silt\Pi$. 
%$\add\{(\nu[-d])^i(P)\mid i\in\Z\}$ of $\per A$.

%Consequently, we obtain a $d$-cluster tilting subcategory in the derived category $\per A$ as a $\nu[-d]$-orbit of a $d$-silting object $P$, which we prove independently in the setting of triangulated categories.
%More generally, for a triangulated category $\T$ with Serre functor $\nu$ and satisfying mild conditions, we prove that each $d$-silting object $P$ in $\T$ gives rise to a $d$-cluster tilting subcategory $\add\{(\nu[-d])^i(P)\mid i\in\Z\}$ of $\T$. 

Secondly, for a connective $(d+1)$-Calabi-Yau dg algebra $\Pi$ and its cluster category $\C(\Pi)$, we study the map from $\silt\Pi$ to the set $d\text{-}\operatorname{ctilt}\C(\Pi)$ of $d$-cluster tilting objects in $\C(\Pi)$ given by the canonical functor $\per\Pi\to\C(\Pi)$.
%gives a map 
%we study the relationship between $\silt\Pi$ and $d$-cluster tilting objects in the cluster category $\ct{d}\C(\Pi)$ of $\Pi$.
We call $\Pi$
%\emph{liftable} if the induced map $\silt\Pi\cap\F\to\ct{d}\C(\Pi)$ is surjective, and
\emph{$\F$-liftable} if the induced map $\silt\Pi\cap\F\to\ct{d}\C(\Pi)$  is bijective, where $\F$ is the fundamental domain in $\per\Pi$. 
We give basic properties of $\F$-liftable Calabi-Yau dg algebras.
We prove that $\F$-liftable Calabi-Yau dg algebras $\Pi$ such that $H^0\Pi$ is hereditary are precisely the Calabi-Yau completions of hereditary algebras. 
As an application, we obtain counter-examples to an open question posed in \cite{IYa1}.
We also study Calabi-Yau dg algebras such that the map $\silt\Pi\to\ct{d}\C(\Pi)$ is surjective, which we call {\it liftable}.

We explain our results by polynomial dg algebras and Calabi-Yau completions of type $A_2$.
\end{abstract}

\maketitle
\setcounter{tocdepth}{1}
\tableofcontents

\section{Introduction}

The notion of tilting objects/subcategories is fundamental to study equivalences of derived categories.
The notion of silting objects/subcategories is a generalization of the notion of tilting objects from a point of view of mutation, see Section \ref{section: silting}, and the notion of $d$-cluster tilting objects/subcategories is a $d$-Calabi-Yau analogue of the notion of tilting objects, see Section \ref{section: cluster tilting}.
A standard class of Calabi-Yau triangulated categories with cluster tilting objects is formed by cluster categories, which are constructed as the cosingularity categories of Calabi-Yau dg algebras (e.g.\ \cite{Am09,Guo,IYa1,Ke11}).
They play a central role in several branches of mathematics e.g.\ the categorification of cluster algebras (e.g.\ \cite{Ke10,KeY,Pa,Pl,Re}), quadratic differentials on marked surfaces (e.g.\ \cite{BS,HKK,KQ2}), Cohen-Macaulay representation theory, non-commutative crepant resolutions, higher dimensional Auslander-Reiten theory (e.g.\ \cite{HaI,Iy2,IO,IW,KaY,V,We}), and so on.
Also there are remarkable recent developments on Calabi-Yau dg algebras including their relative variants (e.g.\ \cite{BD19,KeL,Pr,Wu23,Ye16}).

Throughout this paper, let $k$ be an arbitrary field.
We study the correspondences between silting objects and $d$-cluster tilting objects/subcategories.
The prototypes of such correspondences are given by the following results.
%The following correspondences plays important role.
%for a finite dimensional algebra $A$ of finite global dimension, the following results giving relationships between titling objects/cluster tilting subcategories in the derived category $\per A$, and cluster tilting objects in the cluster category $\C_2(A)$, are well-established.
\begin{itemize}
\item[(I)] \cite[Theorem 3.3]{BMRRT} For an acyclic quiver $Q$, the canonical functor $\per kQ\to\C_2(kQ)$ sends each tilting $kQ$-module $T\in\mod kQ$ to a $2$-cluster tilting object $T\in\C_2(kQ)$.
%Moreover, $\add\{\nu_2^i(T)\mid i\in\Z\}$ is a $2$-cluster tilting subcategory of $\per kQ$.
\item[(II)] \cite[Proposition 2.5]{DI} Let $d\ge2$. For a finite dimensional algebra $A$ and a tilting complex $T\in\per A$ with $\gldim\End_{\D(A)}(T)<d$, $\add\{\nu_d^i(T)\mid i\in\Z\}$ is a $d$-cluster tilting subcategory of $\per A$.
%From derived category of $kQ$ to cluster category of $kQ$
\item[(III)] \cite[Theorem 3.5]{Am09} For a Ginzburg dg algebra $\Gamma$ of a Jacobi-finite quiver with potential $(Q,W)$, the canonical functor $\per\Gamma\to\C(Q,W)$ sends the silting object $\Gamma\in\per\Gamma$ to the $2$-cluster tilting object $\Gamma\in\C(Q,W)$.
\end{itemize}

The aims of this paper are to prove general results which imply (I) and (II) respectively, and also to study the correspondence (III) in a general setting.
In Section \ref{silt corresp}, we explain our results which generalize (I) and (II), which we call {\it silting-cluster tilting correspondences} and {\it silting-silting correspondences}. Also, in Section \ref{liftables}, we introduce the key notion of {\it ($\F$-)liftable dg algebras} and explain our results on the correspondence (III).

\subsection{Silting correspondences}\label{silt corresp}
We introduce the following key notion (see Definition \ref{define d-silting} for equivalent conditions, and also \ref{non-proper} for a more general version), which generalizes the notion of $d$-tilting objects studied in \cite{Iy,DI,HIMO}.

\begin{Def}
    Let $\T$ be a triangulated category with a Serre functor $\nu$, and $d\ge0$ an integer. A silting subcategory $\P$ of a triangulated category $\T$ is called \emph{$d$-silting} if $\Hom_{\T}(\nu(\P),\P[i])=0$ for each $i>d$.
\end{Def}

For example, for a finite dimensional algebra $A$ of finite global dimension, 
%the above condition is equivalent to $\Hom_{\D(A)}(DA,A[i])=0$ for $i>d$, and hence
$A$ is a $d$-silting object of $\per A$ if and only if $\gldim A\leq d$.
We denote by $\silt^d\T$ the set of additive equivalence classes of $d$-silting subcategories of $\T$.
For a triangulated category $\C$, we denote by $\ct{d}\C$ the set of $d$-cluster tilting subcategories of $\C$.

%\footnote{Different from `$d$-term silting object'.}

Our first main result of this paper generalizes the result (II) above, which played a key role to construct $d$-representation-finite selfinjective algebras in \cite{DI}.
%The first main result of this paper generalizes the result (II) above.

\begin{Thm}[=Theorem \ref{from tilting to silting 2}, Silting-CT correspondence]\label{from tilting to silting 2: intro}
	Let $d\ge1$ be an integer, $\T$ a $k$-linear Hom-finite Krull-Schmidt triangulated category with a Serre functor $\nu$. Assume that $\T$ is $\nu_d$-finite (see Definition \ref{define nu_d finite}) and admits a silting object and adjacent t-structures (see Definition \ref{define adjacent t-structure}). Then we have a map
	\[\silt^d\T\to\ct{d}\T\ \mbox{ given by }\ P\mapsto \U_d(P):=\add\{\nu_d^i(P)\mid i\in\Z\}.\]
\end{Thm}

%Recall that the {\it $d$-cluster category} of a finite dimensional algebra $A$ of global dimension at most $d$ is the triangulated hull $\C_d(A)$ of the orbit category $\per A/-\lotimes_ADA[-d]$. Under the projection functor $\per A\to\C_d(A)$, the $d$-cluster tilting subcategory $\U_d(A)\subset\per A$ in Theorem \ref{from tilting to silting 2: intro} above is mapped to a $d$-cluster tilting subcategory $\add A\subset\C_d(A)$. As its generalization, we prove the following result, which generalizes \cite[Proposition 3.2]{AOce}.

%\begin{Prop}[$\subset$Proposition \ref{AO}]\label{AO: intro}
%Let $d\geq1$ be an integer, $\C$ and $\D$ be triangulated categories with Serre functors $\nu$, and $\pi:\D\to\C$ a triangle functor which commutes with $\nu$ and preserves $d$-rigidity. Assume that there is $\U\in\dctilt\D$ such that $\pi(\U)\in\dctilt\C$ and $\pi(\U)$ has an additive generator. Then we have a map
%\[\dctilt\D\to\dctilt\C\ \mbox{ given by }\ \V\mapsto\pi(\V).\]
%\end{Prop}

Recall that a dg algebra $\Pi$ is {\it $(d+1)$-Calabi-Yau} if there is an isomorphism $\RHom_{\Pi^e}(\Pi,\Pi^e)[d+1]\simeq\Pi$ in $\D(\Pi^e)$.
For example, the Ginzburg dg algebra $\Ga(Q,W)$ of a quiver with potential $(Q,W)$ is a $3$-Calabi-Yau dg algebra.
For a $(d+1)$-Calabi-Yau dg algebra $\Pi$, the corresponding \emph{cluster category} (also called as the \emph{cosingularity category}) is defined by
\[\C(\Pi):=\per\Pi/\pvd\Pi,\]
the Verdier quotient of the perfect derived category by the perfectly valued derived category.
Under the assumption that $\Pi$ is connective (that is, $H^{>0}\Pi=0$, or equivalently, $\Pi\in\per\Pi$ is silting) and \emph{$H^0$-finite} (that is, $H^0\Pi$ is finite dimensional, or equivalently, $\per\Pi$ is $\Hom$-finite \cite[Proposition 2.4]{Am09}), the category $\C(\Pi)$ is a $d$-Calabi-Yau triangulated category with $d$-cluster tilting object $\Pi$ \cite{Am09,Guo,IYa1}.
In fact, the quotient functor $\per\Pi\to\C(\Pi)$ gives rise to a map
\begin{equation}\label{intro-map}
\silt\Pi\to\ct{d}\C(\Pi).
\end{equation}

Calabi-Yau completion is a general construction of Calabi-Yau dg algebras introduced by Keller \cite{Ke11}.
For a smooth connective dg $k$-algebra $A$, its $(d+1)$-Calabi-Yau completion $\Pi=\Pi_{d+1}(A)$ is defined as the tensor dg algebra $\rT_A\Theta$ of the cofibrant resolution $\Theta$ of the $(A,A)$-bimodule $\RHom_{A^e}(A,A^e)[d]$, see Section \ref{subsection CY}.
It is in fact a $(d+1)$-Calabi-Yau dg algebra. Moreover, for a smooth proper connecitve dg algebra $A$, the cluster category $\C(\Pi)$ of its $(d+1)$-Calabi-Yau completion is equivalent to the triangulated hull $\C_d(A)$ of the orbit category $\per A/-\lotimes_ADA[-d]$ and we have the commutative diagram
\begin{equation}\label{3}
\xymatrix{
	\per A\ar[d]_{-\lotimes_A\Pi}\ar[dr]\\
    \per\Pi\ar[r]&\C(\Pi)\simeq\C_d(A). }
\end{equation}
Our second main result generalizes the result (I) above.

\begin{Thm}[=Theorem \ref{from d-silting to ctilt}]\label{from d-silting to ctilt: intro}
Let $A$ be a smooth dg $k$-algebra, $d\ge0$ an integer, and $\Pi$ the $(d+1)$-Calabi-Yau completion of $A$.
Let $P$ be an object in $\per A$.
\begin{enumerate}
%	\item $\thick_{\D(A)}P=\per A$ if and only if $\thick_{\D(\Pi)}(P\lotimes_A\Pi)=\per\Pi$.
\item{\rm (Silting-silting correspondence)} $P$ is a $d$-silting object in $\per A$ if and only if $P\lotimes_A\Pi$ is a silting object in $\per \Pi$. Moreover, we have an embedding $-\lotimes_A\Pi:\silt^dA\to\silt\Pi$ of posets.
%If $A$ is proper, then these conditions are equivalent to that $P$ is a $d$-silting object in $\per A$.
\item{\rm (Silting-CT correspondence)} Assume $d\ge1$ and that $A$ is $\nu_d$-finite. If $P$ satisfies the conditions in (1), then $P\lotimes_A\Pi$ is a $d$-cluster tilting object in $\C(\Pi)$. Thus we have a map $\silt^dA\to\ct{d}\C(\Pi)$.
\end{enumerate}
\end{Thm}

Consequently, the commutative diagram \eqref{3} of triangulated categories induces the one \eqref{diagram} below of silting objects and cluster tilting objects.
\begin{equation}\label{diagram}
	\xymatrix{
	\silt^dA\ar[rr]^-{{\rm Thm.\,\ref{from tilting to silting 2: intro}}}\ar[d]^-{{\rm Thm.\,\ref{from d-silting to ctilt: intro}}}_{-\lotimes_A\Pi}&&\ct{d}(\per A)\ar[d]^-{{\rm Prop.\,\ref{AO}}}\\
	\silt\Pi\ar[rr]^-{\eqref{intro-map}}&&\dctilt\C(\Pi)\simeq\dctilt\C_d(A) }
\end{equation}

We also prove that the map $-\lotimes_A\Pi:\silt^dA\to\silt\Pi$ commutes with mutation (see Proposition \ref{mutation compatible}). To this end, we prove the following result on the Hasse quivers of $\silt^dA$ and of $\silt\Pi$.

\begin{Thm}[=Theorem \ref{Hasse}(1)]
Let $A$ be a smooth dg $k$-algebra, $d\ge0$ an integer, and $\Pi$ the $(d+1)$-Calabi-Yau completion of $A$ such that $\Pi$ is $H^0$-finite and connective. Then the embedding of posets $-\lotimes_A\Pi:\silt^dA\to\silt\Pi$ identifies the Hasse quiver of $\silt^dA$ as a full subquiver of the Hasse quiver of $\silt\Pi$.
%\item If $-\lotimes_AA^\vee$ gives a fully faithful functor $\per A\to\per A$ (e.g.\ $A$ is proper), then the Hasse quiver of $\silt^dA$ has neither sinks nor sources.
%\end{enumerate}
\end{Thm}

\medskip
Having established the diagram \eqref{diagram}, we propose to study the following important problem.
%It is important to study the following.

\begin{Qs}
	How far are the maps in \eqref{diagram} being from surjective?
%	Let $\Lambda$ be a smooth proper connective dg algebra such that $\per\Lambda$ is $\nu_d$-finite. Are all $d$-cluster tilting subcategory of $\per\Lambda$ given by $d$-silting objects?
\end{Qs}

The embedding $-\lotimes_A\Pi\colon\silt^dA\to\silt\Pi$ in \eqref{diagram} is usually far from being surjective, see e.g.\ Theorem \ref{braid} and Section \ref{section: A_2}. However it is an isomorphism in the following special case.
%−⊗L
%A Π : siltdA→silt Π induces an embedding on the Hasse quivers.

\begin{Ex}[=Theorem \ref{CYs}]
Let $d\ge e\ge0$ be integers, $A$ a connective $e$-Calabi-Yau dg algebra, and $\Pi=\Pi_{d+1}(A)$.
\begin{enumerate}
\item We have a quasi-isomorphism $\Pi\simeq k[x]\otimes_kA$ of dg algebras.
\item $\silt A=\silt^eA$ holds, and we have isomorphism of posets
\[\xymatrix{-\lotimes_A\Pi:\silt A\ar@<.2em>[r]&\silt\Pi:-\lotimes_{\Pi}A\ar@<.2em>[l].}\]
\end{enumerate}
\end{Ex}

\subsection{Liftable and $\F$-liftable Calabi-Yau dg algebras}\label{liftables}
In the rest of this section, we concentrate on the map $\silt\Pi\to\dctilt\C(\Pi)$ in \eqref{diagram}, which has not been well understood for $(d+1)$-Calabi-Yau case for $d\ge3$.
%We pose the following conjecte.
A naive approach is to use the \emph{fundamental domain} of the cluster category in the derived category; recall that it is defined as  
\[ \F=\P\ast\P[1]\ast\cdots\ast\P[d-1] \subset\per\Pi, \]
where $\P=\add\Pi\subset\per\Pi$. It is known that the composition $\F\subset\per\Pi\to\C(\Pi)$ is an additive equivalence \cite{Am09,Guo,IYa1}. Thus we have an injection $\silt\Pi\cap\F\to\C(\Pi)$.

\begin{Def}
Let $\Pi$ be a $H^0$-finite connective $(d+1)$-CY dg algebra. We call $\Pi$ {\it $\F$-liftable} if the canonical map $\silt\Pi\cap\F\to\ct{d}\C(\Pi)$ is surjective (or equivalently, bijective).
%\begin{enumerate}
%\item We call $\Pi$ {\it liftable} if the canonical map $\silt\Pi\to\ct{d}\C(\Pi)$ is surjective.
%\item We call $\Pi$ {\it $\F$-liftable} if the canonical map $\silt\Pi\cap\F\to\ct{d}\C(\Pi)$ is surjective (or equivalently, bijective).
%\end{enumerate}
%\[ \xymatrix{\silt\Pi\cap[\Pi[d-1],\Pi]\ar[r]&\ct{d}\C(\Pi) } \]is a bijection.
\end{Def}

The following result shows that the $\F$-liftability is preserved by taking dg quotient by an idempotent.

\begin{Thm}[=Corollary \ref{redm2}]\label{redm2: intro}
	Let $\Pi$ be an $\F$-liftable $(d+1)$-Calabi-Yau dg algebra. For each idempotent $e$ of $H^0\Pi$, the dg quotient of $\Pi$ by $e$ is again an $\F$-liftable $(d+1)$-Calabi-Yau dg algebra. 
%direct summand $P$ of $\Pi$ in $\per A$, the dg quotient $\Pi/P$ is again an $\F$-liftable $(d+1)$-Calabi-Yau dg algebra. 
%	Let $(\T,\T^\fd,\M)$ be a $(d+1)$-CY triple and $\P\subset\M$ a functorially finite subcategory. Let $(\U,\U^\fd,\N)$ be the silting reduction of the CY triple with respect to $\P$. If $(\T,\T^\fd,\M)$ is $\F$-liftable, then so is $(\U,\U^\fd,\N)$.
\end{Thm}

It is known that all $2$- and $3$-Calabi-Yau dg algebras are $\F$-liftable \cite{KN} (see \cite{IYa1}). Moreover it was asked in \cite{IYa1} if this is also the case for any $d$-Calabi-Yau dg algebras.
We give a contrary answer to this question. We present systematic counter-examples, showing that $\F$-liftable Calabi-Yau dg algebras are quite rare. 
%In this paper, we give a number of counter examples. More strongly, we show that $\F$-liftable Calabi-Yau dg algebras are quite rare.
In fact, we give a complete classification of $\F$-liftable Calabi-Yau dg algebras $\Pi$ such that $H^0\Pi$ is hereditary.

%\begin{Thm}[$\subset$]\label{H0=k: intro}
%Let $d\ge 1$ be an integer, and $\Pi$ a connective $(d+1)$-Calabi-Yau dg $k$-algebra such that $H^0\Pi=k$. Then the following conditions are equivalent.
%\begin{enumerate}
%\renewcommand{\labelenumi}{(\alph{enumi})}
%\renewcommand{\theenumi}{\alph{enumi}}
%\item $\Pi$ is $\F$-liftable.
%%\item $\Pi\simeq\Pi[d]$ in $\C(\Pi)$.
%%\item\label{dco} $\ct{d}\C(\Pi)=\{\Pi,\ldots,\Pi[d-1]\}$.
%\item $\C(\Pi)$ is triangle equivalent to $\C_d(k)$, the $d$-cluster category of $k$.
%\item $\Pi$ is quasi-isomorphic to $k[x]$ with $\deg x=-d$ and with zero differentials.
%\end{enumerate}
%%In this case, we have $\Pi[d]\simeq\Pi$ holds in $\C(\Pi)$.
%\end{Thm}

%We next extend this simplest example in two directions.

%Theorem \ref{H0=k: intro} shows that the dg polynomial algebra $k[x]$ with zero differential, or the Calabi-Yau completion of $k$, is $\F$-liftable.
%We extend Theorem \ref{H0=k: intro} to the CY dg algebras $\Pi$ such that $H^0\Pi$ is hereditary.
\begin{Thm}[$\subset$Theorem \ref{hereditary}]\label{hereditary: intro}
Let $k$ be a perfect field, $d\ge2$ an integer, $\Pi$ be a $H^0$-finite connective $(d+1)$-Calabi-Yau dg algebra such that $A=H^0\Pi$ is hereditary. Then the following are equivalent.%Consider the following statements.
\begin{enumerate}
\renewcommand{\labelenumi}{(\alph{enumi})}
\renewcommand{\theenumi}{\alph{enumi}}
\item $\Pi$ is $\F$-liftable.
%\item\label{van} $H^i\Pi=0$ for each $2-d\le i\le-1$.
\item There is a triangle equivalence $\C(\Pi)\simeq\C_d(A)$ taking $\Pi$ to $A$.
\item $\Pi$ is quasi-equivalent to the $(d+1)$-CY completion $\Pi_{d+1}(A)$ of $A$.
\end{enumerate}
%Then the implications {\rm (\ref{qeq})$\Rightarrow$(\ref{m1})$\Rightarrow$(\ref{cda})} hold.
\end{Thm}

%As a special case, $\F$-liftable Calabi-Yau dg algebra $\Pi$ with $H^0(\Pi)$ is quasi-isomorphic to $k[x]$ with $\deg x=-d$ with trivial differential (see Theorem \ref{H0=k}).

In particular, dg (twisted) polynomial algebras with zero differentials are not $\F$-liftable unless it has only one variable, that is, the Calabi-Yau completion of $k$ (see Theorem \ref{H0=k}).
%, and show that most of dg polynomial algebras with zero differentials are Calabi-Yau dg algebras which are non-$\F$-liftable. 

\begin{Ex}[=Theorem \ref{poly}]\label{poly: intro}
Let $\Pi=k[x_0]\otimes_k\cdots\otimes_kk[x_l]$ with $\deg x_i=-a_i<0$ be the dg algebra with trivial differentials which is an $(l+a+1)$-Calabi-Yau dg algebra, where $a=\sum_{i=0}^la_i$.
\begin{enumerate}
	\item $l=0$ if and only if the map $\silt\Pi\cap\F\to (l+a)\text{-}\ctilt\C(\Pi)$ is bijective, that is, $\Pi$ is $\F$-liftable.
	%\item If $n=0$, then $\Pi$ is $\F$-liftable.
	\item If $l=1$, then the map $\silt\Pi\to (l+a)\text{-}\ctilt\C(\Pi)$ is bijective. %Thus $\Pi$ is liftable. %$\Pi$ is not $\F$-liftable but liftable.
	\item If $l>0$, then the map $\silt\Pi\to (l+a)\text{-}\ctilt\C(\Pi)$ is injective.
	\end{enumerate}
\end{Ex}

As an application of Theorem \ref{redm2: intro}, we obtain the following class of non-$\F$-liftable Calabi-Yau dg algebras.
\begin{Ex}[$\subset$ Theorem \ref{gl2}]\label{gl2: intro}
Let $A$ be a finite dimensional algebra of global dimension $2$.
Then for every integer $d\geq3$, the $(d+1)$-Calabi-Yau completion of $A$ is not $\F$-liftable.
\end{Ex}

Notice that $\F$-lifability is not preserved by dg Morita equivalences. For example, let $A$ be a finite dimensional hereditary algebra, and $B$ an algebra with $\gldim B=2$ which is derived equivalent to $A$. Then $\Pi_{d+1}(A)$ is $\F$-liftable by Theorem \ref{hereditary: intro}, but $\Pi_{d+1}(B)$ is not $\F$-liftable by Example \ref{gl2: intro} if $d\ge4$.

These results show that the notion of $\F$-liftablity is quite restrictive. This motivates us to study the following weaker notion. %than $\F$-liftability.

\begin{Def}
Let $\Pi$ be a $H^0$-finite connective $(d+1)$-Calabi-Yau dg algebra. We call $\Pi$ {\it liftable} if the canonical map $\silt\Pi\to\ct{d}\C(\Pi)$ is surjective.
\end{Def}

One of the consequences of liftability of $\Pi$ is that the numbers of indecomposable direct summands of $d$-cluster tilting objects in $\C(\Pi)$ are constant, see Conjecture \ref{constant} and Proposition \ref{liftable and constant}.
On the other hand, we observe that the finiteness of number of summands are always preserved, see Theorem \ref{CT with generator}.

It is clear that liftability is preserved by dg Morita equivalences. In fact, it is preserved by cluster equivalences more strongly.

\begin{Prop}[=Proposition \ref{BO}]\label{BO: intro}
Let $\Pi$ and $\Pi'$ be $H^0$-finite connective $(d+1)$-Calabi-Yau dg algebras such that there exists an algebraic equivalence $\C(\Pi)\simeq\C(\Pi^\prime)$ (see Section \ref{basic}).
If $\Pi$ is liftable, then $\Pi$ and $\Pi'$ are dg Morita equivalent, and therefore $\Pi'$ is also liftable.
\end{Prop}

We give the following sufficient condition for liftability.

\begin{Prop}[=Proposition \ref{connected liftable}]
Let $d\geq1$ and let $\Pi$ be an $H^0$-finite connective $(d+1)$-Calabi-Yau dg algebra.
\begin{enumerate}
\item If $\Pi$ is CT-connected, then $\Pi$ is liftable.
\item The image of $\silt\Pi\to\dctilt\C(\Pi)$ is a union of connected components.
\end{enumerate}
\end{Prop}

We finish this section by posing the following conjecture.

\begin{Cj}
Any $H^0$-finite connective Calabi-Yau dg algebra is liftable.
\end{Cj}

\subsection*{Conventions}
For morphisms $f:X\to Y$ and $g:Y\to Z$ in a category, we denote their composition by $gf:X\to Z$. Also, for arrows $a:i\to j$ and $b:j\to l$ in a quiver, we denote their composition by $ba:i\to l$.

Throughout this paper, let $k$ be an arbitrary field.
We denote by $D$ the $k$-dual.
For a $k$-linear $\Hom$-finite triangulated category $\T$, a \emph{Serre functor} is a $k$-linear autoequivalence $\nu:\T\to\T$ with a functorial isomorphism $D\Hom_{\T}(X,Y)\simeq\Hom_{\T}(Y,\nu(X))$ for each $X,Y\in\T$.

Let $\T$ be a triangulated category. For full subcategories $\X$ and $\Y$ of $\T$, we consider the full subcategory
\[\X*\Y=\{T\in\T\mid\exists X\to T\to Y\ \mbox{ with }\ X\in\X,\ Y\in\Y\}.\]

We say that dg categories $\A$ and $\B$ are {\it dg Morita equivalent} if there is an $(\A,\B)$-bimodule $X$ which induces an equivalence $-\lotimes_\A X\colon\D(\A)\xsimeq\D(\B)$.

\section{Silting and Cluster tilting subcategories}

\subsection{Silting subcategories}\label{section: silting}
Let $\T$ be a triangulated category. A full subcategory $\P$ of $\T$ is called \emph{presilting} if $\Hom_{\T}(\P,\P[i])=0$ for each $i\ge1$, and \emph{silting} if it is presilting satisfying $\T=\thick\P$. An object $P\in\T$ is called \emph{(pre)silting} if the full subcategory $\add P$ of $\T$ is (pre)silting.
We denote by $\silt\T$ the set of additive equivalence classes of silting subcategories of $\T$. Then $\silt\T$ has a canonical partial order given by $\P\ge\Q\Longleftrightarrow\Hom_{\T}(\P,\Q[i])=0$ for each $i\ge1$. For $\P,\Q\in\silt\T$, let
\[[\Q,\P]:=\{\M\in\silt\T\mid\Q\le\M\le\P\}.\]

Recall that a {\it co-$t$-structure} in a triangulated category $\T$ is a pair $(\T_{\geq0},\T_{\leq0})$ of full subcategories such that $\T_{\leq0}[1]\subset\T_{\leq0}$, $\Hom_\T(\T_{\geq1},\T_{\leq0})=0$ and $\T=\T_{\geq1}*\T_{\leq0}$,
%for every $T\in\T$ there is a triangle $X\to T\to Y$ with $X\in\T_{\geq1}$ and $Y\in\T_{\leq0}$,
where $\T_{\geq n}:=\T_{\geq0}[-n]$ for each $n\in\Z$.
%We fix a positive integer $d$ and put $\alpha_{d}:=\alpha[-d]$.
There exists a bijection between the set $\silt\T$
%the additive equivalence classes of silting subcategories $\P$ of $\T$
and the set of equivalence classes of bounded co-$t$-structures of $\T$ \cite{KY} given by
\[\P\mapsto(\T_{\ge0}^\P,\T_{\le0}^\P):=({}^\perp\P[>\!0],\P[<\!0]^\perp).\]
For $\P,\Q\in\silt\T$, $\P\geq\Q$ if and only if $\T^\P_{\geq0}\supset\T^\Q_{\geq0}$. Note that this is the case if and only if $\Hom_\T(\P,\Q[>\!0])=0$.
If $\P=\add\P$ holds (that is, $\P$ is maximal in its additive equivalence class), then the equality
\begin{equation}\label{generate T by P}
\T=\bigcup_{\ell\ge1}(\P[-\ell])*(\P[1-\ell])*\cdots*(\P[\ell-1])*(\P[\ell])
\end{equation}
holds, and the corresponding co-$t$-structure can be described as \cite[Proposition 2.8]{IYa1}
\begin{equation}\label{co-t-structure by P}
\T_{\ge0}^\P=\bigcup_{\ell\ge1}(\P[-\ell])*\cdots*(\P[-1])*\P\ \mbox{ and }\ \T_{\le0}^\P=\bigcup_{\ell\ge1}\P*(\P[1])*\cdots*(\P[\ell]).
\end{equation}

Let $\T$ be a triangulated category with a silting object $T$. For $U\in\add T$, assume that there exists a right $(\add U)$-approximation $f:U_0\to T$. Extending $f$ to a triangle (called the \emph{exchange triangle})
\[T^\prime\to U_0\xrightarrow{f} T,\]
we define the {\it right mutation} of $T$ as
\[\mu_U^+(T):=T^\prime\oplus U.\]
%where $T^\prime$ is defined by }, that is, a triangle
The {\it left mutation} $\mu_U^-(T)$ is defined dually. Then $\add\mu_U^\pm(T)$ is independent of a choice of $f$, and any $\mu_U^\pm(T)$ is again a silting object in $\T$ \cite[Theorem 2.31]{AI}.

Next, we recall the process of {\it silting reduction}. Let $\T$ be a triangulated category with a presilting subcategory $\P$ satisfying the following conditions.
\begin{itemize}
\item[(P1)] $\P$ is covariantly finite in ${}^\perp\P[>\!0]$ and contravariantly finite in $\P[<\!0]^\perp$.
\item[(P2)] For any $X\in\T$ , we have $\Hom_{\T}(X,\P[i]) = 0 = \Hom_{\T}(\P,X[i])$ for $i\gg0$.
\end{itemize}
%which is functorially finite in $\T$. Let $\cZ:=\P[<\!0]^\perp\cap{}^\perp\P[>\!0]$ and $\U:=\T/\thick\P$. Let \[\silt_{\P}\!\T:=\{\M\in\silt\T\mid\P\subset\M\}.\]

\begin{Prop-Def}[{\cite[2.1, 3.6]{IYa1}}]\label{define reduction}
	Let $\T$ be a triangulated category with a presilting subcategory $\P$ satisfying (P1) and (P2). Let $\cZ:=\P[<\!0]^\perp\cap{}^\perp\P[>\!0]$ and $\U:=\T/\thick\P$. 
    \begin{enumerate}
		\item We have $\P\subset\cZ$ and the additive quotient $\cZ/[\P]$ has a natural structure of a triangulated category with suspension $\langle 1\rangle$ defined as follows: for each $X\in\cZ$, form a triangle $X\xrightarrow{f} P\to X\langle 1\rangle\to X[1]$ with $f$ a left $\P$-approximation.	
		\item The functor $\cZ\subset\T\to\U$ induces a triangle equivalence $\cZ/[\P]\xrightarrow{\simeq}\U$.
		\item The quotient functor $\T\to\U$ induces an isomorphism of posets
        \[\silt_\P\!\T:=\{\M\in\silt\T\mid\P\subset\M\}\to\silt\U.\]
	\end{enumerate}
\end{Prop-Def}

\begin{Prop}\label{nterm}
	In the setting of Proposition-Definition \ref{define reduction}, let $\N$ be the image of $\M$ in $\U$.
    The projection functor $\T\to\U$ induces an isomorphism of posets for all $n\geq0$:
	\[\silt_\P\!\T\cap[\M[n],\M]\xrightarrow{\simeq}[\N[n],\N].\]
\end{Prop}

To prove this, we need the following preparation. For subcategories $\A$ and $\B$ in an additive category $\C$, consider the subcategory $\A\oplus\B:=\add\{A\oplus B\mid A\in\A, B\in\B\}\subset\C$.

\begin{Lem}\label{int}
	Write $\M=\P\oplus\Q$ with $\P\cap\Q=0$. Then we have $\silt_\P\!\T\cap[\M[n],\M]=[\Q\left\langle n\right\rangle \oplus\P,\M]$.
\end{Lem}
\begin{proof}
	For each $Q\in\Q$ there are triangles $Q\langle i\rangle \to P^i\to Q\langle i+1\rangle\to Q\langle i\rangle [1]$
	%\[ \xymatrix{ Q\langle i\rangle \ar[r]&P^i\ar[r]&Q\langle i+1\rangle\ar[r]&Q\langle i\rangle [1] } \]
	for $0\leq i\leq n-1$. Then $\Q\left\langle n\right\rangle \oplus\P=\mu^n_{\P}(\M)\in\silt\T$ holds.
	By construction, we have
	\begin{equation}\label{eqin}
		\Q\langle n\rangle\subset\P\ast\cdots\ast\P[n-1]\ast\Q[n].
	\end{equation}
	
	We first show the inclusion `$\supset$'. We have $\Q\langle n\rangle\oplus\P\geq\M[n]$ by (\ref{eqin}), thus $[\M[n],\M]\supset[\Q\left\langle n\right\rangle \oplus\P,\M]$.	Also, since both $\Q\langle n\rangle\oplus\P$ and $\M$ contain $\P$ as a subcategory, so does every silting subcategory between them, thus $\silt_\P\!\T\supset[\Q\left\langle n\right\rangle \oplus\P,\M]$.
	
	We next show the converse inclusion. Let $\N\in\silt_\P\!\T\cap[\M[n],\M]$, and we have to prove $\Q\langle n\rangle\oplus\P\leq\N$, that is, $\Hom_\T(\N,(\Q\langle n\rangle\oplus\P)[>\!0])=0$. Since $\N$ is a silting subcategory containing $\P$, we see $\Hom_\T(\N,\P[>\!0])=0$. Also, since $\N\geq\M[n]$ we have in particular $\Hom_\T(\N,\Q[>\!n])=0$. We then obtain the result by (\ref{eqin}).
\end{proof}

\begin{proof}[Proof of \ref{nterm}]
	By silting reduction \cite[3.7, 3.8]{IYa1} the functor $\cZ\subset\T\to\U$ induces an isomorphism of posets in the first row above. We have to show that it restricts to the intervals in the second row.
	\[ \xymatrix@R=5mm{
		\rho\colon\silt_\P\!\T\ar[r]^-\simeq&\silt\U\\
		\silt_\P\!\T\cap[\M[n],\M]\ar@{}[u]|-\vsubset\ar@{-->}[r]&[\N[n],\N]\ar@{}[u]|-\vsubset } \]
	By \ref{int} the lower left corner is equal to $[\Q\left\langle n\right\rangle \oplus\P,\M]$. We obtain the conclusion since $\rho(\M)=\N$ and $\rho(\Q\left\langle n\right\rangle \oplus\P)=\N[n]$.
\end{proof}

As an application of Proposition \ref{nterm}, we prove the following result. 

\begin{Thm}\label{mutations}
	Let $\T$ be a $k$-linear $\Hom$-finite idempotent complete triangulated category with a silting object $A$. Let $U$ be a presilting object with $|U|=|A|-1$. Then for all $n\ge0$, we have
	\[ \sharp(\silt_U\!\T\cap[A[n],A])\leq n+1.\]
\end{Thm}

\begin{proof}
	We can assume that there exists an indecomposable object $X\in\T$ such that $T:=X\oplus U\in[A[n],A]$.
    Then $\silt_U\T=\{\mu^i_U(T)\mid i\in\Z\}$ holds by \cite[Corollary 3.9]{IYa1}.
    
    We first claim that the set $S=\{\mu_U^i(T)\mid i\in\ZZ\}\cap[A[n],A]$ is finite. By $T[n]\leq A[n]$ we have
	\[ \{\mu_U^i(T)\mid i\in\ZZ\}\cap[A[n],T]\subset\{\mu_U^i(T)\mid i\in\ZZ\}\cap[T[n],T]. \]
	Now the right-hand-side equals $\{X\langle n\rangle\oplus U,\ldots,X\langle1\rangle\oplus U, X\oplus U\}$ by Proposition \ref{nterm}, thus it is finite. It follows that $\sharp(\{\mu_U^i(T)\mid i\in\ZZ\}\cap[A[n],T])\leq n+1$ for all $T\in[A[n],A]$, and we see that we must have $\mu^i_U(T)\not\leq A$ for sufficiently large $i$, which proves finiteness of $S$.
	
	Now take the maximum element $T\in S$. Then $S=\{\mu_U^i(T)\mid i\in\ZZ\}\cap[A[n],T]$, whose cardinality is bounded by $n+1$.
\end{proof}

\begin{Rem}
The equality in Theorem \ref{mutations} holds in some cases: Assume $U\in[A[n],A]$ without loss of generality. If (1) $n=1$ or (2) $\T=\per A$ for a hereditary algebra $A$, then the equality holds. On the other hand, the equality often fails in general: If $\T=\per kQ$ for quiver of type $A_2$, then $[A[2],A]=\add\{X_i\mid 1\le i\le 7\}$, and $\{\mu_{X_4}^i(X_3\oplus X_4)\mid i\in\ZZ\}\cap[A[2],A]=\{X_3\oplus X_4, X_4\oplus X_5\}$.
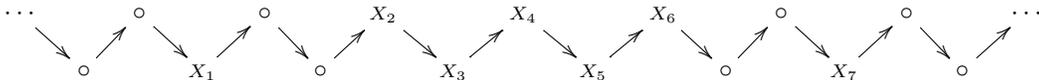
\[ \xymatrix@R=1em@C=1em{
	\cdots\ar[dr]&&\circ\ar[dr]&&\circ\ar[dr]&&{\scriptstyle X_2}\ar[dr]&&{\scriptstyle X_4}\ar[dr]&&{\scriptstyle X_6}\ar[dr]&&\circ\ar[dr]&&\circ\ar[dr]&&\cdots\\
&\circ\ar[ur]&&{\scriptstyle X_1}\ar[ur]&&\circ\ar[ur]&&{\scriptstyle X_3}\ar[ur]&&{\scriptstyle X_5}\ar[ur]&&\circ\ar[ur]&&{\scriptstyle X_7}\ar[ur]&&\circ\ar[ru] } \]
\end{Rem}

We will discuss the behavior of certain properties of silting objects under mutations. For this we first recall the following result.
\begin{Prop}[{\cite[Proposition 2.36]{AI}}]\label{AI}
Let $\T$ be a Krull-Schmidt triangulated category and $P\in\silt\T$. For each $T\in\T_{\le0}^P$, take $P_i\in\add P$ such that $T\in P_0*P_1[1]*\cdots*P_\ell[\ell]$. If $X$ is a direct summand of $P$ such that $\add X\cap\add P_0=0$, then $T\in\T_{\le0}^{\mu_Y^-(P)}$ holds for the complement $Y$ of $X$ in $P$.
\end{Prop}

Let $\T$ be a triangulated category, and $G:\T\to\T$ a fully faithful triangle functor.
Let $\silt^G\T$ be the set of all $P\in\silt\T$ satisfying $\Hom_{\T}(P,GP[i])=0$ for each $i>0$, in other words, $GP\in\T_{\le0}^P$.

\begin{Prop}\label{F-silting mutation}
Let $\T$ be a Krull-Schmidt triangulated category, $G:\T\to\T$ a fully faithful triangle functor, and $P\in\silt^G\T$ such that each summand is mutable. If $\add P\neq\add GP$, then there exists a decomposition $P=X\oplus Y$ with $X$ indecomposable such that $\mu^-_Y(P)\in\silt^G\T$.
\end{Prop}
\begin{proof}
	By our assumption $GP\in\T_{\le0}^P$, there exist $P_i\in\add P$ such that $GP\in P_0*P_1[1]*\cdots*P_\ell[\ell]$, which we may take so that $\ell$ is minimal. We claim that there is a direct summand of $P$ which does not lie in $\add P_0$.
	Since $\add P\neq\add GP$ by assumption, we must have $\ell>0$. Also, since $G$ is fully faithful, we have $\Hom_\T(GP,GP[l])\ysimeq\Hom_\T(P,P[l])=0$, and thus \cite[Lemma 2.25]{AI} shows $\add P_\ell\cap\add P_0=0$. It follows that any direct summand of $P_\ell$ does not lie in $\add P_0$.
	This claim shows that there is an indecomposable direct summand $X$ of $P$ such that $X\not\in\add P_0$. Then Proposition \ref{AI} gives the desired result.
\end{proof}

\subsection{Cluster tilting subcategories}\label{section: cluster tilting}

Let $\C$ be a triangulated category and $d$ a positive integer.
A full subcategory $\U\subset\C$ is called {\it $d$-rigid} if $\Hom_{\C}(\U,\U[i])=0$ for each $1\le i\le d-1$,
%A full subcategory $\U\subset\C$ satisfying $\U=\add\U$ is 
and \emph{$d$-cluster tilting} if it satisfies the following equivalent conditions.
\begin{itemize}
\item $\U$ is $d$-rigid and satisfies $\U=\add\U$ and  $\C=\U*\U[1]*\cdots*\U[d]$.
\item $\U=\bigcap_{i=1}^{d-1}\U[-i]^\perp$ and $\U$ is contravariantly finite in $\T$.
%\{X\in\T\mid\forall 1\le i\le d-1,\ \Hom_{\T}(\P,X[i])=0\}$.
\item $\U=\bigcap_{i=1}^{d-1}{}^\perp\U[i]$ and $\U$ is covariantly finite in $\T$.
\item $\U=\bigcap_{i=1}^{d-1}\U[-i]^\perp=\bigcap_{i=1}^{d-1}{}^\perp\U[i]$ and $\U$ is functorially finite in $\C$.
\end{itemize}
%$\P=\bigcap_{i=1}^{d-1}{}^\perp\P[i]=\bigcap_{i=1}^{d-1}\P[-i]^\perp$ and $\P$ is functorially finite in $\T$. 
We denote by $\dctilt\C$ the set of $d$-cluster tilting subcategories of $\C$.

We call an object $U_0\in\U$ a \emph{$\nu_d$-additive generator} of $\U$ is $\U=\add\{\nu_d^i(U_0)\mid i\in\Z\}$.

\begin{Thm}\label{CT with generator}
Let $\C$ be a Krull-Schmidt triangulated category with a Serre functor $\nu$.
%\begin{enumerate}\item 
Assume that $\C$ has a $d$-cluster tilting subcategory which has an additive generator (respectively, $\nu_d$-additive generator). Then, each $d$-cluster tilting subcategory has an additive generator (respectively, $\nu_d$-additive generator).
%\item \end{enumerate}
\end{Thm}

%\begin{Thm}\label{drop weak}
%Let $\C$ be a Krull-Schmidt triangulated category with a Serre functor $\nu$. Assume that $\C$ has a $d$-cluster tilting subcategory which is locally bounded and has a $\nu_d$-additive generator.
%Then, each weak $d$-cluster tilting subcategory is a $d$-cluster tilting subcategory which is locally bounded and has a $\nu_d$-additive generator.
%%If $\V\subset\C$ is a subcategory satisfying $\V=\bigcap_{i=1}^{d-1}\V[-i]^\perp=\bigcap_{i=1}^{d-1}{}^\perp\V[i]$,
%%${X\in\T\mid\forall 1\le i\le d-1,\ \Hom_{\T}(\V,X[i])=0\}$, 
%then $\V$ is a $d$-cluster tilting subcategory which is locally bounded and has a $\nu_d$-additive generator.
%\end{Thm}

To prove these results, we recall the following basic fact.

\begin{Lem}\label{ast}
Let $\X,\Y$ be two subcategories of a triangulated category $\C$. If $\Hom_\C(\Y,\X[1])=0$ then $\X\ast\Y=\X\oplus\Y\subset\Y\ast\X$ holds.
\end{Lem}

\begin{Lem}\label{chain rule}
Let $\V$ be a $d$-rigid subcategory of a triangulated category. Then $\W:=\V*\V[1]*\cdots*\V[d-1]$ satisfies $\W[1-d]*\cdots*\W[-1]*\W\subset\V[-d+1]\ast\cdots\ast\V\ast\cdots\ast\V[d-1]$.
\end{Lem}

\begin{proof}
We have
\[\W[1-d]*\cdots*\W[-1]*\W=(\V[1-d]*\cdots*\V)*\cdots*(\V[-1]*\cdots*\V[d-2])*(\V*\cdots*\V[d-1]).\]
Since $\V$ is $d$-rigid, by applying Lemma \ref{ast} repeatedly, we obtain the assertion.
%	\begin{equation}\label{eq2bai}	\C=\V'[-d+1]\ast\cdots\ast\V'\ast\cdots\ast\V'[d-1].	\end{equation}
%Assume $X\in\T$ satisfies $\Hom_{\T}(\P,X[i])=0$ for each $1\le i\le d-1$. Since $\nu_d(\P)=\P$, the Serre duality implies $\Hom_{\T}(X,\P[i])=0$ for each $1\le i\le d-1$. Since $X\in\left(\P[1-d]*\cdots\P[-1]\right)*\left(\P*\cdots\P[d-1]\right)$,
%Take a triangle $Q\xto{f} X\to R\to Y[1]$ with $Q\in\P[1-d]*\cdots*\P[-1]$ and $R\in\P*\cdots*\P[d-1]$. the first vanishing condition implies $X\in\add(\P*(\P[1]\cdots*\P[d-1]))$, and the second vanishing condition implies $X\in\add\P=\P$.
\end{proof}

\begin{proof}[Proof of Theorem \ref{CT with generator}]
Let $\U$ and $\V$ be $d$-cluster tilting subcategories of $\C$, and $U\in\U$ an additive generator (respectively, $\nu_d$-additive generator).
Take $V_i\in\V$ for $0\le i\le d-1$ such that
\begin{equation}\label{define V_U^i}
    U\in V_0*V_1[1]*\cdots*V_{d-1}[d-1].
\end{equation}
We claim that $V:=\bigoplus_{i=0}^{d-1}V_i\in\V$ is an additive generator (respectively, $\nu_d$-additive generator).

Let $\V'=\add V$ (respectively, $\V':=\add\{\nu_d^i(V)\mid i\in\Z\}$). By construction, we have $\U\subset\V'*\V'[1]*\cdots*\V'[d-1]$. By Lemma \ref{chain rule},
\[\C=(\V'[1-d]*\cdots*\V'[-1])*\V'*(\V'[1]*\cdots*\V'[d-1]).\]
For each $X\in\V$, since $\Hom_{\C}(\V'[1-d]*\cdots*\V'[-1],X)=0$, we have $X\in\add(\V'*\cdots*\V'[d-1])$.
Since we have $\Hom_{\C}(X,\V'[1]*\cdots*\V'[d-1])=0$, we have $X\in\add\V'=\V'$. Thus $\V=\V'$ holds.
\end{proof}

\begin{Rem}\label{CT with generator2}
The above proof of Theorem \ref{CT with generator} shows the following assertion, which will be used later.

Let $\C$ be a Krull-Schmidt triangulated category with a Serre functor $\nu$.
Let $\U\subset\C$ be a $d$-cluster tilting subcategory, and $\V\subset\C$ a $d$-rigid subcategory satisfying $\V=\add\V$ and $\U\subset\V*\V[1]*\cdots*\V[d-1]$.
\begin{enumerate}
\item If $\U$ has an additive generator, then so does $\V$.
\item If $\U$ has a $\nu_d$-additive generator and $\V=\nu_d(\V)$ holds, then $\V$ has a $\nu_d$-additive generator.
\end{enumerate}
\end{Rem}

%Assume $X\in\T$ satisfies $\Hom_{\T}(\P,X[i])=0$ for each $1\le i\le d-1$. Since $\nu_d(\P)=\P$, the Serre duality implies $\Hom_{\T}(X,\P[i])=0$ for each $1\le i\le d-1$. Since $X\in\left(\P[1-d]*\cdots\P[-1]\right)*\left(\P*\cdots\P[d-1]\right)$,
%Take a triangle $Q\xto{f} X\to R\to Y[1]$ with $Q\in\P[1-d]*\cdots*\P[-1]$ and $R\in\P*\cdots*\P[d-1]$. the first vanishing condition implies $X\in\add(\P*(\P[1]\cdots*\P[d-1]))$, and the second vanishing condition implies $X\in\add\P=\P$.
%Since $\V=\nu_d(\V)$, we may assume by Lemma \ref{terms} that $V^{\nu_d(U)}_i\simeq\nu_d(V^U_i)$ for each $U$ and $i$. Consider the full subcategory
%\[\V':=\add\{V^U_i\mid U\in\U,\ 0\le i\le d-1\}\subset\V.\]
%By construction, we have $\U\subset\V'*\V'[1]*\cdots*\V'[d-1]$ and $\nu_d(\V')=\V'$.
%%Since $\U$ is $d$-cluster tilting and $\V'$ is $d$-rigid, Lemma \ref{chain rule} gives an equality 
%%	\begin{equation}\label{eq2bai}
%%	\C=\U[1-d]*\cdots*\U[-1]*\U=\V'[-d+1]\ast\cdots\ast\V'\ast\cdots\ast\V'[d-1].
%%	\end{equation}
%%Now assume that $X\in\C$ satisfies $\Hom_\C(\V',X[i])=0$ for $0<i<d$, or equivalently, $\Hom_\C(X,\V'[i])=0$ for $0<i<d$ (by Serre duality and $\nu_d(\V')=\V'$). Then (\ref{eq2bai}) shows $X\in\V'$.
%In particular, we obtain $\V=\V'$ since $\V\supset\V'$ is $d$-rigid. Thus the claim (1) follows.
%(2) For a ($\nu_d$-)additive generator $U\in\U$, take $V^U_i\in\V$ for each $0\le i\le d-1$ in \eqref{define V_U^i}. Then the object $V:=\bigoplus_{i=0}^{d-1}V^U_i\in\V$ is a ($\nu_d$-)additive generator.
%\end{proof}

%This is an immediate consequence of Lemma \ref{ct}(2).
%\end{proof}

We say that a $k$-linear, $\Hom$-finite Krull-Schmidt category $\A$ is {\it locally bounded} if for each $A\in\A$, we have $\sum_{B\in\ind\A}(\dim\Hom_\A(B,A)+\dim\Hom_\A(A,B))<\infty$, where $\ind$ means the set of indecomposable objects.

We prove the following dichotomy result for locally bounded $d$-cluster tilting subcategory.

\begin{Prop}\label{bunkai}
Let $\C$ be a Krull-Schmidt triangulated category with a Serre functor $\nu$, and let $\U\subset\C$ a $d$-cluster tilting subcategory which is locally bounded.
%Decompose $\U=\U^1\oplus \U^2$ so that each object of $\U^1$ is $\nu_d$-periodic and for each indecomposable object $U'\in\U^2$, the objects $\{\nu_d^i(U')\mid i\in\Z\}$ are mutually non-isomorphic. Put $\C_1:=\U^2[\Z]^\perp$ and $\C_2:=\U^1[\Z]^\perp$.
Then we have a decomposition $\C\simeq\C_p\times\C_a$ as a triangulated category such that each object in $\U\cap\C_p$ is $\nu_d$-periodic and each indecomposable object in $\U\cap\C_a$ is not $\nu_d$-periodic.
\end{Prop}

\begin{proof}
Decompose $\U=\U_p\oplus \U_a$ so that each object of $\U_p$ is $\nu_d$-periodic and each indecomposable object of $\U_a$ is not $\nu_d$-periodic.

We first show that $\U_p$ and $\U_a$ are mutually orthogonal, that is, we have
\[ \Hom_\C(U_a,U_p[n])=0=\Hom_\C(U_p,U_a[n]) \]
for all $U_p\in\U_p$, $U_a\in\U_a$ and $n\in\Z$.
We prove the first equality.
Take $0\neq m\in\Z$ such that $\nu_d^m(U_p)\simeq U_p$. Take $U_0,\ldots,U_{d-1}\in\U$ such that
\begin{equation}\label{resolution Up[n]}
U_p[n]\in U_0*\cdots*U_{d-1}[d-1]
\end{equation}
and the corresponding triangle $U_0\xrightarrow{f}U_p[n]\to Y$ with $Y\in U_1[1]*\cdots*U_{d-1}[d-1]$ satisfies that $f$ is a minimal right $\U$-approximation of $U_p[n]$.
%By Lemma \ref{terms}, we have $U_0\in\U_p$.
%holds for each $0\le j\le d-1$.
Then, since $\nu_d^m(f):\nu_d^m(U_0)\to\nu_d^m(U_p[n])\simeq U_p[n]$ is a minimal right $\U$-approximation, we have $\nu_d^m(U_0)\simeq U_0$. 
Thus, for each $i\in\Z$, we have
\[\Hom_\C(U_a,U_0)\simeq\Hom_\C(\nu_d^{im}(U_a),\nu_d^{im}(U_0))\simeq\Hom_\C(\nu_d^{im}(U_a),U_0).\]
Since $U_a\in\U_a$, the local boundedness of $\U$ forces $\Hom_\C(U_a,U_0)=0$. 
On the other hand, since $\U$ is $d$-rigid,  $\Hom_\C(U_a,U_j[j])=0$ holds for each $1\le j\le d-1$. Thus \eqref{resolution Up[n]} shows $\Hom_\C(U_a,U_p[n])=0$.
The second equality is proved similarly.

Consider the full subcategories
\begin{align*}
\C_p:=\U_a[\Z]^\perp\supset\C_p':=\U_p*\U_p[1]*\cdots*\U_p[d-1],\\ 
\C_a:=\U_p[\Z]^\perp\supset\C_a':=\U_a*\U_a[1]*\cdots*\U_a[d-1],
\end{align*}
where $\C_p$ and $\C_a$ are thick subcategories of $\C$.
Now we claim $\C=\C_p'\oplus\C_a'$. We have $\U_p[n]*\U_a[m]=\U_p[n]\oplus\U_a[m]=\U_a[n]*\U_p[m]$ for each $n,m\in\Z$. Since $\U$ is $d$-cluster tilting, we have
\begin{align*}
\C&=\U*\U[1]*\cdots\U[d-1]=(\U_p*\U_a)*(\U_p[1]*\U_a[1])*\cdots*(\U_p[d-1]*\U_a[d-1])\\
&=(\U_p*\U_p[1]*\cdots*\U_p[d-1])*(\U_a*\U_a[1]*\cdots*\U_a[d-1])
=\C_p'*\C_a'=\C_p'\oplus\C_a'.
\end{align*}
%For $i=p,a$, put $\C_i':=\U_i*\U_i[1]*\cdots*\U_i[d-1]$. We next claim $\C=\C_p'\oplus\C_a'$. Let $X\in\C$ be an indecomposable object and pick a resolution
%\[ \xymatrix{ X\ar[r]&U^0\ar[r]&U^1\ar[r]&\cdots\ar[r]&U^{d-1} } \]
%in $\C$ with $U^0,\ldots,U^{d-1}\in\U$. Decomposing each object $U^j=U^j_p\oplus U^j_a$ along $\U=\U_p\oplus \U_a$, we see by our first claim that each morphism $U^j_p\oplus U^j_a\to U^{j+1}_p\oplus U^{j+1}_a$ is a direct sum of $U^j_p\to U^{j+1}_p$ and $U^{j}_a\to U^{j+1}_a$. It follows that the above complex itself is also a direct sum of two complexes with terms in $\U_p$ and $\U_a$ respectively, and therefore we get a decomposition $X=X_p\oplus X_a$ with $X_p\in\C_p'$ and $X_a\in\C_a'$.
%for each $i=p,a$. %\U_i*\U_i[1]*\cdots*\U_i[d-1]$.
%Finally, we see from the first step that we have $\C_p'\subset\U_a[\Z]^\perp=\C_p$ and $\C_a'\subset\U_p[\Z]^\perp=\C_a$. 
Finally, we claim $\C_p=\C_p'$ and $\C_a=\C_a'$, which imply $\C=\C_p\times\C_a$ as a triangulated category.
By definition, $\C_p\cap\C'_a=0$ holds.
%For each $X\in\C_p$, take a decomposition $X=X_p\oplus X_a$ with $X_p\in\C_p'$ and $X_a\in\C_a'$. Then $X_a\in\C_p\cap\C_a'$
For each $X\in\C_p$, take a decomposition $X=X_p\oplus X_a$ with $X_p\in\C_p'$ and $X_a\in\C_a'$. Then $X_a\in\C_p\cap\C_a'=0$ and hence $X=X_p\in\C_p'$. Thus $\C_p=\C_p'$ holds. We obtain $\C_a=\C_a'$ similarly.
%which is in particular a thick subcategory, and thus
\end{proof}

%\begin{Lem}
%Let $\C$ be a triangulated category and $F$ an autoequivalence of $\C$. Let $\U\subset\C$ a be subcategory such that $\thick\U=\C$.
%\begin{enumerate}
%\item The following conditions are equivalent.
%\begin{enumerate}
%\item For each pair of objects $X,Y\in\U$, we have $\Hom_\C(X,F^iY)=0$ for almost all $i\in\Z$.
%\item For each pair of objects $X,Y\in\C$, we have $\Hom_\C(X,F^iY)=0$ for almost all $i\in\Z$.
%\end{enumerate}
%\item %Suppose $\C$ is $k$-linear and $\Hom$-finite. If the conditions in (1) are satisfied, then the following are equivalent.
%\begin{enumerate}
%\item $\U$ has an $F$-additive generator.
%\item $\U$ is locally bounded.
%\item $\U\subset\C$ is functorially finite.
%\end{enumerate}
%\end{enumerate}
%\end{Lem}

A full subcategory $\U$ of a triangulated category $\C$ is called \emph{weak $d$-cluster tilting} if it satisfies $\U=\bigcap_{i=1}^{d-1}\U[-i]^\perp=\bigcap_{i=1}^{d-1}{}^\perp\U[i]$.
The following result shows that, under certain conditions, each weak $d$-cluster tilting subcategory is $d$-cluster tilting.

\begin{Prop}\label{ct}
Let $\C$ be a triangulated category with a Serre functor $\nu$, $\U\subset\C$ a $d$-cluster tilting subcategory, and $\V$ a $d$-rigid subcategory satisfying $\V=\add\V=\nu_d(\V)$ and $\U\subset\V*\V[1]*\cdots*\V[d-1]$.
\begin{enumerate}
\item The subcategory $\V\subset\C$ is weak $d$-cluster tilting.
\item If $\U$ has an additive generator, then $\V$ is a $d$-cluster tilting subcategory with an additive generator.
\item If $\U$ is locally bounded and has a $\nu_d$-additive generator, then $\V$ is a $d$-cluster tilting subcategory which is locally bounded and has a $\nu_d$-additive generator.
\end{enumerate}
%We have $\V=\bigcap_{i=1}^{d-1}\V[-i]^\perp=\bigcap_{i=1}^{d-1}{}^\perp\V[i]$.
%\[\begin{aligned}\V&=\{X\in\C\mid \Hom_\C(\V,X[i])=0 \text{ for all }0<i<d\}\\ &=\{X\in\C\mid \Hom_\C(X,\V[i])=0 \text{ for all }0<i<d\}.\end{aligned}\]
%\item If there exists an object $U_0\in\U$ such that $\U=\add\{\nu_d^i(U_0)\mid i\in\Z\}$, then there exists an object $V_0\in\V$ such that $\V=\add\{\nu_d^i(V_0)\mid i\in\Z\}$. 
\end{Prop}

To prove this, we prepare the following easy observations.

\begin{Lem}\label{nu_d-finite}
Let $\C$ be a triangulated category with a Serre functor $\nu$. Then (1)$\Rightarrow$(2)$\Rightarrow$(3) holds.
\begin{enumerate}
\item $\C$ has a $d$-cluster tilting subcategory which is locally bounded and each indecomposable object is not $\nu_d$-periodic.
\item For each $X,Y\in\C$, $\Hom_{\C}(X,\nu_d^i(Y))=0$ holds for almost all $i\in\Z$.
\item Each subcategory of $\C$ admitting a $\nu_d$-additive generator is locally bounded and functorially finite in $\C$.
\end{enumerate}
\end{Lem}

\begin{proof}
(1)$\Rightarrow$(2) First we claim that, for each $X\in\C$ and $U\in\U$, $\Hom_{\C}(X,\nu_d^i(U))=0$ for almost all $i\in\Z$. Since $\U$ is $d$-cluster tilting, there exists $U^0,\ldots,U^{d-1}\in\U$ such that $X\in U^{d-1}[1-d]*\cdots*U^0$. Then $\Hom_{\C}(U^j[-j],\nu_d^i(U))=0$ holds for all $1\le j\le d-1$ and $i\in\Z$, and $\Hom_{\C}(U^0,\nu_d^i(U))=0$ holds for almost all $i\in\Z$. Thus the claim follows.

(2)$\Rightarrow$(3) Assume that $\V\subset\C$ has a $\nu_d$-additive generator $V$, that is, $\V=\add\{\nu_d^i(V)\mid i\in\Z\}$.
For each $X\in\C$, the condition (2) implies $\Hom_{\C}(\nu_d^i(V),X)=0$ for almost all $i\in\Z$. Thus the assertion follows.
%Since $\U$ is $d$-cluster tilting, there exists $U_0,\ldots,U_{d-1}\in\U$ such that $Y\in U_0*\cdots*U_{d-1}[d-1]$. For each $0\le j\le d-1$, we have $\Hom_{\C}(X,\nu_d^i(U_j)[j])=0$ for almost all $i\in\Z$ by the first claim. Thus $\Hom_{\C}(X,\nu_d^i(Y))=0$ holds for almost all $i\in\Z$.
\end{proof}

\begin{proof}[Proof of Proposition \ref{ct}]
(1) Since $\U$ is $d$-cluster tilting and $\V$ is $d$-rigid, Lemma \ref{chain rule} gives an equality 
	\begin{equation}\label{eq2bai}
	\C=\U[1-d]*\cdots*\U[-1]*\U=\V[-d+1]\ast\cdots\ast\V\ast\cdots\ast\V[d-1].
	\end{equation}
Now assume that $X\in\C$ satisfies $\Hom_\C(\V,X[i])=0$ for $0<i<d$, or equivalently, $\Hom_\C(X,\V[i])=0$ for $0<i<d$ (by Serre duality and $\nu_d(\V)=\V$).
Then (\ref{eq2bai}) shows $X\in\V$. Thus $\V$ is weak $d$-cluster tilting.
%\begin{proof}[Proof of Theorem \ref{drop weak}]
%We have to prove that $\V$ is functorially finite in $\C$.

%Assume $\U,\V\in\dctilt\C$ and $\U$ is locally bounded and has an $\nu_d$-additive generator.
(2) By Remark \ref{CT with generator2}, $\V$ has an additive generator and hence functorially finite. By (1), it is $d$-cluster tilting.

(3) By Proposition \ref{bunkai}, we may assume that either (a) each object in $\U$ is $\nu_d$-periodic, or (b) each indecomposable object in $\U$ is not $\nu_d$-periodic.
In the case (a), $\U$ has an additive generator, and the assertion follows from (2).%$\V$ is functorially finite.

Consider the case (b). By Remark \ref{CT with generator2}(2), $\V$ has a $\nu_d$-additive generator. By Lemma \ref{nu_d-finite}(1)$\Rightarrow$(3), $\V$ is locally bounded and functorially finite. By (1), $\V$ is $d$-cluster tilting.
%decompose $\V=\V^1\oplus \V^2$ so that each object of $\V^1$ is $\nu_d$-periodic and for each indecomposable object $V'\in\U^2$, its $\nu_d$-orbits are mutually non-isomorphic. We denote by $\U=\U^1\oplus\U^2$ the same decomposition for $\U$.
\end{proof}

%\begin{Lem}
%Suppose as in \ref{ct}(2) that there is $V_0\in\V$ such that $\V:=\add\{\nu_d^i(U_0)\mid i\in\Z\}$ satisfies \ref{ct}(1). Then $\V$ is also locally bounded, and functorially finite in $\C$. Therefore, $\V\subset\C$ is $d$-cluster tilting.
%\end{Lem}

\section{Preliminaries on Calabi-Yau dg algebras, Calabi-Yau completions and cluster categories}

\subsection{Calabi-Yau dg algebras and cluster categories}\label{cluster}
We collect some basic results on cluster categories arising from Calabi-Yau dg algebras \cite{Am09,Guo}.

Let $A$ be a dg $k$-algebra, and $A^e:=A^{\op}\otimes_kA$ its enveloping algebra. We call $A$ \emph{smooth} if $A\in\per A^e$. For an integer $n$, we call $A$ \emph{$n$-Calabi-Yau} if it is smooth and $\RHom_{A^e}(A,A^e)\simeq A[-n]$ in $\D(A^e)$.

We call a dg algebra $\La$ \emph{connective} if $H^i\La=0$ holds for all $i>0$ (or equivalently, $\La\in\per\La$ is silting), and \emph{$H^0$-finite} if $H^0\La$ is finite dimensional (or equivalently, $\per\La$ is Hom-finite).
Throughout this section, let $d$ be a positive integer, and let $\Pi$ be a $(d+1)$-CY dg algebra over a field $k$ which is connective and $H^0$-finite.
Amiot \cite{Am09} and Guo \cite{Guo} introduced the {\it cluster category} as 
\[ \C(\Pi):=\per\Pi/\pvd\Pi, \]
the Verdier quotient of the perfect derived category $\per\Pi$ by the perfectly valued derived category $\pvd\Pi$ consisting of dg $\Pi$-modules of finite dimensional total cohomology.
The cluster category $\C(\Pi)$ is a $d$-Calabi-Yau triangulated category.
%and the image of $\Pi\in\per\Pi$ in $\C(\Pi)$ is a $d$-cluster tilting object.
Consider the canonical functor
\[ \xymatrix{ \pi:\per\Pi\ar[r]& \C(\Pi) }. \]
Then $\Pi$ is a silting object in $\per\Pi$ and a $d$-cluster tilting object in $\C(\Pi)$ \cite[2.1]{Am09}\cite[2.2]{Guo}. More generally, the functor $\pi$ sends each silting objet in $\per\Pi$ to a $d$-cluster tilting object in $\C(\Pi)$ \cite[5.12]{IYa1}. Therefore we have a map
%\begin{equation}\label{sct}
%	\xymatrix{\silt\Pi\ar[r]& \ct{d}\C(\Pi) },
%\end{equation}
%the canonical functor $\per\Pi\to\C(\Pi)$ gives a map
\begin{equation}\label{silt to dctilt}
\silt\Pi\to\dctilt\C(\Pi)
\end{equation}
where $\silt\Pi$ (resp. $\ct{d}\C(\Pi)$) is the set of isomorphism classes of silting objects in $\per\Pi$ (resp. $d$-cluster tilting objects in $\C(\Pi)$).

Recall the concept of mutations for silting objects from Section \ref{section: silting}.
We define left and right mutations $\mu_U^\pm(T)$ of a $d$-cluster tilting object $T$ in a triangulated category $\T$ by exactly the same formulas. Then they are again cluster tilting objects in $\T$ \cite[Theorem 5.1]{IYo}.

%Assume that $\T$ is Krull-Schmidt and $T$ is a basic silting (respectively, cluster tilting) object such that $T=X\oplus U$ with indecomposable $X$. Then 
%If $X$ is indecomposable, we call silting (respectively, cluster tiliting) mutation $\mu^\pm(T)$ \emph{irreducible}.

\begin{Prop}\label{pm}
	The map $\pi:\silt\Pi\to\dctilt\Pi$ commutes with mutations. More precisely, for each $P\in\silt\Pi$ and $Q\in\add P$, both $\mu_Q^\pm(P)$ and $\mu_{\pi Q}^\pm(\pi P)$ exist and we have $\pi\mu_Q^-(P)=\mu_{\pi Q}^-(\pi P)$ and $\pi\mu_Q^+(P)=\mu_{\pi Q}^+(\pi P)$.
\end{Prop}

\begin{proof}
	We only prove the assertion for $\mu^-$. Replacing $\Pi$ by $\REnd_\Pi(P)$, we can assume $P=\Pi$. By \eqref{negative H}, the functor $\pi:\per\Pi\to\C(\Pi)$ restricts to an equivalence $\add P\simeq\add\pi P$. Let $X\xrightarrow{f} Q'\to Y\to X[1]$ be an exchange triangle, that is, $f$ is a minimal left $(\add Q)$-approximation. Then $\pi f:\pi P\to\pi Q'$ is a minimal left $(\add\pi Q)$-approximation, thus $\pi X\xrightarrow{\pi f}\pi Q'\to\pi Y\to\pi X[1]$ is an exchange triangle. Therefore $\mu_{\pi Q}^-(\pi P)=\pi Y\oplus \pi Q=\pi\mu_Q^-(P)$.
	%Consider the fundamental domain $\F_P$ in $\per\Pi$. Then each term in the exchange triangle $X\to Q'\to Y\to X[1]$ belongs to $\F_P$.
\end{proof}

One can enhance the Verdier quotient to a dg quotient \cite{Ke98,Dr04} as follows. We denote by $\per_\dg\!\Pi$ (resp. $\pvd_\dg\!\Pi$) the canonical dg enhancement of $\per\Pi$ (resp. $\pvd\Pi$). Precisely, we may take $\per_\dg\!\Pi$ to be the smallest full dg subcategory of $\C_\dg(\Pi)$, the dg category of dg $\Pi$-modules, containing $\Pi$ and closed under taking shifts, mapping cones, and direct summands (in $\K(\Pi)$).
Now, let $\C(\Pi)_\dg:=\per_\dg\!\Pi/\pvd_\dg\!\Pi$ be the dg quotient and we define $\Ga$ to be the endomorphism ring of $\Pi$ in $\C(\Pi)_\dg$, so that there is a (unital) morphism $\Pi\to\Ga$ of dg algebras which is a localization and such that we have a commutative diagram
%the induction functor
\begin{equation}\label{induction}
\xymatrix@R=5mm@C=10mm{ \per\Pi\ar[r]\ar@{=}[d]&\C(\Pi)\ar[d]^\rsimeq\\
\per\Pi\ar[r]^{-\lotimes_\Pi\Ga}&\per\Ga\\
}
\end{equation}
%identifies with the Verdier quotient $\per\Pi\to\C(\Pi)$,
and therefore $\Ga\in\per\Ga$ is $d$-cluster tilting.
%We call this $\Ga$ the {\it cluster localization} \comment{better name?} of $\Pi$.
By \cite[Proposition 5.9]{IYa1}, the canonical map 
\begin{equation}\label{negative H}
	H^{-i}\Pi=\Hom_{\D(\Pi)}(\Pi[i],\Pi)\to\Hom_{\C(\Pi)}(\Pi[i],\Pi)
\end{equation}
is an isomorphism for each $i\geq0$. More generally, for each silting object $M\in\per\Pi$, the canonical map
\begin{equation}\label{negative H2}
	H^{-i}\REnd_\Pi(M)=\Hom_{\D(\Pi)}(M[i],M)\to\Hom_{\C(\Pi)}(M[i],M)
\end{equation}
is an isomorphism for each $i\geq0$. These translate into the following in terms of $\Ga$.
For a complex $X=(\cdots\to X^{i-1}\to X^i\to X^{i+1}\to\cdots)$ and $n\in\Z$, we denote by $X^{\leq n}$ the cohomological truncation $(\cdots\to X^{n-1}\to Z^nX\to 0\to \cdots)$. When $\La$ is a dg algebra, the truncated complex $\La^{\leq0}$ is a dg subalgebra.
\begin{Lem}\label{leq0}
	%Let $\Pi\to\Ga$ be the canonical localization such that $\per\Ga=\C(\Pi)$.
	We have $\Pi=\Ga^{\leq0}$. More generally for any silting object $M\in\per\Pi$ we have $\REnd_\Pi(M)=\REnd_\Ga(M\lotimes_\Pi\Ga)^{\leq0}$.
\end{Lem}
\begin{proof}
	By \eqref{negative H2} and \eqref{induction}, we have isomorphisms for each $i\geq0$:
	\[\Hom_{\D(\Pi)}(M[i],M)\simeq\Hom_{\C(\Pi)}(M[i],M)\simeq\Hom_{\D(\Ga)}(M\lotimes_\Pi\Ga[i],M\otimes_\Pi\Ga).\qedhere\]
%It follows that $\Pi$ and $\Ga$ has the same negative cohomologies. 
%\[\xymatrix{
%\per\Pi\ar[r]\ar[d]&\per\Gamma\ar[d]\\
%\per\REnd_\Pi(M)\ar[r]&\per\REnd_\Gamma(M\lotimes_\Pi\Gamma)
%}\]	
%Replacing $\Pi$ by $\REnd_\Pi(M)$ the second assertion follows.
\end{proof}

\subsection{Calabi-Yau completions and dg orbit algebras}\label{subsection CY}
In this section, we recall the construction of Calabi-Yau completions and the description of their cluster categories.

For a given smooth dg $k$-algebra $A$ and an integer $d$, 
%Let $A$ be a smooth connective dg $k$-algebra, and $d$ an integer.
%Assume that $\gl A\le d$ for some positive integer $d$.
%Let $B=A\op DA[-d-1]$ be the trivial extension dg $k$-algebra. 
%\old{The \emph{$d$-cluster category} $\C_{d}(A)$ is  defined as the triangulated hull of the orbit category $\D^{\rm b}(\mod A)/\nu [-d]$.}
%Let $\Theta$ be a $\mathcal H$-projective resolution of the dg $A$-bimodule $\RHom_{A^e}(A,A^e)$ and let $\theta:=\Theta[d]$.
let $\Theta\to \RHom_{A^e}(A,A^e)[d]$ be a cofibrant resolution of the $d$-shifted inverse dualizing dg bimodule, and we define the \emph{$(d+1)$-Calabi-Yau completion} (also known as the \emph{derived $(d+1)$-preprojective algebra}) as the tensor algebra:
\[ \Pi:=\Pi_{d+1}(A)=T_{A}(\Theta). \]
Then $\Pi$ is a $(d+1)$-CY dg algebra by \cite[Theorem 4.8]{Ke11}, see also \cite{Ke11+}.

For such Calabi-Yau dg algebras, its cluster category $\C(\Pi_{d+1}(A))$ can be described as (the triangulated hull of) the orbit category of the derived category $\per A$ as follows.  
For a positive integer $d$, we say that a proper dg algebra $A$ is {\it $\nu_d$-finite} if $H^{\geq0}\Pi_{d+1}(A)$ is finite dimensional. We shall later generalize the notion of $\nu_d$-finiteness in Definitions \ref{define nu_d finite} and \ref{non-proper}.
Let $\Pi=\Pi_{d+1}(A)$ be the $(d+1)$-Calabi-Yau completion of a dg algebra $A$. Then the natural inclusion $A\subset\Pi$ gives a triangle functor $-\lotimes_{A}\Pi:\per A\to \per\Pi$. 

For an algebraic triangulated category $\T=\per A$ given by a dg algebra $A$ and a triangle autoequivalence $F=-\lotimes_A X$ given by a cofibrant $A^e$-module $X$, the canonical triangulated hull of the orbit category $\T/F$ is defined as the derived category $\per\Gamma(A,X)$, where
\[ \Ga(A,X)=\colim\left( \xymatrix{\disoplus_{n\geq0}\cHom_A(A,A\otimes_A X^n)\ar[r]^-{-\otimes_A X}&\disoplus_{n\geq0}\cHom_A(A\otimes_A X,A\otimes_A X^n)\ar[r]^-{-\otimes_A X}&\cdots }\right)  \]
is the dg orbit algebra, see \cite{Ke05} and also \cite[Section 2]{HaI} for details. It admits an Adams grading whose degree $i$ part is the image of $\cHom_A(A\otimes_AX^m,A\otimes_A X^{i+m})$ for $m\gg0$. %$\Gamma(A,X)=\bigoplus_{n\in\Z}\oplus\Gamma(A,X)_n$.

For a smooth dg algebra $A$ and an integer $d$, we consider $\Ga(A,\Theta)$ for the cofibrant resolution $\Theta\to\RHom_{A^e}(A,A^e)[d]$. Then we have $\Ga(A,\Theta)_{\ge0}=\Pi_{d+1}(A)$, and thus a morphism $\Pi_{d+1}(A)\to\Ga(A,\Theta)$.
If moreover $A$ is connective, proper and $\nu_d$-finite, then we define the {\it $d$-cluster category} of $A$ as %the triangulated hull
\[ \C_d(A):=\per\Ga(A,\Theta), \]
thus it is the triangulated hull of the orbit category $\per A/-\lotimes_A\Theta$. 
%The image of $A\in\per A$ in $\C_d(A)$ is a $d$-cluster tilting object. 
The following result gives a description of the cluster category of the Calabi-Yau completion in terms of the derived category.
\begin{Thm}[\cite{Am09}]
Let $A$ be a smooth, connective, proper dg algebra which is $\nu_d$-finite, and let $\Pi=\Pi_{d+1}(A)$ be its $(d+1)$-Calabi-Yau completion. The induction along the morphism $\Pi\to\Ga(A,\Theta)$ induces an equivalence
\[ \C(\Pi)\xsimeq\C_d(A). \]
%We denote this category by $\C_d(A)$ and call it the {\rm $d$-cluster category of $A$}.
\end{Thm}
Similarly, if $U$ is an $A^e$-module such that $U\lotimes_A\cdots\lotimes_AU\simeq \Theta$ for some $a\geq1$ (where the left-hand-side has $a$ tensor factors), we define the {\it $a$-foled cluster category} of $A$ as
\[ \C_d^{(1/a)}(A):=\per\Ga(A,U).\]
%(\per A/-\lotimes_AU)_\triangle. \]

%If $A$ is a smooth proper dg algebra, then the triangulated category $\per A$ has a Serre functor $-\otimes_A\Theta^{-1}[d]$. We will not assume that $A$ is proper, in particular $\per A$ does not have a Serre functor. 
%%In the previous section, we have discussed the notions of {\it $d$-silting objects} (\ref{define d-silting}) and {\it $\nu_d$-finitenss} (\ref{nu_d-finite}) of silting subcategories in a triangulated category, under the existence of a Serre functor $\nu$.
%In this setting, we generalize the notions of {$d$-silting objects} (\ref{define d-silting}) and {$\nu_d$-finitenss} (\ref{nu_d-finite}) of silting objects in a triangulated category to $\per A$ by replacing the inverse Serre functor $\nu^{-1}$ by the inverse dualizing bimodule $\RHom_{A^e}(A,A^e)$.

%\subsection{Cluster categories as orbit categories}
%We next recall another description of the cluster category $\C(\Pi)$ from the previous subsection. 
%\comment{Notice that, if $A$ is non-proper, then $\C(\Pi)$ does not contain the orbit category $\per A/\nu_d$ as a full subcategory. We need to modify $\C(\Pi)$ by changing the denominator.}

\section{Correspondences between cluster tilting and silting}\label{A}

We consider the following setting.
\begin{itemize}
\item $A$ is a smooth proper connective dg algebra over $k$.
\item $d\geq1$ is an integer such that $A$ is $\nu_d$-finite.
\item $\Pi$ is the $(d+1)$-Calabi-Yau completion of $A$.
\end{itemize}
We have recalled in the previous section that a Calabi-Yau dg algebra $\Pi$ yields the cluster category $\C(\Pi)$ which is equipped with a cluster tilting object.
When $\Pi$ is the $(d+1)$-Calabi-Yau completion of $A$, then by \cite{Ke05} and \cite{Guo}, there is a triangle equivalence $\C_{d}(A)\simeq\C(\Pi)$ making the following diagram commutative, where the horizontal functors are the canonical functors.
\begin{equation}\label{square}
	\xymatrix@C5em{\per A \ar[r]\ar[d]_{-\lotimes_{A}\Pi}&\C_{d}(A)\ar[d]^-\rsimeq\\
			\per\Pi \ar[r] & \C(\Pi)}
\end{equation}
Since $\Pi$ is a $d$-cluster tilting object in $\C(\Pi)$, the commutativity of \eqref{square} shows that $A$ is a $d$-cluster tilting object in $\C_d(A)$.

%Recall that we say $P\in \T$ is a $d$-silting, if $\nu_{d}(\T_{\ge 0})\subset \T_{\ge 0}$.
%\begin{Rem}
%Let $A$ be a proper smooth dg algebra, and $\T=\per A$. Then Theorem \ref{from tilting to silting} follows from Theorem \ref{from d-silting to ctilt}(2).
%\end{Rem}
The aim of this section is to study correspondences between silting and cluster tilting objects as in the following diagram.
\[
\xymatrix{
	\silt^dA\ar[rr]^-{{\rm Cor.\,\ref{from silting to cluster tilting}}}\ar[d]_-{{\rm Thm.\,\ref{from d-silting to ctilt}}}&&\ct{d}(\per A)\ar[d]^-{{\rm Thm.\,\ref{ao}}}\\
	\silt\Pi\ar[rr]^-{\rm Section\, \ref{liftable}}&&\dctilt\C_d(A) }
\]

\subsection{Silting-CT correspondence in derived categories via orbit construction}\label{Section:silting-CT}
In this section, we study the correspondence between silting objects and cluster-tilting subcategories in the derived category $\per A$ of a (dg) algebra $A$.
Let us place ourselves in a more general setting. 
%\old{Let $\T$ be a $k$-linear $\Hom$-finite Krull-Schmidt triangulated category with Serre functor $\nu$.} 

%\begin{Prop-Def}\label{characterize d-silting}
%	Let $\alpha$ be an autoequivalence of $\T$. A silting subcategory $\P$ of $\T$ is called \emph{$\alpha$-silting} if the following conditions are equivalent, where $(\T_{\ge0},\T_{\le0}):=(\T_{\ge0}^\P,\T_{\le0}^\P)$.
%	\begin{enumerate}
%	\renewcommand{\labelenumi}{(\alph{enumi})}
%		\item %$\Hom_{\T}(\alpha(P),P[i])=0$ for each $i>d$, or equivalently, 
%		$\Hom_{\T}(\P,\alpha(\P)[i])=0$ for each $i>0$.
%		\item $\alpha^{-1}(\T_{\geq0})\subset\T_{\geq0}$.
%		\item $\alpha(\T_{\leq0})\subset\T_{\leq0}$.
%		\item $\Hom_{\T}(\P,\alpha^j(\P)[i])=0$ for each $i,j\geq1$.
%	\end{enumerate}
%\end{Prop-Def}

%\comment{Up to here, we do not need that $\nu$ is the Serre functor and $\nu^{-1}=\RHom_{A^e}(A,A^e)$ is enough.}
In the rest of this section, we assume that $\T$ has a Serre functor $\nu$, and we fix a positive integer $d\in\Z$. Let $\nu_d:=\nu\circ[-d]$. 
For a silting subcategory $\P$ of $\T$ satisfying $\P=\add\P$, we have the corresponding t-structure \eqref{co-t-structure by P}.

\begin{Prop-Def}\label{define d-silting}
Let $d\geq0$ be an integer, and $\T$ a triangulated category with Serre functor $\nu$.
A silting subcategory/object $\P$ of $\T$ is called \emph{$d$-silting} if it satisfies one of the following equivalent conditions.
	\begin{enumerate}
	\renewcommand{\labelenumi}{(\alph{enumi})}
		\item %$\Hom_{\T}(\alpha(P),P[i])=0$ for each $i>d$, or equivalently, 
		$\Hom_{\T}(\nu(\P),\P[i])=0$ for each $i>d$.
		\item $\nu_d(\T_{\geq0}^\P)\subset\T_{\geq0}^\P$.
		\item $\nu_d^{-1}(\T_{\leq0}^\P)\subset\T_{\leq0}^\P$.
		\item $\Hom_{\T}(\nu_d^j(\P),\P[i])=0$ for each $i,j\geq1$.
	\end{enumerate}
We denote by $\silt^d\T$ the subset of $\silt\T$ consisting of $d$-silting subcategories.
%We define a \emph{$d$-silting object} similarly.
\end{Prop-Def}

\begin{proof}
	(a)$\Rightarrow$(b) Since $\Hom_{\T}(\P,\nu_d^{-1}(\P)[i])=0$ holds for each $i>0$, we have $\Hom_\T(\P,\nu_d^{-1}(\T_{<0}))=0$.
	Since $(\T_{\ge0},\T_{\le0})$ is a co-$t$-structure, $\nu_d(\P)\subset\T_{\ge0}$ holds. Thus $\nu_d(\T_{\ge0})\subset\T_{\ge0}$.
	
	(b)$\Rightarrow$(c) Since $(\T_{\ge0},\T_{\le0})$ and $(\nu_d^{-1}(\T_{\ge0}),\nu_d^{-1}(\T_{\le0}))$ are co-$t$-structures, we obtain $\nu_d^{-1}(\T_{\le0})\subset\T_{\le0}$. Thus (c) holds.
	
	(c)$\Rightarrow$(d) Since $\P\subset\T_{\ge0}$ and $\nu_d^{-j}(\P)\subset\nu_d^{-j}(\T_{\le0})\subset\T_{\le0}$, we obtain $\Hom_{\T}(\P,\nu_d^{-j}(\P)[i])=0$.
	
	(d)$\Rightarrow$(a) Setting $j=1$ in (d), we obtain (a).
\end{proof}

We have obvious implications 0-silting\ $\Rightarrow$\ 1-silting\ $\Rightarrow$\ 2-silting\ $\Rightarrow\cdots$. Clearly $d$-silting objects are preserved by autoequivalences of $\T$.

\begin{Rem}\label{remark d-silting}
%Let $A$ be a finite dimensional Iwanaga-Gorenstein algebra. 
\begin{enumerate}\renewcommand{\labelenumi}{(\alph{enumi})}
\item The notion of $d$-silting objects should be understood as silting objects of ``self-injective dimension at most $d$''. Indeed, for a finite dimensional Iwanaga-Gorenstein algebra $A$, we have that $A\in\per A$ is $d$-silting if and only if $\id A\le d$ holds. 
\item Let $A$ be a finite dimensional Iwanaga-Gorenstein algebra. A silting object $P$ in $\per A$ is called \emph{$n$-term} if $P\in(\proj A)*\dots*(\proj A)[n-1]$ holds. It is easy to check that each $n$-term silting object is $(\id A+n-1)$-silting. 
%Note that $d$-term silting objects are not necessarily preserved by autoequivalences of $\T$ (e.g.\ the suspension functor).
\end{enumerate}
%Do not confuse these two notions.
\end{Rem}

We have the following observation.
\begin{Prop}
	Each $d$-silting subcategory $\P$ of $\T$ gives a $d$-rigid subcategory
	\[\U_d(\P):=\add\{\nu_d^i(P)\mid i\in\Z,\ P\in\P\}.\]
\end{Prop}

\begin{proof}
	Fix $i\ge0$ and $1\le j\le d-1$. By Proposition-Definition \ref{define d-silting}(d), we have
%	\begin{equation}\label{-i}
$\Hom_\T(\P,\nu_d^{-i}(\P)[j])=0$.
%	\end{equation}
	Also we have
	\[\Hom_{\T}(\P,\nu_d^{i+1}(\P)[j])=\Hom_{\T}(\nu_d^{-i}(\P),\nu(\P)[j-d])=D\Hom_{\T}(\P[j-d],\nu_d^{-i}(\P))\stackrel{{\rm\ref{define d-silting}(d)}}{=}0.\qedhere\]
%	Thus the assertion holds.
\end{proof}

Later we need the following property.

\begin{Lem}\label{vanishing Hom}
	For a $d$-silting subcategory $\P$, we have $\Hom_{\T}((\T_{\le d}^\P)^\perp,\T_{\le0}^\P)=0$.
\end{Lem}

\begin{proof}
	This follows from $(\T_{\le d}^\P)^\perp={}^\perp\nu(\T_{\le d}^\P)=\nu({}^\perp(\T_{\le d}^\P))=\nu(\T_{>d}^\P)\subset\T_{>0}^\P$.
\end{proof}

In the rest, we will give a sufficient condition on $\T$ such that $\U_d(\P)$ is a $d$-cluster tilting subcategory of $\T$.

First, we consider the following condition on $\T$.

\begin{Def}\label{define adjacent t-structure}
We say that a silting subcategory $\P$ of $\T$ \emph{admits an adjacent $t$-structure} if there exists a t-structure $(\X,\Y)$ satisfying $\X=\T_{\le0}^\P$. In this case, we write $\T^{\ge i}_\P:=\Y[-i]$ for each $i\in\Z$.
We say that $\T$ \emph{admits adjacent $t$-structures} if each silting subcategory of $\T$ admits an adjacent $t$-structure.
%\old{	\begin{enumerate}
%	\renewcommand{\labelenumi}{(\alph{enumi})}
%		\item There exist a bounded co-t-structure $(\T_{\ge0},\T_{\le0})$ and a t-structure $(\T_{\le0},\T^{\ge0})$.
%		\item For each bounded co-t-structure $(\T_{\ge0},\T_{\le0})$, there exists a t-structure $(\T_{\le0},\T^{\ge0})$.
%	\end{enumerate}}
\end{Def}

Note that, if $\T$ admits a silting object, then each silting subcategory admits an additive generator. Moreover, in this case, there exists a silting subcategory of $\T$ which admits an adjacent t-structure if and only if $\T$ admits adjacent t-structures \cite[Theorem 4.4]{IYa1}.
%\begin{proof}
%	\cite[Theorem 4.4]{IYa1}.
%\end{proof}

Later we need the following property.

\begin{Lem}\label{splitting criterion}
	Let $\P$ be a $d$-silting subcategory admitting an adjacent t-structure $(\T_{\le0}^\P,\T^{\ge0}_\P)$.
	If $X\in\T$ satisfies $\Hom_\T(\P,X[i])=0$ for each $1\le i\le d-1$, then $X\in\T_{\le0}^\P\oplus\T^{\ge d}_\P$.
\end{Lem}
\begin{proof}
	Take a triangle
	\[X^{\le0}\to X\to Y\xrightarrow{f} X^{\le0}[1]\]
	with $X^{\le0}\in\T_{\le0}^\P$ and $Y\in\T^{\ge1}_\P$.
	Applying $\Hom_{\T}(P,-)$, we obtain $\Hom_\T(P,Y[i])=0$ for each $1\le i\le d-1$. Thus $Y\in\T^{\ge d}_\P=(\T_{\le d-1}^\P)^\perp$. By Lemma \ref{vanishing Hom}, we have $f=0$ and hence $X\simeq X^{\le0}\oplus  Y$ as desired.
\end{proof}

%Next, we consider the following condition, {which has previously been considered} in some references.
%\begin{Def}\label{define nu_d-finite}
%Let $\T$ be a triangulated category with a Serre functor $\nu$, and $\P$ a silting subcategory.
%We say that $\P$ is \emph{$\nu_d$-finite} if the following conditions hold. 
%\begin{enumerate}
%\renewcommand{\labelenumi}{(\alph{enumi})}
%%\item For each $P\in\P$, we have $\nu_d^{-i}(P)\in\T^\P_{<0}$ for $i\gg0$.
%%\item For each $X\in\T$ and $\ell\in\Z$, we have $\nu_d^{-i}(X)\in\T^\P_{\le-\ell}$ for $i\gg0$.
%\item For each $X\in\T$ we have $\nu_d^{-i}(X)\in\T^\P_{\le0}$ for $i\gg0$.
%\item For each $0\neq X\in\T$ there exists $i\in\Z$ such that $\nu_d^i(X)\not\in\T^\P_{\leq0}$.
%%\item For each $X\in\T$ and $\ell\in\Z$, we have $\nu_d^{-i}(X)\subset\T^\P_{\le\ell}$ for $i\gg0$.
%%\item \old{For each $X\in\T$, we have $\Hom_\T(\P,\nu_d^i(X)[<0])=0$ for $i\gg0$.}
%%\item \old{For each $X\in\T$, we have $\Hom_\T(\P,\nu_d^i(X)[>0])=0$ for $i\ll0$.}
%%		\item There exists a bounded co-t-structure $(\T_{\ge0},\T_{\le0})$ such that $\nu_d^i(\T_{\ge0})\subset\T_{\ge1}$ for $i\gg0$.
%%		\item There exists a bounded co-t-structure $(\T_{\ge0},\T_{\le0})$ such that $\nu_d^i(\T_{\le0})\subset\T_{\le-1}$ for $i\ll0$.
%%		\item For each bounded co-t-structure $(\T_{\ge0},\T_{\le0})$, we have $\nu_d^i(\T_{\ge0})\subset\T_{\ge1}$ for $i\gg0$.
%%		\item For each bounded co-t-structure $(\T_{\ge0},\T_{\le0})$, we have $\nu_d^i(\T_{\le0})\subset\T_{\le-1}$ for $i\ll0$.
%\end{enumerate}
%\end{Def}

Now we recall the following notion.

\begin{Def}\label{define nu_d finite}
Let $\T$ be a triangulated category with a Serre functor $\nu$. We say that $\T$ is {\it $\nu_d$-finite} if for each $X,Y\in\T$, we have $\Hom_\T(X,\nu_d^{-i}(Y)[\geq\!0])=0$ for $i\gg0$.
\end{Def}

For example, if $\T=\per A$ for a finite dimensional Iwanaga-Gorenstein algebra with $\id A\le d-1$, then $\T$ is $\nu_d$-finite.

The following is a generalization of \cite[2.5]{DI} (see also \cite[1.23]{Iy}) from tilting to silting.

\begin{Thm}[Silting-CT correspondence]\label{from tilting to silting 2}
	Let $d\geq1$ be an integer, and $\T$ a $k$-linear Hom-finite Krull-Schmidt triangulated category with a Serre functor $\nu$. Assume that $\T$ is $\nu_d$-finite and admits adjacent t-structures and a silting object. Then we have a map
	\[\silt^d\T\to\ct{d}\T\ \mbox{ given by }\ P\mapsto \U_d(P):=\add\{\nu_d^i(P)\mid i\in\Z\}.\]
%	$P$ is a $d$-silting object gives a $d$-cluster tilting subcategory
%		\[\U=\U_d(P):=\add\{\nu_d^i(P)\mid i\in\Z\}.\]
\end{Thm}

In fact, we prove a more general result for $\T$ admitting a $d$-silting subcategory. In this case, we need to replace $\nu_d$-finiteness by the following condition.

\begin{Def}
Let $\T$ be a triangulated category with a Serre functor $\nu$, and let $\P\subset\T$ be a silting subcategory. We say that $\P$ is {\it $\nu_d$-non-degenerate} if it satisfies
\[ \bigcap_{i\geq0}\nu_d^{-i}(\T^\P_{\leq0})=0, \qquad \bigcup_{i\geq0}\nu_d^i(\T^\P_{\leq0})=\T. \]
\end{Def}
Applying $\nu_d^{\pm n}$ to the above equations, we see that $\bigcap_{i\geq n}\nu_d^{-i}(\T^\P_{\leq0})=0$ and $\bigcup_{i\geq n}\nu_d^i(\T^\P_{\leq0})=\T$ for every $n$.

\begin{Thm}[Silting-CT correspondence: General version]\label{from tilting to silting}
	Let $d\geq1$ be an integer, and $\T$ a $k$-linear Hom-finite Krull-Schmidt triangulated category with a Serre functor $\nu$. Assume that $\P$ is a $\nu_d$-non-degenerate $d$-silting subcategory which admits an adjacent t-structure.
	Then $\T$ admits a $d$-cluster tilting subcategory
	\[\U=\U_d(\P):=\add\{\nu_d^i(\P)\mid i\in\Z\}.\]
\end{Thm}
\begin{proof}
	We show that, if an indecomposable object $X\in\T$ satisfies $\Hom_{\T}(\U,X[i])=0$ for all $1\le i\le d-1$, then $X\in\U$.
	Since $\P$ is $d$-silting and {$\nu_d$-non-degenerate}, by replacing $X$ by $\nu_d^i(X)$ for some $i\in\Z$, we can assume 
	\[X\notin\T_{\le0}^\P\ \mbox{ and }\ \nu_d^{-1}(X)\in\T_{\le0}^\P.\]
	By Lemma \ref{splitting criterion}, we have $X\in\T^{\ge d}_\P$.
	%\comment{Seems to be using (b) only in Prop-Def.}
	On the other hand, since $\T_{\le0}^\P=\P*\T_{<0}^\P$, there is a triangle
	\[Q\to \nu_d^{-1}(X)\xrightarrow{f} Y\to Q[1]\]
	with $Q\in\P$ and $Y\in\T_{<0}^\P$. Since
	\[\Hom_{\T}(\nu_d^{-1}(X),Y)=D\Hom_{\T}(Y,X[d])\in D\Hom_{\T}(\T_{<0}^\P,\T^{\ge0}_\P)=0,\]
	we have $f=0$ and hence $\nu_d^{-1}(X)\in\add Q\subset \P $.
	Thus $X\in\U$.
\end{proof}

%\comment{Converse? If $\U_d(P)$ is $d$-CT, then is $P$ $d$-silting?}

We give the following sufficient condition for $\nu_d$-non-degeneracy.

\begin{Lem}\label{suff}
	Suppose that $\P$ has an additive generator $P$ and consider the following conditions.
	\begin{enumerate}
		\renewcommand{\labelenumi}{(\alph{enumi})}
		\renewcommand{\theenumi}{(\alph{enumi})}
		\item\label{Hom} $\T$ is $\nu_d$-finite.%For each $X,Y\in\T$, we have $\Hom_\T(Y,\nu_d^{-i}(X)[\geq\!0])=0$ for $i\gg0$.
		\item\label{left} $\nu_d^{-i}(\T^\P_{\le0})\subset\T^\P_{<0}$ holds for $i\gg0$.
		\item\label{nondeg} $\P$ is $\nu_d$-non-degenerate.
	\end{enumerate}
	Then the implications {\rm \ref{Hom}$\Leftrightarrow$\ref{left}$\Rightarrow$\ref{nondeg}} hold. If $\P$ is $d$-silting, then {\rm \ref{nondeg}$\Rightarrow$\ref{left}} also holds.
	%Therefore it does not on a choice of the silting subcategory $\P$. In this case, we call $\T$ $\nu_d$-finite.
\end{Lem}
\begin{proof}
	%Consider the following intermediate conditions.
	%\begin{itemize}
	%\item[(a$'$)] For each $X\in\T$, we have $\nu_d^{-i}(X)\in\T^\P_{\le0}$ for $i\gg0$.
	%\item[(b$'$)] For each $P\in\P$ and $\ell\in\Z$, we have $\nu_d^{-i}(P)\in\T^\P_{\le-\ell}$ for $i\gg0$.
	%\end{itemize}
	%
	%(a$'$)$\Rightarrow$(a) and (b)$\Rightarrow$(b$'$) is clear. (b)$\Rightarrow$(a) follows by setting $X:=P[-1]$. \comment{Complete!}
	%Let $P$ be an additive generator of $\P$.	
	\ref{Hom}$\Leftrightarrow$\ref{left} Since $\P=\add P$, the condition \ref{left} is equivalent to $\Hom_\T(P,\nu_d^{-i}(P)[\geq\!0])=0$ for $i\gg0$. Then we get \ref{Hom}$\Rightarrow$\ref{left} by setting $Y=X=P$.
	Let us prove the converse. Since $\P\subset\T$ is silting, we have $X,Y\in\P[m]\ast\cdots\ast\P[n]$ for some $m\leq n$. By \ref{left}, we have $\nu_d^{-i}(X)\in\nu_d^{-i}(\T_{\leq -m}^\P)\subset\T^\P_{\leq-n-1}=\bigcup_{\ell>0}\P[n+1]\ast\cdots\ast\P[n+\ell]$ for $i\gg0$. It follows that we have $\Hom_\T(Y,\nu_d^{-i}(X)[\geq\!0])=0$ for such $i$.
	
	\ref{left}$\Rightarrow$\ref{nondeg} By \ref{left}, we have $\bigcap_{j\geq0}\nu_d^{-ij}(\T^\P_{\leq0})\subset\bigcap_{j\geq0}\T^\P_{\leq-j}=0$ and $\bigcup_{j\geq0}\nu_d^{-ij}(\T^\P_{\geq0})\supset\bigcup_{j\geq0}\T^\P_{\geq-j}=\T$.	
%Also by \ref{Hom}, for any $X\in\T$ we have $\Hom_\T(P[\leq\!0],\nu_d^{-i}(X))=0$ for $i\gg0$, which means $\nu_d^{-i}(X)\in\T^\P_{<0}$, thus $X\in\bigcup_{i\geq0}\nu_d^i(\T^\P_{\leq0})$.
	
	Finally we prove the last statement.
	%\ref{nondeg}$\Rightarrow$\ref{left} using that $\P$ is $d$-silting. 
	If $\P$ is $\nu_d$-non-degenerate, then $P[1]\in\nu_d^i(\T^\P_{\leq0})$ for some $i\geq0$, hence $\nu_d^{-i}(P)\in\T^\P_{<0}$. Since $P$ is $d$-silting, the functor $\nu_d^{-1}$ preserves $\T^\P_{<0}$, so $\nu_d^{-i}(P)\in\T^\P_{<0}$ holds for any larger $i$.	
	%	Then we have (a)$\Rightarrow$(c) by setting $X=P$, and also (c)$\Rightarrow$(a) since $X\in\thick P$.
	%	We show (c)$\Rightarrow$(b). Suppose to the contrary that $\nu_d^i(X)\in\T^\P_{\leq0}$ for all $i\in\Z$. Then we have $X\in\bigcap_{i\geq0}\nu_d^{-i}(\T_{\leq0}^\P)$, which is contained in $\bigcap_{N\geq0}\T_{\leq-N}^\P$ by (c). This means that we have $\Hom_\T(P,X[>\!-N])=0$ for all $N\geq0$, and hence $X=0$.
	%	We have therefore proved (a)+(b)$\Leftrightarrow$(c).
\end{proof}

Now we are ready to prove Theorem \ref{from tilting to silting 2}.

\begin{proof}[Proof of Theorem \ref{from tilting to silting 2}]
Immediate from Theorem \ref{from tilting to silting} and Lemma \ref{suff}{\rm \ref{Hom}$\Rightarrow$\ref{nondeg}}.
\end{proof}

%A typiclFor comparison, let us recall the following notion which has been considered in some references.
%A typical example of a silting subcategory $\P$ satisfying the conditions in \ref{from tilting to silting} is given by the following which has been considered in some references.
%\begin{Def}
%Let $\La$ be a smooth, proper, connective dg algebra. We say that $\La$ is {\it $\nu_d$-finite} if it satifies the following conditions.
%\begin{itemize}
%\item $\La\in\per\La$ is $d$-silting, that is, $\nu_d^{-1}(\La)\subset\T^\La_{\leq0}$.
%\item $\nu_d^{-i}(\La)\in\T_{<0}^\La$ for $i\gg0$.
%\end{itemize}
%\end{Def}

	Let $A$ be a smooth proper connective dg algebra and $\T=\per A$. Then $\T$ is a $k$-linear, $\Hom$-finite, Krull-Schmidt triangulated category, and the silting object $A\in\T$ admits an adjacent $t$-structure. Indeed, the standard $t$-structure $(\T^{\leq0},\T^{\geq0})$ gives an adjacent $t$-structure.
	
%	\begin{Def}\label{nu_d-finite}
%	Under the above setting, we say that $\La$ is {\it $\nu_d$-finite} if it satifies the following conditions.
%	\begin{itemize}
%		\item $\La\in\per\La$ is $d$-silting, that is, $\nu_d^{-1}(\La)\subset\T^\La_{\leq0}$.
%		\item $\nu_d^{-i}(\La)\in\T_{<0}^\La$ for $i\gg0$.
%	\end{itemize}
%	\end{Def}
%Then by \ref{suff}, the the silting subcategory $\add\La\subset\per\La$ is $\nu_d$-non-degenerate. 
	Theorem \ref{from tilting to silting 2} gives the following result, where $\silt^dA$ is the subset of $\silt A$ consisting of $d$-silting objects.
	%In this case, $\La\in\per\La$ is $d$-silting if and only if $\sup\RHom_{\La^e}(\La,\La^e)\leq d$, where $\sup X=\sup\{i\in\Z\mid H^iX\neq0\}$ for a complex $X$. Also, $\T$ is $\nu_d$-finite if and only if $\sup\nu_d^{-i}(\La)\to-\infty$ as $i\to\infty$.
	%	
	%	For example, let $\Lambda$ be a finite-dimensional algebra and $\T=\K^{\bb}(\proj\Lambda)$. For the tilting object $P=\Lambda$, we have
	%	\[\T_{\ge0}=\{\cdots\to0\to P^0\to P^1\to\cdots\}\ \mbox{ and }\ \T_{\le0}=\{\cdots\to P^{-1}\to P^0\to 0\to\cdots\}.\]
	%	Assume $\gl\Lambda<\infty$. Then $(\T_{\le0},\T_{\le0}{}^\perp)$ is the standard $t$-structure $(\D^{\le0}(\mod\Lambda),\D^{\ge0}(\mod\Lambda))$ and hence $\T$ has adjacent $t$-structures. 
	%	%\marginpar{What conditions?\;/ED}
	%	One can check easily that $P$ is $d$-silting if and only if $\gl\Lambda\le d$.
	%	The condition that $\T$ is $\nu_{d}$-finite is nothing but $\nu_d$-finiteness of $\Lambda$ introduced in \cite{DI}.

%As a special case, we obtain the following result.

\begin{Cor}\label{from silting to cluster tilting}
	Let $d\geq1$ be an integer, and $A$ a smooth proper dg algebra such that $\per A$ is $\nu_d$-finite.
\begin{enumerate}
\item {\rm(Silting-CT correspondence: Dg version)} Each $d$-silting object $P$ of $\per A$ gives a $d$-cluster tilting subcategory of $\per A$:
	\[\U_d(P):=\add\{\nu_d^i(P)\mid i\in\Z\}.\]
	Thus we have a map $\silt^dA\to\ct{d}(\per A)$, $P\mapsto\U_d(P)$.
\item Each $d$-cluster tilting subcategory of $\per A$ is locally bounded and has a $\nu_d$-additive generator.
\end{enumerate}
\end{Cor}

\begin{proof}
%Since $\gl\Lambda$ is finite, we have $\T:=\K^{\bb}(\proj\Lambda)=\D^{\bb}(\mod\Lambda)$. Thus 
(1) By replacing $A$ by the dg endomorphism algebra of $P$, we can assume $P=A$.
Then $A$ admits an adjacent t-structure given by the standard t-structure since $A$ is smooth.
%Thus $\T:=\perA$ 
%, thanks to silting-SMC correspondence \cite[Theorem 6.1]{SY} shows that $(\T_{\le0},\T_{\le0}{}^\perp)$ gives a t-structure of $\T$.
%Since $A$ is $\nu_d$-finite, $\nu_d^i(\T_{\ge0})\subset\T_{\ge1}$ holds for $i\gg0$.
	Thus all assumptions in Theorem \ref{from tilting to silting 2} are satisfied, and the conclusion follows.

(2) By (1), $\per A$ has a $d$-cluster tilting subcategory $\U_d(A)$ with a $\nu_d$-additive generator $A$. By Theorem \ref{CT with generator}, each $\U\in\ct{d}$ has an $\nu_d$-additive generator. Since $A$ is $\nu_d$-finite, $\U$ is locally bounded by Lemma \ref{nu_d-finite}(2)$\Rightarrow$(3).
\end{proof}

\begin{Ex}
Let $A$ be a smooth proper dg algebra, and $P$ a tilting object in $\per A$.
By Remark \ref{remark d-silting}(a), $P$ is $d$-silting if and only if $\gldim\End_{\D(A)}(P)\leq d$. %Also the condition that $\T$ is $\nu_{d}$-finite is nothing but $\nu_d$-finiteness of $\Lambda$ introduced in \cite{DI}. In particular, this is the case, for example, when $\gldim\La\leq d-1$
Moreover, if $\gldim\End_{\D(A)}(P)\leq d-1$, then $\per A$ is $\nu_d$-finite.
Therefore we obtain a $d$-cluster tilting subcategory $\U_d(P)$ of $\per A$.
\end{Ex}

It is natural to pose the following.

\begin{Cj}
The map $\silt^dA\to\ct{d}(\per A)$ given in Corollary \ref{from silting to cluster tilting} is surjective.
\end{Cj}

Notice that, even if $A$ is an algebra, the restriction of the map $\silt^dA\to\ct{d}(\per A)$ to the set of tilting objects in $\per A$ is not surjective in general, see \cite[Example 6.9]{DI} due to \cite{AOce}.

\subsection{CT-CT correspondence in derived and cluster categories}

We start with the following general observation.

\begin{Prop}\label{AO}
Let $d\geq1$ be an integer, $\C$ and $\D$ be triangulated categories with Serre functors $\nu$, and $\pi:\D\to\C$ a triangle functor which commutes with $\nu$ and preserves $d$-rigidity. Assume that $\U\in\dctilt\D$ satisfies $\pi(\U)\in\dctilt\C$.
%Assume that a group $G$ acts on $\D$ such that $\pi\circ g=\pi$ holds for all $g\in G$ and the induced functor $\D/G\to\C$ is fully faithful.
\begin{enumerate}
\item If $\pi(\U)$ has an additive generator, then for each $\V\in\dctilt\D$, $\pi(\V)\in\dctilt\C$ holds and $\pi(\V)$ has an additive generator.
\item If $\pi(\U)$ is locally bounded and has a $\nu_d$-additive generator, then for each $\V\in\dctilt\D$, $\pi(\V)\in\dctilt\C$ holds and $\pi(\V)$ is locally bounded and has a $\nu_d$-additive generator.
\end{enumerate}
\end{Prop}

\begin{proof}
%We include the proof for the convenience of the reader. 
%We write $\D=\D^b(A)$ and $\C=\C_d(A)$.
	Since $\V\subset\D$ is closed under $\nu_d^{\pm1}$ and $\V$ is $d$-rigid, the same holds for $\pi(\V)\subset\C$.
Since $\V\subset\D$ is $d$-cluster tilting, we have $\U\subset\D=\V\ast\V[1]\ast\cdots\ast\V[d-1]$, thus $\pi(\U)\subset\pi(\V)\ast\cdots\ast\pi(\V)[d-1]$ in $\C$. 
    Applying Proposition \ref{ct} for $(\C,\U,\V):=(\C,\pi(\U),\pi(V))$, we obtain the assertion. 
\end{proof}
%It remains to prove the maximality.

%On the other hand, we have by assumption that $\pi(\U)\in\C$ is $d$-cluster tilting, so $\C=\pi(\U)[-d+1]\ast\cdots\ast\pi(\U)[-1]\ast\pi(\U)$. It follows that
%\[\begin{aligned}\C=\left(\pi(\U)[-d+1]\ast\cdots\ast\pi(\U)\right)\ast\left(\pi(\U)[-d+2]\ast\cdots\ast\pi(\U)[1]\right)\ast\cdots\ast\left(\pi(\U)\ast\cdots\ast\pi(\U)[d-1]\right).\end{aligned}\]
%Since $\pi(\V)$ is $d$-rigid, the assertion follows from Proposition \ref{ct}(3).
%repeatedly, we have
%\begin{equation}\label{eq2bai}\C=\pi(\U)[-d+1]\ast\cdots\ast\pi(\U)\ast\cdots\pi(\U)[d-1].\end{equation}
%Now let $X\in\C$ be such that $\Hom_\C(X,\pi(\U)[i])=0$ for $0<i<d$. Note that this also implies $\Hom_\C(\pi(\U)[-i],X)=0$ for $0<i<d$ by Serre duality. Therefore (\ref{eq2bai}) shows $X\in\pi(\U)$, so $\pi(\U)\subset\C$ is $d$-cluster tilting.

%\comment{Generalize to dg setting} Let $A$ be a finite dimensional algebra which is $\nu_d$-finite, and let $\pi\colon\D^b(A)\to\C_d(A)$ be the canonical projection functor.

Consider the functor
\[ \xymatrix{ \per A\ar[r]&\C_d(A) } \]
from the diagram \eqref{square}, and we discuss the correspondence between cluster tilting subcategories in $\per A$ and in $\C_d(A)$. The following result was given in \cite[Proposition 3.2]{AOce} for the case $d=2$, and the same argument applies to the general case.
Applying Proposition \ref{AO}, we immediately obtain the following consequence.

\begin{Thm}\label{ao}
Let $d\geq1$ be an integer, and $A$ a smooth proper connective dg algebra which is $\nu_d$-finite. 
\begin{enumerate}
\item Any $d$-cluster tilting subcategory of $\C_d(A)$ has an additive generator.
\item Let $\U$ be an arbitrary $d$-cluster tilting subcategory of $\per A$.
\begin{enumerate}
\item $\U$ has a $\nu_d$-additive generator $U$.
\item The image $\pi(U)\in\C_d(A)$ is a $d$-cluster tilting object.
\end{enumerate}
\end{enumerate}
\end{Thm}
\begin{proof}
(1)  Since $\C_d(A)$ has a $d$-cluster tilting object $A$, this is a direct consequence of Theorem \ref{CT with generator}.

(2)  (a)  This is nothing but Corollary \ref{from silting to cluster tilting}.

(b)  We apply Proposition \ref{AO}(1) to $\D=\per A$, $\C=\C_d(A)$, and $\V=\add\{\nu_d^i(U)\mid i\in\Z\}$. The projection functor $\per A\to\C_d(A)$ certainly commutes with the Serre functors and preserves $d$-rigidity. Also, by Corollary \ref{from silting to cluster tilting}, there exists a $d$-cluster tilting subcategory $\add\{\nu_d^{-i}(A)\mid i\in\Z\}\subset\per A$ whose image in $\C_d(A)$ is the $d$-cluster tilting object $A\in\C_d(A)$. Therefore Proposition \ref{AO}(1) to see that $\pi(\U)=\pi(U)$ is a $d$-cluster tilting object of $\C_d(A)$.
\end{proof}

\subsection{Silting-silting and silting-CT correspondences via Calabi-Yau completions}
%\section{Calabi-Yau completions}
%\subsection{Preliminaries}

Let $A$ be a smooth dg $k$-algebra, and $d\ge0$ an integer.
%Assume that $\gl A\le d$ for some positive integer $d$.
%Let $B=A\op DA[-d-1]$ be the trivial extension dg $k$-algebra. 
%\old{The \emph{$d$-cluster category} $\C_{d}(A)$ is  defined as the triangulated hull of the orbit category $\D^{\rm b}(\mod A)/\nu [-d]$.}
%Let $\Theta$ be a $\mathcal H$-projective resolution of the dg $A$-bimodule $\RHom_{A^e}(A,A^e)$ and let $\theta:=\Theta[d]$.
Let $\Theta=A^\vee[d]$ the $d$-shifted inverse dualizing bimodule of $A$, and $\Pi=\Pi_{d+1}(A)=T_A(\Theta)$ the $(d+1)$-Calabi-Yau completion of $A$, see Section \ref{subsection CY}.
%$\Pi:$ as the tensor algebra $T_{A}(\Theta)$. Note that $\Pi$ is smooth and $(d+1)$-Calabi-Yau as a bimodule by \cite[Theorem 4.8]{Ke11}, see also \cite{Ke11+}.
The natural inclusion $A\subset\Pi$ gives a triangle functor $-\lotimes_{A}\Pi:\per A\to \per\Pi$. 
We will give a characterization of silting objects of $A$ which are sent to silting objects of $\Pi$ via the functor $-\lotimes_A\Pi$. For this, the following observation is useful.

\begin{Prop}[{\cite[4.2]{Ke11}, see also Proposition \ref{from d-silting to ctilt 0}(2)}]\label{P and CY completion}
Let $A$ be a smooth dg $k$-algebra, $d\ge0$ an integer, and $\Pi=\Pi_{d+1}(A)$. For an object $P\in\per A$ satisfying $\thick P=\per A$, $\REnd_\Pi(P\lotimes_A\Pi)$ is quasi-equivalent to $\Pi_{d+1}(\REnd_A(P))$.
\end{Prop}

If $A$ is a smooth proper dg algebra, then the triangulated category $\per A$ has a Serre functor $-\otimes_A\Theta^{-1}[d]$. We will not assume that $A$ is proper, in particular $\per A$ does not have a Serre functor. 
%In the previous section, we have discussed the notions of {\it $d$-silting objects} (\ref{define d-silting}) and {\it $\nu_d$-finitenss} (\ref{nu_d-finite}) of silting subcategories in a triangulated category, under the existence of a Serre functor $\nu$.
In this setting, we generalize the notions of {$d$-silting objects} (Definition \ref{define d-silting}) and {$\nu_d$-finitenss} (Definition \ref{define nu_d finite}) of silting objects in a triangulated category to $\per A$ by replacing the inverse Serre functor $\nu^{-1}$ by the inverse dualizing bimodule $\RHom_{A^e}(A,A^e)$.

\begin{Def}\label{non-proper}
Let $A$ be a smooth dg algebra. 
\begin{enumerate}
\item A silting object $P$ in $\per A$ is
%let $A'$ be the dg endomorphism algebra of $P$. 
called {\it $d$-silting} if $\Hom_{\D(A)}(P,P\otimes_A\Theta[i])=0$ for each $i>0$. This is equivalent to that the inverse dualizing bimodule of $\REnd_A(P)$ is concentrated in degrees $\le d$, and also to that $(d+1)$-Calabi-Yau completion of $\REnd_A(P)$ is connective. We denote by $\silt^dA$ the subset of $\silt A$ consisting of $d$-silting objects.
%$\Pi_{d+1}(A')$ is connective. 
%$P$ is $d$-silting and $H^0\Pi_{d+1}(B)$ is finite dimensional. 
\item We call $A$ {\it $\nu_d$-finite} if for each $X,Y\in\per A$, we have $\Hom_{\D(A)}(X,Y\lotimes_A\Theta^i[\geq\!0])=0$ for $i\gg0$.
%This is equivalent to that $H^{\ge0}(\Pi_{d+1}(A))$ is finite dimensional.
\end{enumerate}
\end{Def}
%We call $P$ \emph{$d$-silting} if $\Pi(B)$ is connective, and \emph{$\nu_d$-finite} if $H^0(\Pi(B))$ is finite dimensional.
We have obvious inclusions $\silt^0A\subset\silt^1A\subset\silt^2A\subset\cdots$.

%dg algebra $\bigoplus_{i\ge0}\RHom_A(P,P\otimes_A\theta^i)$ is connective.

\begin{Rem}
Let $A$ be a smooth dg algebra.
%\begin{enumerate}
%\item 
%Notice that $H^{\ge0}(\Pi_{d+1}(A))$ is finite dimensional if and only if for each $X,Y\in\per A$, we have $\Hom_{\D(A)}(X,(Y\otimes_A\Theta^i)[\geq\!0])=0$ for $i\gg0$.
Assume that $A$ is proper so that $\per A$ has a Serre functor $\nu_d=-\otimes_A\Theta^{-1}$. Then  these notions coincide with the ones given in Definitions \ref{define d-silting} and \ref{define nu_d finite}.
%\item Suppose that a silting object $P\in\per A$ is $\nu_d$-finite. Then any $d$-silting object is $\nu_d$-finite.
%\end{enumerate}
\end{Rem}
%\begin{proof}
%	to be written
%\end{proof}

%\begin{Lem}\label{silting and Pi}
%Under the setting above, the following assertions hold.
%\begin{enumerate}
%\item $A$ is a $d$-silting object in $\per A$ if and only if $\Pi$ is connective.
%\item Assume that the conditions in (1) hold. Then $A$ is $\nu_d$-finite if and only if $H^0(\Pi)$ is finite dimensional.
%\end{enumerate}
%\end{Lem}

%We say $A$ is {\it $\nu_d$-finite} if $\Pi$ is connective and $H^0\Pi$ is finite dimensional.
We are now able to state the main result of this subsection.
We call a map $f:S\to S'$ between posets an \emph{embedding of posets} if for $s,t\in S$, $s\ge t$ holds if and only if $f(s)\ge f(t)$ holds.

\begin{Thm}\label{from d-silting to ctilt}
Let $A$ be a smooth dg $k$-algebra, $d\ge0$ an integer, and $\Pi$ the $(d+1)$-Calabi-Yau completion of $A$.
Let $P$ be an object in $\per A$.
\begin{enumerate}
%	\item $\thick_{\D(A)}P=\per A$ if and only if $\thick_{\D(\Pi)}(P\lotimes_A\Pi)=\per\Pi$.
\item{\rm (Silting-silting correspondence)} $P$ is a $d$-silting object in $\per A$ if and only if $P\lotimes_A\Pi$ is a silting object in $\per \Pi$. Moreover, we have an embedding $-\lotimes_A\Pi:\silt^dA\to\silt\Pi$ of posets.
%If $A$ is proper, then these conditions are equivalent to that $P$ is a $d$-silting object in $\per A$.
\item{\rm (Silting-CT correspondence)} Assume that $A$ is $\nu_d$-finite. If $P$ satisfies the conditions in (1), then $P\lotimes_A\Pi$ is a $d$-cluster tilting object in $\C(\Pi)$. Thus we have a map $\silt^dA\to\ct{d}\C(\Pi)$.
\end{enumerate}
\end{Thm}

Theorem \ref{from d-silting to ctilt}(1) is in fact a special case of Proposition \ref{from d-silting to ctilt 0} below.

Let $A$ be a dg $k$-algebra. For a dg $A^e$-module $\theta$ and $j\ge0$, let $\theta^j:=\theta\lotimes_A\cdots\lotimes_A\theta$ and 
%For simplicity, a $(-\lotimes_A\theta)$-silting object in $\per A$ is called \emph{$\theta$-silting}.
$B=T_A^{\rm L}(\theta)=\bigoplus_{j\ge0}\theta^j$ the tensor dg $k$-algebra.
%Assume that $\gl A\le d$ for some positive integer $d$.
%Let $B=A\op DA[-d-1]$ be the trivial extension dg $k$-algebra. 
%Let $\Theta$ be a $\mathcal H$-projective resolution of the dg $A$-bimodule $\RHom_{A^e}(A,A^e)$ and let $\theta:=\Theta[d]$. We define the \emph{$(d+1)$-derived preprojective algebra} as the tensor algebra $\Pi:=\Pi_{d+1}(A)$ as the tensor algebra $T_{A}(\theta)$.  Note that $\Pi$ is homologically smooth and $(d+1)$-Calabi-Yau as a bimodule by \cite[Theorem 4.8]{Ke11}.
The natural inclusion $A\subset B$ gives a triangle functor $-\lotimes_{A}B: \per A\to \per B$. The following description is \cite[Proposition 4.2]{Ke11} for the case $\theta$ is the inverse dualizing bimodule.
\begin{Lem}\label{theta theta'}
Let $P\in\per A$ be an object satisfying $\per A=\thick P$, $A':=\REnd_A(P)$ and $\theta':=\RHom_A(P,P\lotimes_A\theta)$ an $A'{}^e$-module. Then for each $X\in\D(A)$ and $j\ge0$, we have a natural isomorphism in $\D(A)$:
\begin{align*}
\RHom_A(P,X\lotimes_{A}\theta^j)&\simeq\RHom_A(P,X)\lotimes_{A'}\theta'{}^j,\\
\RHom_{B}(P\lotimes_AB,X\lotimes_AB)&\simeq \RHom_A(P,X)\lotimes_{A'}T^\rL_{A'}(\theta').
\end{align*}
In particular, $\REnd_B(P\lotimes_AB)$ is quasi-equivalent to $T^\rL_{A'}(\theta')$.
\end{Lem}

\begin{proof}
Let $P^{-1}:=\RHom_A(P,A)$ be an $(A,A')$-bimodule. Then $\theta\simeq P\lotimes_A\theta\lotimes_AP^{-1}$. Thus, for $j\ge0$, we have
$P^{-1}\lotimes_{A'}\theta'{}^j\simeq P^{-1}\lotimes_{A'}(P\lotimes_A\theta\lotimes_AP^{-1})^j\simeq \theta^j\lotimes_AP^{-1}$ and hence
\begin{align*}
&\RHom_A(P,X\lotimes_{A}\theta^j)\simeq X\lotimes_{A}\theta^j\lotimes_AP^{-1}\simeq X\lotimes_AP^{-1}\lotimes_{A'}\theta'{}^j\simeq \RHom_A(P,X)\lotimes_{A'} \theta'{}^j\ \mbox{ and }\\
%\end{align*}
%	\begin{align*}%\label{T_A'(theta')} 
	&\RHom_{B}(P\lotimes_AB,X\lotimes_AB)=\RHom_{A}(P, X\lotimes_AB)=\bigoplus_{j\ge0}\RHom_{A}(P, X\lotimes_A\theta^j)\\
	\simeq&\bigoplus_{j\ge0}\RHom_A(P,X)\lotimes_{A'}\theta'{}^j
	\simeq\RHom_A(P,X)\lotimes_{A'} T_{A'}(\theta').\qedhere
	\end{align*}
%Thus the assertions follow.
\end{proof}

\begin{Prop}\label{from d-silting to ctilt 0}
	Let $P$ be an object in $\per A$.
	\begin{enumerate}
		\item $\thick_{\D(B)}(P\lotimes_AB)=\per B$ if and only if $\thick_{\D(A)}P=\per A$.
%		\item If $\thick P=\per A$, then $\REnd_B(P\lotimes_AB)$ is quasi-equivalent to $T^\rL_{\REnd_A(P)}(\RHom_A(P,P\lotimes_A\theta))$.
		\item $P\lotimes_AB\in\silt B$ holds if and only if $P\in\silt A$ and $\Hom_{\D(A)}(P,P\otimes_A\theta[i])=0$ for each $i>0$. 
		\item Assume that both $P,Q\in\silt A$ satisfy the conditions in (2). Then $P\lotimes_AB\ge Q\lotimes_AB$ in $\silt B$ if and only if $P\ge Q$ in $\silt A$.
	\end{enumerate}
\end{Prop}
\begin{proof}
	(1)  If $\per A=\thick P$, then applying $-\lotimes_AB$, we obtain $\per B=\thick(P\lotimes_AB)$. If $\thick(P\lotimes_AB)=\per B$, then applying $-\lotimes_B A$, we obtain $\thick P=\thick((P\lotimes_AB)\lotimes_B A)=\per A$.\\
	%\old{Suppose conversely that $\thick(P\lotimes_A\Pi)=\per\Pi$. We prove $\thick P=\per A$ by showing $\RHom_{A}(P,X)=0$ implies $X=0$. We use the following notation: given a morphism $\L\to\G$ of dg algebras and a $\G$-module $M$, we denote by $M_\L$ the restricted $\L$-module. Let $X\in\D(A)$ and suppose that $\RHom_{A}(P,X)=0$. Consider the natural morphisms $A\to \Pi\to A$ whose composite is the identity. Then $\RHom_\Pi(P\lotimes_A\Pi,X_\Pi)=\RHom_A(P,(X_\Pi)_A)=\RHom_A(P,X)$, which is $0$ by assumption. Since $\thick(P\lotimes_A\Pi)=\per\Pi$, this shows $X_\Pi=0$, hence $X=0$.}\\
%	(2) Let $A':=\REnd_A(P)$. Then $P$ is an $(A',A)$-bimodule and $P^{-1}:=\RHom_A(P,A)$ is an $(A,A')$-bimodule satisfying $P\lotimes_AP^{-1}\simeq A$ as $A^e$-modules and $P\lotimes_AP^{-1}\simeq A'$ as $A'{}^e$-modules. Let $\theta':=\RHom_A(P,P\lotimes_A\theta)\simeq P\lotimes_A\theta\lotimes_AP^{-1}$ be an $A'{}^e$-module. For $j\ge0$, let $\theta^j:=\theta\lotimes_A\cdots\lotimes_A\theta$ and $\theta'{}^j:=\theta'\lotimes_{A'}\cdots\lotimes_{A'}\theta'$ ($j$ times). Then we have
%	\[\theta'{}^j\simeq (P\lotimes_A\theta\lotimes_AP^{-1})^j\simeq P\lotimes_A\theta^j\lotimes_AP^{-1}\simeq\RHom_A(P,P\lotimes_A\theta^j),\]
%	and hence
%	\begin{equation}\label{T_A'(theta')} \REnd_{B}(P\lotimes_AB)=\RHom_{A}(P, P\lotimes_AB)=\bigoplus_{j\ge0}\RHom_{A}(P, P\lotimes_A\theta^j)\simeq\bigoplus_{j\ge0}\theta'{}^j\simeq T_{A'}(\theta').\end{equation}
(2) By Lemma \ref{theta theta'}, $\REnd_B(P\lotimes_AB)$ is quasi-equivalent to $T^\rL_{A'}(\theta')$. Thus $\REnd_{B}(P\lotimes_AB)$ is connective if and only if $A'$ is connective and $H^i(\theta')=0$ for all $i>0$. Thus the assertion follows from (1).\\
(3) $A'$ and $\REnd_B(P\lotimes_AB)\simeq T^\rL_{A'}(\theta')$ are connective by our assumption. By Lemma \ref{theta theta'}, $\RHom_B(P\lotimes_AB,Q\lotimes_AB)\simeq\RHom_A(P,Q)\lotimes_{A'}T^\rL_{A'}(\theta')$. Thus $H^i(\RHom_{B}(P\lotimes_AB,Q\lotimes_AB))=0$ holds for each $i>0$ if and only if $H^i(\RHom_A(P,Q))=0$ holds for each $i>0$. Thus the assertion follows.
\end{proof}

Now we summarize the proof of Theorem \ref{from d-silting to ctilt}.
\begin{proof}
(1) The assertion is immediate from Proposition \ref{from d-silting to ctilt 0}(2)(3). 

(2) By our assumption, $H^0\Pi$ is finite dimensional and hence we have Amiot's map $\silt\Pi\to\dctilt\C(\Pi)$ in \eqref{silt to dctilt}. Thus the assertion follows from (1).
%By Lemma \ref{characterize d-silting}, other direct summands are zero if and only if $P$ is $d$-silting.
\end{proof}

%We call a dg algebra $A$ \emph{connective} if $H^i(A)=0$ holds for all $i>0$. This is equivalent to that $A\in\per A$ is silting.
%Throughout this section, let $d$ be a positive integer, and let $\Pi$ be a bimodule $(d+1)$-CY dg algebra over a field $k$ such that $\Pi$ is connective and $H^0\Pi$ is finite dimensional. Amiot \cite{Am09} and Guo \cite{Guo} introduced the {\it cluster category} as 
%\[ \C(\Pi):=\per\Pi/\D^b(\Pi). \]

%By \cite{Ke05} and \cite{Guo}, there is a triangle equivalence $\C_{d}(A)\simeq\C(\Pi)$ making the following diagram commutative, where the horizontal functors are the canonical functors.
%\begin{equation}
%	\xymatrix@C5em{\D^{\bb}(A) \ar[r]\ar[d]_{-\lotimes_{A}\Pi}&\C_{d}(A)\ar[d]^\wr\\
%		\per\Pi \ar[r] & \C(\Pi)}\label{clustercommu}
%\end{equation}
%Recall that we say $P\in \T$ is a $d$-silting, if $\nu_{d}(\T_{\ge 0})\subset \T_{\ge 0}$.
%\begin{Summary}\comment{Delete?}
%The results in this section can be summarized as the well-definedness of maps in the following diagram.%In particular, we obtain well-defined maps
%\begin{equation}\label{silt-ctilt}
%	\xymatrix@C5em{\dsilt A\ar[d]_{{\rm Thm.\ \ref{from d-silting to ctilt}}}\ar[r]^(.45){{\rm Thm.\  \ref{from tilting to silting}}
%		}&\dctilt\D^{\rm b}(A)\ar[d]^-{\rm Thm.\ \ref{ao}}\\
%		\silt\Pi\ar[r]^-{\eqref{silt to dctilt}}&\dctilt\C(\Pi)\simeq\dctilt\C_d(A).}
%\end{equation}
%We will study the lower horizontal map in the following section.
%\end{Summary}

\begin{Rem}
	Notice that Theorem \ref{from tilting to silting} for algebraic triangulated categories $\T$ follows from Theorem \ref{ao} and Theorem \ref{from d-silting to ctilt}. In fact, if $P$ is $d$-silting, then $P\lotimes_A\Pi\in\silt\Pi$, and hence $P\in\dctilt\C_d(A)$. By Theorem \ref{ao}, we get $\U_d(P)\in\dctilt\D^{b}(A)$.
    {We refer to \cite[Theorem 5.2]{Ha7} for a variation of Theorem \ref{from d-silting to ctilt} based on this observation.}
\end{Rem}

Let us note the following compatibility of mutations {which preserve $d$-silting objects} and the Calabi-Yau completions.
\begin{Prop}\label{mutation compatible}
Let $\Pi=\Pi_{d+1}(A)$ be the $(d+1)$-Calabi-Yau completion of a smooth dg algebra $A$. Let $P\in\silt^dA$ and $Q\in\add P$.
\begin{enumerate}
\item If $\mu^+_Q(P)$ exists and is $d$-silting, then $\mu^+_{Q\lotimes_A\Pi}(P\lotimes_A\Pi)$ exists and is isomorphic to $\mu^+_Q(P)\lotimes_A\Pi$.
\item If $\mu^-_Q(P)$ exists and is $d$-silting, then $\mu^-_{Q\lotimes_A\Pi}(P\lotimes_A\Pi)$ exists and is isomorphic to $\mu^-_Q(P)\lotimes_A\Pi$.
\end{enumerate}
%Let $P\in\per A$ be a $d$-silting object, and let $\Pi=\Pi_{d+1}(A)$ be the $(d+1)$-Calabi-Yau completion. Suppose that the right mutation $\mu^+_Q(P)$ (resp. the left mutation $\mu^-_Q(P)$) at a direct summand $Q$ is still a $d$-silting object. 
\end{Prop}
\begin{proof}
	We only prove (1). %We let $U$ be the direct summand of $P$ such that $\add P=\add(Q\oplus U)$ and $\add Q\cap\add U=0$. 
	By assumption, there exists a right $(\add Q)$-approximation $Q_0\xrightarrow{f} P$	 and an exchange triangle $P'\to Q_0\xrightarrow{f} P$ in $\per A$. It is enough to show that $Q_0\lotimes_A\Pi\xrightarrow{f\otimes1} P\lotimes_A\Pi$ is a right $(\add Q\lotimes_A\Pi)$-approximation. Applying $\Hom_{\D(\Pi)}(Q\lotimes_A\Pi,-)$ to the triangle $P'\lotimes_A\Pi\to Q_0\lotimes_A\Pi\xrightarrow{f\otimes1} P\lotimes_A\Pi$, we get an exact sequence
	\[ \xymatrix{ \Hom_{\D(\Pi)}(Q\lotimes_A\Pi,Q_0\lotimes_A\Pi)\ar[r]^{f\otimes1}&\Hom_{\D(\Pi)}(Q\lotimes_A\Pi,P\lotimes_A\Pi)\ar[r]&\Hom_{\D(\Pi)}(Q\lotimes_A\Pi,P'\lotimes_A\Pi[1]) }. \]
	Now, since $\mu_Q^+(P)=P'\oplus Q$ is a $d$-silting object by assumption, $(P'\oplus Q)\lotimes_A\Pi$ is a silting object of $\Pi$ by Theorem \ref{from d-silting to ctilt}, and hence
	%the Calabi-Yau completion of its derived endomorphism algebra is connective, and therefore 
	$\Hom_{\D(\Pi)}(Q\lotimes_A\Pi,P'\lotimes_A\Pi[1])=0$. We conclude that $f\otimes1$ is a right $(\add Q\lotimes_A\Pi)$-approximation and hence $\mu^+_{Q\lotimes_A\Pi}(P\lotimes_A\Pi)=(P'\lotimes_A\Pi)\oplus (Q\lotimes_A\Pi)=\mu^+_Q(P)\lotimes_A\Pi$
\end{proof}

We obtain the following consequence.

\begin{Thm}\label{Hasse}
Let $A$ be a smooth dg $k$-algebra, $d\ge0$ an integer, and $\Pi$ the $(d+1)$-Calabi-Yau completion of $A$ such that $\Pi$ is $H^0$-finite and connective. 
\begin{enumerate}
\item The embedding of posets $-\lotimes_A\Pi:\silt^dA\to\silt\Pi$ identifies the Hasse quiver of $\silt^dA$ as a full subquiver of the Hasse quiver of $\silt\Pi$.
\item If $-\lotimes_AA^\vee$ gives a fully faithful functor $\per A\to\per A$ (e.g.\ $A$ is proper), then the Hasse quiver of $\silt^dA$ has neither sinks nor sources.
\end{enumerate}
\end{Thm}

\begin{proof}
(1) Let $\T=\per A$ or $\per\Pi$.
By our assumption, $\T$ is Hom-finite and Krull-Schmidt. In particular, silting mutations always exist in $\T$. By \cite[Theorem 2.35]{AI}, the Hasse quiver of $\silt\T$ is given by irreducible mutations.
By Proposition \ref{mutation compatible}, the map $\silt^dA\to\silt\Pi$ preserves (irreducible) mutations.
	
It remains to show that $-\lotimes_A\Pi$ detects irreducible mutations, that is, if $T,T'\in\silt^dA$ and $T'\lotimes_A\Pi$ is an irreducible left mutation of $T'\lotimes_A\Pi$, then $T'$ is an irreducible left mutation of $T$. By assumption, the objects $T\lotimes_A\Pi$ and $T'\lotimes_A\Pi$ share all indecomposable summands but one, thus so do $T$ and $T'$. Therefore, we may write $T=U\oplus X$ and $T'=U\oplus X'$ with $X$ and $X'$ indecomposable, and thus $T\lotimes_A\Pi=(U\lotimes_A\Pi)\oplus(X\lotimes_A\Pi)$ and $T'\lotimes_A\Pi=(U\lotimes_A\Pi)\oplus(X'\lotimes_A\Pi)$. Now, by $T'\lotimes_A\Pi=\mu^-_{U\lotimes_A\Pi}(T\lotimes_A\Pi)$, we have $T\lotimes_A\Pi> T'\lotimes_A\Pi\geq T\lotimes_A\Pi[1]=(T[1])\lotimes_A\Pi$, so Proposition \ref{from d-silting to ctilt 0}(3) yields $T>T'\geq T[1]$ since $T,T',T[1]\in\silt^dA$. Then Theorem \ref{mutations} shows that the only possiblity for such $T'$ containing $U$ as a direct summand is $\mu_U^-(T)$. We therefore conclude that $T'$ must be an irreducible left mutation of $T$.

(2) Let $G:=-\lotimes_AA^\vee[d]:\per A\to\per A$. For each $P\in\silt A$, 
Let $A':=\REnd_A(P)$ and $\Pi'$ the $(d+1)$-Calabi-Yau completion of $A'$.
Since $\per\Pi'\simeq\per\Pi$ is $\Hom$-finite, $\Pi'$ is $H^0$-finite.
We claim $\add P\neq\add GP$. Otherwise, $\add G^iP=\add P$ holds for all $i\in\Z$, and hence $\Pi'$ is not $H^0$-finite by the first isomorphism of Lemma \ref{theta theta'}, a contradiction.

Applying Proposition \ref{F-silting mutation} to the equivalence $G:\per A\to\per A$, we obtain the assertion. 
\end{proof}

Now we consider the special case below. Recall that, for dg algebras $A$ and $B$, their \emph{tensor product} is a dg algebra whose underlying complex is $A\otimes_kB$ and multiplication is
given by $(a\otimes b)(a'\otimes b')=(-1)^{\deg b\cdot\deg a'}(aa'\otimes bb')$ for homogeneous elements $a,a'\in A$ and $b,b'\in B$.

\begin{Thm}\label{CYs}
Let $d> e\ge0$ be integers, $A$ a connective $e$-Calabi-Yau dg algebra, and $\Pi=\Pi_{d+1}(A)$.
\begin{enumerate}
\item We have a quasi-isomorphism $\Pi\simeq k[x]\otimes_kA$ of dg algebras with $\deg x=e-d$.
\item $\silt A=\silt^eA$ holds, and we have mutually inverse isomorphisms of posets
\[\xymatrix{-\lotimes_A\Pi:\silt A\ar@<.2em>[r]&\silt\Pi:-\lotimes_{\Pi}A\ar@<.2em>[l].}\]
\end{enumerate}
\end{Thm}

This is a consequence of a more general result below.
The \emph{graded center} of a dg algebra $C$ is the dg subalgebra $Z$ of $C$ whose homogeneous element $a\in Z$ satisfies $ab=(-1)^{\deg a\cdot \deg b}ba$ for every homogeneous $b\in C$. Each $a\in Z$ gives a morphism $(a\cdot):C\to C[\deg a]$ in $\D(C^e)$.

\begin{Prop}\label{A[x]}
Let $A$ be a dg algebra, and $B=k[x]\otimes_kA$ with $\deg x=-n\in\Z$.
\begin{enumerate}
\item We have a quasi-isomorphism $T_A(A[n])\xsimeq k[x]\otimes_kA$ of dg algebras.
\item For every $n\geq0$, we have morphisms of posets
\[ \xymatrix{ -\lotimes_AB:\silt A\ar@<.2em>[r]&\silt B:-\lotimes_{B}A\ar@<.2em>[l] }. \]
which are mutually inverse isomorphisms if $n\ge1$.
\end{enumerate}
\end{Prop}
\begin{proof}
	(1)  Consider the map $A\to k[x]\otimes_kA$, $a\mapsto 1\otimes a$ of dg algebras, and a map $A[n]\to k[x]\otimes_kA$, $a\mapsto x\otimes a$ of $(A,A)$-bimodules. These induce a dg algebra homomorphism $T_A(A[n])\to k[x]\otimes_kA$, which is clearly a quasi-isomorphism.
	
	(2)  We apply Proposition \ref{from d-silting to ctilt 0} to the bimodule $\theta=A[n]$. Suppose $n\geq0$. Then for every $P\in\silt A$ we have $\Hom_{\D(A)}(P,P\lotimes_AA[n][i])=0$ for $i>0$, so Proposition \ref{from d-silting to ctilt 0}(2) shows that $-\lotimes_AB\colon\silt A\to\silt B$ is well-defined, and this preserves (and detects) partial orders by (3).
	
	We next claim that for $X,Y\in\per B$, we have $\Hom_{\D(B)}(X,Y[i])=0$ for every $i>0$ if and only if $\Hom_{\D(A)}(X\lotimes_B A,Y\lotimes_B A[i])=0$ holds for every $i>0$.
	Indeed, $\RHom_{A}(X\lotimes_B A,Y\lotimes_B A)=\RHom_{B}(X,Y\lotimes_B A)$ holds.
	Since $x$ is in the graded center of $B$, we have a triangle $B[n]\xrightarrow{x\cdot}B\to A$ in $\D(B^e)$. Applying $\RHom_{B}(X,Y\lotimes_B-)$, we get a triangle
	\[\RHom_{B}(X,Y)[n]\xrightarrow{x}\RHom_{B}(X,Y)\to\RHom_{B}(X,Y\lotimes_B A)\to\RHom_{\Pi}(X,Y)[n+1].\]
	Now, if $\RHom_B(X,Y[i])=0$ for $i>0$, then the first two terms are concentrated in non-positive degrees, so is the third one.
	Conversely, suppose that the third term is concentrated in non-positive degrees. Noting that $n>0$, we easily see that the top cohomologies of the last terms coincide. Thus the claim follows.
	
	Now we apply the claim for $X=Y:=Q\in\silt B$. Then we obtain $Q\lotimes_B A\in\silt A$ since $\thick Q=\per\Pi$ implies $\thick(Q\lotimes_{\Pi}A)=\per A$.
\end{proof}

\begin{proof}[Proof of Theorem \ref{CYs}]
(1) Since $A$ is $e$-Calabi-Yau, we have a quasi-isomorphism $\Theta=\RHom_{A^e}(A,A^e)[d]\to A[d-e]$ of $A^e$-modules.
Thus we have a quasi-isomorphism $\Pi=T_A(\Theta)\to T_A(A[d-e])$ of dg algebras. Also, for a dg algebra $k[x]=(k[x],0)$ with $\deg x=e-d$, we have a quasi-isomorphism of dg algebras $T_A(A[d-e])\simeq k[x]\otimes_kA$.\\
(2) We have the isomorphism of posets by Proposition \ref{A[x]}. It remains to prove $\silt^eA=\silt A$. For each $P\in\silt A$, $A':=\REnd_A(P)$ is dg Morita equivalent to $A$, so it is $e$-Calabi-Yau. Thus inverse dualizing complex of $A'$ is $A'[-e]$, and hence $\silt A=\silt^eA$ holds.
%Thanks to Theorem \ref{from d-silting to ctilt}, we have a map $-\lotimes_A\Pi:\silt A\to\silt\Pi$.
%
%For $X,Y\in\per\Pi$, we claim that, if $\Hom_{\D(\Pi)}(X,Y[i])=0$ holds for each $i>0$, then $\Hom_{\D(A)}(X\lotimes_\Pi A,Y\lotimes_\Pi A[i])=0$ holds for each $i>0$.
%In fact, $\RHom_{A}(X\lotimes_\Pi A,Y\lotimes_\Pi A)=\RHom_{\Pi}(X,Y\lotimes_\Pi A)$ holds.
%We may take a triangle $\Pi[d-e]\xrightarrow{x}\Pi\to A$ in $\D(\Pi^e)$ since $x$ is in the graded center of $\Pi$ (i.e. one has $xy=(-1)^{|x||y|}yx$ for every homogeneous $y\in\Pi$). Applying $Y\lotimes_\Pi-$, we have a triangle $Y[d-e]\xrightarrow{x}Y\to Y\lotimes_\Pi A$ in $\D(\Pi)$.
%Applying $\RHom_{\Pi}(X,-)$, we have a triangle
%\[\RHom_{\Pi}(X,Y)[d-e]\xrightarrow{x}\RHom_{\Pi}(X,Y)\to\RHom_{\Pi}(X,Y\lotimes_\Pi A)\to\RHom_{\Pi}(X,Y)[d-e+1].\]
%Since the first two terms are concentrated in non-positive degrees, so is the third one. Thus the claim follows.
%
%For $Q\in\silt\Pi$, we apply the claim for $X=Y:=Q$. Then we obtain $Q\lotimes_\Pi A\in\silt\Pi$ since $\thick Q=\per\Pi$ implies $\thick(Q\lotimes_{\Pi}A)=\per A$.
\end{proof}

%We give an example of CY completion of a smooth and non-proper dg algebra. \comment{Think CY completion of $A:=kQ$ for 2-cycle $Q$.}
\begin{Ex}
Let $e\geq0$ be an integer and let $A=k[y]$ be the dg algebra with zero differential and $\deg y=-e+1$ which is $e$-Calabi-Yau. Then for $d> e$, its $(d+1)$-Calabi-Yau completion is $\Pi=k[x]\otimes_kk[y]$ with $\deg x=e-d$. Then the bijection in Theorem \ref{CYs} is nothing but the identity map on $\Z$:
\[ \xymatrix@R=2mm{
	\silt A\ar@{=}[d]\ar@{-}[r]&\silt\Pi\ar@{=}[d]\\
	\{A[i]\mid i\in\Z\}\ar@{=}[r]&\{\Pi[i]\mid i\in\Z\}. } \]
Note that $\Pi=k[x]\otimes_kk[y]\simeq k\left\langle x,y\right\rangle /(xy-(-1)^{(e-d)(e-1)}yx)$ as can be seen from the equality $(x\otimes1)(1\otimes y)=(-1)^{(e-d)(e-1)}(1\otimes y)(x\otimes1)$ in the tensor product.
\end{Ex}

\section{$\F$-Liftable and liftable Calabi-Yau dg algebras}\label{liftable}

\subsection{Definition and basic properties}\label{basic}

Throughout this section, let $d$ be a positive integer, and let $\Pi$ be a $(d+1)$-CY dg algebra over a field $k$ which is connective and $H^0\Pi$ is finite dimensional. Let $\C(\Pi)$ be the cluster category, which is a $d$-Calabi-Yau triangulated category, and
%and the image of $\Pi\in\per\Pi$ in $\C(\Pi)$ is a $d$-cluster tilting object.
$\pi:\per\Pi\to\C(\Pi)$ the canonical functor. Then we have a map 
\begin{equation}\label{sct}
	\xymatrix{\silt\Pi\ar[r]& \ct{d}\C(\Pi).}
\end{equation}
%$\silt\Pi\to\dctilt\C(\Pi)$, see \eqref{silt to dctilt}.
%\subsection{Definition and basic properties}
Now we consider the full subcategory (called the {\it fundamental domain})
\[ \F=\P\ast\P[1]\ast\cdots\ast\P[d-1] \subset\per\Pi, \]
where $\P=\add\Pi\subset\per\Pi$. Then the composition $\F\subset\per\Pi\to\C(\Pi)$ is an additive equivalence \cite[proof of 2.9]{Am09}\cite[2.15]{Guo} (see also \cite[5.8]{IYa1} and \cite[2.12]{IYa2}). In particular the map (\ref{sct}) restricts to an injection
\begin{equation}\label{fsct}
	\xymatrix{\silt\Pi\cap\F:=\{X\in\F\mid X\in\silt\Pi\}\ar@{^(->}[r]& \ct{d}\C(\Pi).}
\end{equation}
%where $\silt^\F\!\Pi$ is the subset of $\silt\Pi$ contained in $\F$.

\begin{Def}
%We call $\Pi$ {\it liftable} if the map $\silt\Pi\to\ct{d}\Pi$ is surjective.
We call an $H^0$-finite connective $(d+1)$-Calabi-Yau dg algebra $\Pi$ {\it $\F$-liftable} if the map (\ref{fsct}) is surjective, or equivalently, bijective.
\end{Def}

This is automatic for $d=1$ and $2$ by the following result.

\begin{Thm}[{\cite{KN}, see also \cite[Corollary 5.12]{IYa1}}]
If $d=1$ or $2$, then every $H^0$-finite connective $(d+1)$-Calabi-Yau dg algebra is $\F$-liftable.
\end{Thm}

%For $d=1$ and $2$, $\Pi$ is always mild [Keller-Nicolas, in preparation] (see also ).
This motivates the following question.

\begin{Qs}
For $d\ge3$, which $H^0$-finite connective $(d+1)$-Calabi-Yau dg algebra is $\F$-liftable?
\end{Qs}
It was conjectured in \cite[Conjecture 5.14]{IYa1} that every $(d+1)$-Calabi-Yau dg algebra is $\F$-liftable. We present a simple counter-example.
For this we note the following necessary condition for $\F$-liftability.
\begin{Prop}\label{lift to F}
	Let $\Pi$ be an $H^0$-finite connective $(d+1)$-Calabi-Yau dg algebra. Let $l\geq0$.
	\begin{enumerate}
		\item The preimage of the $d$-cluster tilting object $\Pi[-l]\in\C(\Pi)$ in $\F$ is $\Pi^{\leq-l}[-l]$.
		\item If $\Pi$ is $\F$-liftable, then $\Pi^{\leq-l}$ belongs to $\silt\Pi$.
	\end{enumerate}	
\end{Prop}
\begin{proof}
	We have to prove $\Pi^{\leq-l}\in\Pi[l]\ast\cdots\ast\Pi[l+d-1]$, that is, $\Hom_{\D(\Pi)}(\Pi[<\!l],\Pi^{\leq-l})=0$ and $\Hom_{\D(\Pi)}(\Pi^{\leq-l},\Pi[\geq\!l+d])=0$. The first vanishing is clear, so we look at the second one. Consider the triangle
	\[ \xymatrix{ \Pi^{\leq-l}\ar[r]&\Pi\ar[r]&\Pi^{[-l+1,0]}\ar[r]&\Pi^{\leq-l}[1] } \]
	in $\D(\Pi)$. Applying $\Hom_{\D(\Pi)}(-,\Pi[\geq\!l+d])$, we have an isomorphism
    \[\Hom_{\D(\Pi)}(\Pi^{\leq-l},\Pi[\geq\!l+d])\xrightarrow{\simeq}\Hom_{\D(\Pi)}(\Pi^{[-l+1,0]},\Pi[\geq\!l+d+1]).\]
    Now the last term is isomorphic by relative Serre duality to $D\Hom_{\D(\Pi)}(\Pi[\geq\!l+d+1],\Pi^{[-l+1,0]}[d+1])=D\Hom_{\D(\Pi)}(\Pi,\Pi^{[-l+1,0]}[\leq\!-l])=0$, which completes the proof.
\end{proof}

The following gives a simple counter-example to \cite[Conjecture 5.14]{IYa1}. We will give more systematic results in Section \ref{bunrui}, see Theorems \ref{H0=k} and \ref{hereditary}.
\begin{Ex}\label{EX}
Consider the dg algebra
\[ \Pi=k\!\left\langle x,y\right\rangle\!/(xy+yx), \quad \deg x=-1, \, \deg y=-1 \]
with vanishing differentials. By \cite[Theorem 5.2]{ha3} (see also Examples 5.9 and 5.10 there) $\Pi$ is a $4$-Calabi-Yau dg algebra. Moreover, $\Pi$ is not $\F$-liftable.
\end{Ex}
\begin{proof}
If $\Pi$ is $\F$-liftable, then Proposition \ref{lift to F} shows that $\Pi^{\leq-1}$ is a silting object, in particular $\Hom_\D(\Pi^{\leq-1},\Pi^{\leq-1}[2])=0$. We show that this is not the case.
Consider the truncation $\Pi^{\leq-1}$ which fits into a triangle
\[ \xymatrix{ \Pi^{\leq-1}\ar[r]& \Pi\ar[r]& k }\]
in $\D:=\per\Pi$. Applying $\Hom_\D(-,\Pi^{\leq-1})$ yields an exact sequence
\[ \xymatrix@R=2mm{
	\Hom_\D(\Pi,\Pi^{\leq-1}[2])\ar[r]\ar@{=}[d]&\Hom_\D(\Pi^{\leq-1},\Pi^{\leq-1}[2])\ar[r]&\Hom_\D(k,\Pi^{\leq-1}[3])\ar[r]&\Hom_\D(\Pi,\Pi^{\leq-1}[3])\ar@{=}[d] \\0&&&0} \]
in which both end terms are $0$.
Also by Serre duality we have $\Hom_\D(k,\Pi^{\leq-1}[3])=D\Hom_\D(\Pi^{\leq-1},k[1])=k^2\neq0$, and therefore $\Hom_\D(\Pi^{\leq-1},\Pi^{\leq-1}[2])=k^2\neq0$.
We therefore conclude that $\Pi$ is not $\F$-liftable.
\end{proof}

These examples suggest that it is more reasonable to consider the following condition.
%\subsection{Liftable Calabi-Yau dg algebras}%{Definition and basic properties}
%Let $d\geq1$ and let $\Pi$ be a connective $(d+1)$-Calabi-Yau dg algebra such that $H^0\Pi$ is finite dimensional. The canonical functor $\pi:\per\Pi\to\C(\Pi)$ gives a map
%\[ \xymatrix{ \per\Pi\ar[r]& \C(\Pi) }. \]
%Then $\Pi$ is a silting object in $\per\Pi$ and a $d$-cluster tilting object in $\C(\Pi)$ \cite[2.1]{Am09}\cite[2.2]{Guo}. More generally, the functor $\per\Pi\to\C(\Pi)$ sends each silting objet in $\per\Pi$ to a $d$-cluster tilting object in $\C(\Pi)$ \cite[5.12]{IYa1}. Therefore we have a map
%where $\silt\Pi$ (resp. $\ct{d}\C(\Pi)$) is the set of isomorphism classes of silting objects in $\per\Pi$ (resp. $d$-cluster tilting objects in $\C(\Pi)$).

\begin{Def}
We call an $H^0$-finite connective $(d+1)$-Calabi-Yau dg algebra $\Pi$ {\it liftable} if the map \eqref{sct} is surjective.
\end{Def}

We shall see in Theorem \ref{poly}(\ref{bij}) that the dg algebra in Example \ref{EX} is indeed liftable.

Certainly, the liftability is invariant under dg Morita equivalence while this is not the case for $\F$-liftability (see Example \ref{tilted}).
The following proposition shows that liftability is in fact invariant under the equivalence of their cluster categories.

We say that a functor between algebraic triangulated categories (with fixed enhancements) are {\it algebraic} if it is induced by a bimodule. Precisely, let $\A$ and $\B$ be dg categories, and $F\colon\per\A\to\per\B$ a triangle functor. We say it is {\it algebraic} if there is an $(\A,\B)$-bimodule $X$ which is perfect as a right $\B$-module such that $F\simeq-\lotimes_\A X$.
%Let $\T$ be an algebraic triangulated category with a fixed enhancement. Suppose that there is a $(d+1)$-Calabi-Yau dg algebra $\Pi$ and an algebraic equivalence $\C(\Pi)\simeq\T$ such that $\silt\Pi\to d\text{-}\!\ctilt\T$ is surjective. We denote by $F$ the composite algebraic functor $\per\Pi\to\C(\Pi)\xsimeq\T$. 

\begin{Prop}\label{BO}
Let $\Pi$ and $\Pi'$ be $H^0$-finite connective $(d+1)$-Calabi-Yau dg algebras such that there exists an algebraic equivalence $\C(\Pi)\simeq\C(\Pi^\prime)$.
If $\Pi$ is liftable, then $\Pi$ and $\Pi'$ are dg Morita equivalent, and therefore $\Pi'$ is also liftable.
%We write $F^\prime\colon\per\Pi^\prime\to\C(\Pi)$.
%If $\Pi$ is liftable, then there exists a commutative diagram below of algebraic equivalences.
%\[ \xymatrix@R=5mm@C=10mm{
%	\per\Pi\ar[r]\ar[d]^-\rsimeq&\C(\Pi)\ar[d]^-\rsimeq\\
%	\per\Pi^\prime\ar[r]&\C(\Pi^\prime)} \]
%In particular, we have the following.
%\begin{enumerate}
%	\item The dg algebras $\Pi$ and $\Pi^\prime$ are derived Morita equivalent.
%	\item $\Pi'$ is also liftable.
%	%\item For any $d$-cluster tilting object $X\in\T$ its truncated derived endomorphism ring is smooth.
%\end{enumerate}
\end{Prop}
\begin{proof}
Let $\Pi\to\Ga$ and $\Pi'\to\Ga'$ be the canonical localizations such that $\per\Ga=\C(\Pi)$ and $\per\Ga'=\C(\Pi')$, see \eqref{induction}. Let $X\in\per\Ga$ be the image of $\Ga^\prime$ under the algebraic equivalence $\per\Ga^\prime\xsimeq\per\Ga$. Then we have a quasi-isomorphism $\Ga'\xsimeq\RHom_\Gamma(X,X)$.
By Lemma \ref{leq0} we have a quasi-isomorphism $\Pi^\prime\xsimeq\Ga'{}^{\leq0}\xsimeq\RHom_\Gamma(X,X)^{\leq0}$. 

Since $X$ is a $d$-cluster tilting object in $\C(\Pi)=\per\Ga$, there exists a silting object $M\in\per\Pi$ such that $X\xsimeq M\lotimes_\Pi\Ga$. Clearly $\Pi$ and $\REnd_{\Pi}(M)$ are dg Morita equivalent, and again by Lemma \ref{leq0} there is a quasi-isomorphism $\REnd_\Pi(M)\xsimeq\RHom_\Gamma(X,X)^{\leq0}$. It follows that $\REnd_\Pi(M)$ and $\Pi^\prime$ are quasi-equivalent. Now, letting the functor be $\RHom_\Pi(M,-)\colon\per\Pi\to\per\Pi^\prime$, we obtain the commutative diagram.
\end{proof}

We give another application of liftability.
For an object $X$ in a Krull-Schmidt category $\C$, we denote by $|X|_{\C}$ the number of isomorphism classes of indecomposable direct summands of $X$. Although the following natural question is well-established for the case $d=1$ and $d=2$ \cite{AIR,DK}, it is widely open for $d\ge 3$.

\begin{Cj}\label{constant}
Let $\Pi$ be an $H^0$-finite connective $(d+1)$-Calabi-Yau dg algebra. Then $|T|_{\C(\Pi)}$ is constant for each $T\in\dctilt\C(\Pi)$.
\end{Cj}

The liftability implies that Conjecture \ref{constant} has a positive answer. 

\begin{Prop}\label{liftable and constant}
Let $\Pi$ be a liftable $(d+1)$-Calabi-Yau dg algebra. Then, for each $T\in\ct{d}\C(\Pi)$, we have $|T|_{\C(\Pi)}=|\Pi|_{\D(\Pi)}$.
\end{Prop}

\begin{proof}
By \cite{AI}, $|P|_{\D(\Pi)}=|\Pi|_{\D(\Pi)}$ for each $P\in\silt\Pi$ is constant. Thus $|P|_{\C(\Pi)}=|P|_{\D(\Pi)}=|T|_{\D(\Pi)}$ holds. Since $\silt\Pi\to\dctilt\C(\Pi)$ is surjective, the claim follows.
\end{proof}

We end this subsection by posing the following.

\begin{Cj}
Any $H^0$-finite connective Calabi-Yau dg algebra is liftable.
\end{Cj}

\subsection{Mutation, reduction and ($\F$-)liftablity}

%\comment{Define Silting mutation and cluster tilting mutation.}
%We call $\Pi$ \emph{silting-connected} if the exchange graph of $\silt\Pi$ is connected.
We call $\Pi$ \emph{CT-connected} if any $d$-cluster tilting object in $\C(\Pi)$ is obtained from $\Pi$ by a finite senqeunce of (not necessarily irreducible) cluster tilting mutation.
For example, if the exchange graph of $\dctilt\C(\Pi)$ is connected, then $\Pi$ is CT-connected.

\begin{Thm}\label{connected liftable}
Let $d\geq1$ and let $\Pi$ be an $H^0$-finite connective $(d+1)$-Calabi-Yau dg algebra.
\begin{enumerate}
\item If $\Pi$ is CT-connected, then $\Pi$ is liftable.
\item The image of $\silt\Pi\to\dctilt\C(\Pi)$ is a union of connected components.
\end{enumerate}
\end{Thm}

\begin{proof}
(1) Let $X\in\C(\Pi)$ be an arbitrary $d$-cluster tilting object. By assumption, we can take a sequence of mutations from the initial cluster tilting object $\Pi\in\C(\Pi)$ to $X$. By Proposition \ref{pm} we may lift each exchange triangle in $\C(\Pi)$ to a one in $\per\Pi$. It follows that $X$ is also the image of a silting object.

(2) is immediate from Proposition \ref{pm}.
\end{proof}

We next prove that silting reduction preserves $\F$-liftablity. For this it is convenient to work in the setting of Calabi-Yau triples which we now recall.

%\begin{Def}
For an integer $n$, recall that an \emph{$n$-Calabi-Yau triple} is a triple $(\T,\T^\fd,\M)$ satisfying the following conditions.
\begin{enumerate}
\renewcommand\labelenumi{(\roman{enumi})}
\renewcommand\theenumi{\roman{enumi}}
\item $\T$ is a Hom-finite Krull-Schmidt triangulated category and $\T^\fd$ is a thick subcategory of $\T$.
\item There exists a functorial isomorphism $D\Hom_\T(X,Y)\simeq\Hom_\T(Y,X[n])$ for each $X\in\T^\fd$ and $Y\in\T$.
\item $\M$ is a silting subcategory of $\T$ such that $(\M[<\!0]^\perp,\M[>\!0]^\perp)$ is a t-structure of $\T$ and satisfies $\M[>\!0]^{\perp}\subset\T^\fd$. Moreover, $\M$ is a dualizing $k$-variety.
\end{enumerate}
For example, each $n$-Calabi-Yau dg algebra $\Pi$ gives an $n$-Calabi-Yau triple $(\per\Pi,\pvd\Pi,\add\Pi)$.

The {\it cluster category} of an $n$-Calabi-Yau triple $(\T,\T^\fd,\M)$ is the Verdier quotient $\C:=\T/\T^\fd$. It has an $(n-1)$-cluster tilting subcategory $\pi(\M)$ where $\pi\colon\T\to\C$, and $\pi$ has a fundamental domain $\F:=\M\ast\cdots\ast\M[n-2]$. As in \eqref{silt to dctilt}, we have a well-defined map $\silt\T\to\ct{(n-1)}\C$. We call a Calabi-Yau triple $(\T,\T^\fd,\M)$ {\it $\F$-liftable} if $\F\cap\silt\T\to\ct{(n-1)}\C$ is bijective.

For an $n$-Calabi-Yau triple $(\T,\T^\fd,\M)$ and a functorially finite subcategory $\P$, let $\U:=\T/\thick\P$ be the silting reduction of $\T$. This gives an $n$-Calabi-Yau triple
\[(\U,\U^\fd,\N):=(\T/\thick\P,\T^\fd\cap(\thick\P)^\perp,\M),\]
which we call the \emph{silting reduction} of $(\T,\T^\fd,\M)$.
We are ready to state the main result of this subsection.
%\end{Def}

%Now we apply Proposition \ref{nterm} to the following setting.

\begin{Thm}\label{redm}
	Assume $d\ge1$. For a $(d+1)$-Calabi-Yau triple $(\T,\T^\fd,\M)$ and a functorially finite subcategory $\P\subset\M$, let $(\U,\U^\fd,\N)$ be the silting reduction.
	% of the CY triple with respect to $\P$. %, thus $\U=\T/\thick\P$, $\U^\fd$ is the image of $\T^\fd$
	If $(\T,\T^\fd,\M)$ is $\F$-liftable, then so is $(\U,\U^\fd,\N)$.
\end{Thm}
\begin{proof}
	Put $\C=\T/\T^\fd$ and $\D=\U/\U^\fd$, and consider the following diagram.
	\[ \xymatrix@R=5mm{
		[\M[d-1],\M]\ar[r]&d\text{-}\ctilt\C\\
		\silt_\P\!\T\cap[\M[d-1],\M]\ar[d]_-\lsimeq\ar[r]\ar@{}[u]|-\vsubset&d\text{-}\ctilt_\P\C\ar[d]^-\rsimeq\ar@{}[u]|-\vsubset\\
		[\N[d-1],\N]\ar[r]&d\text{-}\ctilt\D } \]
	Suppose that $(\T,\T^\fd,\M)$ is $\F$-liftable, so that the top horizontal map is bijective. Since the quotient functor $\pi\colon\T\to\C$ restricts to an additive equivalence $\M\ast\M[1]\ast\cdots\M[d-1]\to\C$, a silting subcategory $\A\in[\M[d-1],\M]$ contains $\P$ if and only if the corresponding $d$-cluster tilting subcategory $\pi(\A)$ of $\C$ contains $\pi(\P)$. Therefore we see that the middle horizontal map is also bijective. Now by Proposition \ref{nterm}, the lower vertical maps are bijective. We conclude that the bottom horizontal map is also bijective, that is, $(\U,\U^\fd,\N)$ is $\F$-liftable.
\end{proof}

Immediately, we obtain the following consequence.

\begin{Cor}\label{redm2}
	Let $\Pi$ be an $\F$-liftable $(d+1)$-Calabi-Yau dg algebra. For each idempotent $e$ of $H^0\Pi$, the dg quotient of $\Pi$ by $e$ is again an $\F$-liftable $(d+1)$-Calabi-Yau dg algebra. 
%direct summand $P$ of $\Pi$ in $\per A$, the dg quotient $\Pi/P$ is again an $\F$-liftable $(d+1)$-Calabi-Yau dg algebra. 
%	Let $(\T,\T^\fd,\M)$ be a $(d+1)$-CY triple and $\P\subset\M$ a functorially finite subcategory. Let $(\U,\U^\fd,\N)$ be the silting reduction of the CY triple with respect to $\P$. If $(\T,\T^\fd,\M)$ is $\F$-liftable, then so is $(\U,\U^\fd,\N)$.
\end{Cor}

\begin{proof}
$(\T,\T^\fd,\M):=(\per\Pi,\pvd\Pi,\add\Pi)$ and $\P:=\add e\Pi$, we obtain the assertion.
\end{proof}

\section{Classifications of certain $\F$-liftable Calabi-Yau dg algebras}\label{bunrui}
Having defined the notion of $\F$-liftable Calabi-Yau dg algebras, the fundamental question is the following.
\begin{Pb}
Classify $\F$-liftable Calabi-Yau dg algebras. %$\Pi$ with the fixed $0$-th cohomology $H^0\Pi$.
\end{Pb}
%Let $\Pi$ be a connective $(d+1)$-Calabi-Yau dg algebra whose $0$-th cohomology is finite dimensional. 
%Note that $H^0\Pi$ is nothing but the endomorphism ring of a $d$-cluster tilting object $\Pi$ in a $d$-Calabi-Yau triangulated category $\C(\Pi)$ (see \eqref{negative H}). 
%We therefore discuss the above problem for dg algebras with fixed $0$-th cohomologies.

The aim of this section is to establish this classification for Calabi-Yau dg algebras $\Pi$ such that $H^0\Pi=k$ or $H^0\Pi$ is hereditary (Theorems \ref{H0=k} and \ref{hereditary}).
In particular, we show that the most $(d+1)$-Calabi-Yau dg algebras for $d\geq 3$ are \emph{not} $\F$-liftable. Thus we obtain systematic counter-examples to the question posed in \cite[5.14]{IYa1}.

%It also shows that the most $(d+1)$-Calabi-Yau dg algebras for $d\geq 3$ are \emph{not} $\F$-liftable, answering 

\subsection{Calabi-Yau dg algebras with $H^0\Pi=k$}
Let $\Pi$ be a $(d+1)$-Calabi-Yau dg algebra. The aim of this first subsection is to characterize $\F$-liftablity for CY dg algebras $\Pi$ with $H^0\Pi=k$.

\begin{Thm}\label{H0=k}
Let $d\ge 1$, $\Pi$ be a connective $(d+1)$-Calabi-Yau dg $k$-algebra such that $H^0\Pi=k$. Then the following are equivalent.
\begin{enumerate}
\renewcommand{\labelenumi}{(\alph{enumi})}
\renewcommand{\theenumi}{\alph{enumi}}
\item\label{m} $\Pi$ is $\F$-liftable.
%\item $\Pi\simeq\Pi[d]$ in $\C(\Pi)$.
\item\label{dco} $\ct{d}\C(\Pi)=\{\Pi,\ldots,\Pi[d-1]\}$.
\item\label{tai} $\C(\Pi)$ is triangle equivalent to $\C_d(k)$, the $d$-cluster category of $k$.
\item\label{qis} $\Pi$ is quasi-isomorphic to $k[x]$ with $\deg x=-d$ and with zero differentials.
\end{enumerate}
In this case, we have $\Pi[d]\simeq\Pi$ holds in $\C(\Pi)$.
\end{Thm}
We refer to Section \ref{k} for examples of non-$\F$-liftable CY dg algebras with $H^0\Pi=k$, see Theorem \ref{poly} and Proposition \ref{A_2}.

To prove Theorem \ref{H0=k}, we prepare the following more general version for some of the above equivalences where we only require $H^0\Pi$ to be local.
\begin{Prop}\label{local}
Let $d\geq1$, and let $\Pi$ be a connective $(d+1)$-Calabi-Yau dg algebra such that $H^0\Pi$ is a finite dimensional local algebra. %Consider the following conditions.
\begin{enumerate}
\item $\Pi$ is $\F$-liftable if and only if $\ct{d}\C(\Pi)=\{\Pi,\ldots,\Pi[d-1]\}$.
%The following are equivalent.
%\begin{enumerate}
%\renewcommand{\labelenumi}{(\alph{enumii})}
%\renewcommand{\theenumi}{\alph{enumii}}
%\item\label{M} $\Pi$ is $\F$-liftable.
%\item $\Pi\simeq\Pi[d]$ in $\C(\Pi)$.
%\item\label{Dco} $\ct{d}\C(\Pi)=\{\Pi,\ldots,\Pi[d-1]\}$.
%\end{enumerate}
\suspend{enumerate}
Suppose that the above conditions are satisfied.
\resume{enumerate}
\item We have an isomorphism $\Pi\simeq\Pi[d]$ in $\C(\Pi)$.
\item\label{Pid} There is a triangle $\Pi[d]\to\Pi\to H^0\Pi$ in $\per\Pi$, and for each $i\geq0$, we have $H^{-i}\Pi=0$ when $i\mathrel{\not|}d$ and $H^{-i}\Pi=H^0\Pi$ when $i\mathrel{|}d$. 
\end{enumerate}
%Then the implications {\rm (\ref{m})$\Leftrightarrow$(\ref{dco})$\Rightarrow$(\ref{Pid})} hold.
\end{Prop}
\begin{proof}
	(1)  Since $H^0\Pi$ is local, the object $\Pi$ is indecomposable silting in $\per\Pi$. Thus $\silt\Pi=\{\Pi[i]\mid i\in\Z\}$ and $\silt\Pi\cap\F=\{\Pi[i]\mid0\leq i\leq d-1\}$. Therefore, $\F$-liftablity is equivalent to $\ct{d}\C(\Pi)=\{\Pi,\ldots,\Pi[d-1]\}$.
	
	(2)  Suppose that $\ct{d}\C(\Pi)=\{\Pi[i]\mid0\leq i\leq d-1\}$. In particular the $d$-cluster tilting object $\Pi[d]\in\C(\Pi)$ has to be isomorphic to $\Pi[i]$ for some $0\leq i\leq d-1$. The only possible $i$ is $0$ since $\Hom_{\C(\Pi)}(\Pi,\Pi[i])=0$ for $1\leq i\leq d-1$. 
	
	(3)  By \cite[5.9]{IYa1} the functor $\per\Pi\to\C(\Pi)$ induces bijections
	\begin{equation}\label{negative ext}
	\Hom_{\per\Pi}(\Pi[i],\Pi)\to\Hom_{\C(\Pi)}(\Pi[i],\Pi)\ \mbox{ for each }\ i>-d.
	\end{equation}
	This shows $H^{-i}\Pi=\Hom_{\C(\Pi)}(\Pi[i],\Pi)$ for $i\geq0$. Since $\Pi[d]\simeq\Pi$ in $\C(\Pi)$, this is $H^0\Pi$ if $i\mid d$, and $0$ if $i\mathrel{\not|}d$ since $\Hom_{\C(\Pi)}(\Pi,\Pi[i])=0$ for $0<i<d$.

	This \eqref{negative ext} also shows that we may lift an isomorphism $\Pi[d]\to\Pi$ in $\C(\Pi)$ to a morphism in $\per\Pi$. Extend it to the triangle
	\[ \xymatrix{ \Pi[d]\ar[r]&\Pi\ar[r]& X } \]
	in $\per\Pi$. We claim that $X=H^0\Pi$. Applying $\Hom_{\D(\Pi)}(\Pi[i],-)\to\Hom_{\C(\Pi)}(\Pi[i],-)$ to the triangle above, we get a commutative square
	\[ \xymatrix@R=5mm{
		\Hom_\D(\Pi[i],\Pi[d])\ar[r]^-\simeq\ar[d]&\Hom_\C(\Pi[i],\Pi[d])\ar[d]^-\rsimeq\\
		\Hom_\D(\Pi[i],\Pi)\ar[r]^-\simeq&\Hom_\C(\Pi[i],\Pi), } \]
	in which the horizontal maps are isomorphism for every $i>0$ by \eqref{negative ext}, and the right vertical map is an isomorphism by our choice of $\Pi[d]\to\Pi$. It follows that so is the left vertical map, and therefore $H^{-i}X=0$ for all $i>0$, which consequently yields $H^0X=H^0\Pi$.
\end{proof}

\begin{proof}[Proof of Theorem \ref{H0=k}]
	(\ref{m})$\Leftrightarrow$(\ref{dco})  This is \ref{local}(1).
	%Since $H^0\Pi=k$, $\Pi$ is an indecomposable silting object in $\per\Pi$. Thus $\silt\Pi=\{\Pi[i]\mid i\in\Z\}$ and $\silt\Pi\cap\F=\{\Pi[i]\mid0\leq i\leq d-1\}$.\\
	%(\ref{m})$\Rightarrow$(\ref{dco})+($\Pi[d]\simeq\Pi$)  Suppose that $\Pi$ is $\F$-liftable. Then $\ct{d}\C(\Pi)=\{\Pi[i]\mid0\leq i\leq d-1\}$. In particular the $d$-cluster tilting object $\Pi[d]\in\C(\Pi)$ has to be isomorphic to $\Pi[i]$ for some $0\leq i\leq d-1$. The only possible $i$ is $0$ since $\Hom_{\C(\Pi)}(\Pi,\Pi[i])=0$ for $1\leq i\leq d-1$. Thus (\ref{dco}) holds.\\
	
	(\ref{dco})$\Rightarrow$(\ref{qis})  
	By \ref{local}(2) we know that $\Pi\simeq\Pi[d]$ in $\C(\Pi)$ and that the cohomology $H^{-i}\Pi$ is $k$ if $i\mid d$, and $0$ if $i\mathrel{\not|}d$. Now lift an isomorphism $\Pi[d]\to\Pi$ in $\C(\Pi)$ to a morphism $f\colon\Pi[d]\to\Pi$ in $\per\Pi$, and let $y\in Z^{-d}\Pi$ give the morphism $f$ in $H^{-d}\Pi$. Then we obtain a homomorphism $k[x]\to\Pi$ of DG algebras, taking $x$ to $y$. Consider the power $y^n\in Z^{nd}\Pi$ of $y$. It presents the morphism $\Pi[nd]\xrightarrow{f[(n-1)d]}\Pi[(n-1)d]\to\cdots\xrightarrow{f[1]}\Pi[d]\xrightarrow{f}\Pi$ in $\per\Pi$ which is an isomorphism in $\C(\Pi)$. Therefore $y^n$ is non-zero in $H^{-nd}\Pi$. We conclude that $k[x]\to\Pi$ is a quasi-isomorphism.
	
	(\ref{qis})$\Rightarrow$(\ref{tai})  Since $k[x]$ with $\deg x=-d$ is the derived $(d+1)$-preprojective algebra of $k$, the assertion follows.
	
	(\ref{tai})$\Rightarrow$(\ref{m})  If $\C(\Pi)\simeq\C_d(k)$ then $\ct{d}\C(\Pi)=\{\Pi[i]\mid0\leq i\leq d-1\}$, thus $\Pi$ is $\F$-liftable.
\end{proof}

%We note that in \ref{local} the condition (2) does not imply (1).
%\begin{Ex}
%Let $A$ be the path algebra of the quiver of linearly oriented type $A_n$ modulo radical square: $A=kQ/\rad^2kQ$ for $Q=1\to\cdots\to n$.
%\end{Ex}
%
%
%\begin{Ex}
%Let $\Pi$ be a connective $2$-Calabi-Yau dg algebra such that $H^0\Pi$ is finite dimensional algebra with finite global dimension. Then $H^0\Pi$ is semisimple. \comment{Perhaps determine $\Pi$?}
%\end{Ex}

\subsection{Calabi-Yau dg algebras with hereditary $H^0\Pi$}
We next give characterizations of $\F$-liftability similar to Theorem \ref{H0=k} for CY dg algebras whose $0$-th cohomologies are hereditary.
%We give another example of $\F$-liftable Calabi-Yau dg algebra.
%Let
%\[\F:=(\add\Pi)*(\add\Pi)[1]*\cdots*(\add\Pi)[d-1]\subset\per\Pi.\]
%Recall that $\C(\Pi):=\per\Pi/\D^{\bb}\Pi$ is the cluster category of $\Pi$.
%Our first result gives a class of examples of $\F$-liftable CY dg algebras.
%We now give a certain converse of \ref{Pi of H is mild}, that is,
In particular, such Calabi-Yau dg algebras are precisely the Calabi-Yau completions of hereditary algebras.
%is quasi-isomorphic to the derived preprojective algebra of $H^0\Pi$.
%It also shows that the most $(d+1)$-Calabi-Yau dg algebras for $d\geq 3$ are \emph{not} $\F$-liftable, answering the question posed in \cite[5.14]{IYa1} negatively.
%More precisely, we prove the following result.
%Next we generalize Theorem \ref{H0=k} to the case when $H^0\Pi$ is hereditary.
%We have the following analogue of the previous section for hereditary $H^0\Pi$.

For dg algebras $A$ and $B$, we say that an $(A,B)$-bimodule $M$ is {\it invertible} if $-\lotimes_AM\colon\D(A)\to\D(B)$ is an equivalence.
Also, we say that dg algebras $A$ and $B$ are {\it quasi-equivalent} if there is an invertible $(A,B)$-bimodule $M$ such that $M\simeq B$ in $\D(B)$.
\begin{Thm}\label{hereditary}
Let $k$ be a perfect field, $d\ge2$ an integer, and $\Pi$ a $H^0$-finite connective $(d+1)$-Calabi-Yau dg algebra such that $A=H^0\Pi$ is hereditary. Then the following are equivalent.%Consider the following statements.
\begin{enumerate}
\renewcommand{\labelenumi}{(\alph{enumi})}
\renewcommand{\theenumi}{\alph{enumi}}
\item\label{m1} $\Pi$ is $\F$-liftable.
\item\label{van} $H^{-i}\Pi=0$ for each $1\le i\le d-2$.
\item\label{cda} There is a triangle equivalence $\C_d(A)\simeq\C(\Pi)$ taking $A$ to $\Pi$.
\item\label{qeq} $\Pi$ is quasi-equivalent to the $(d+1)$-CY completion $\Pi_{d+1}(A)$ of $A$.
\end{enumerate}
%Then the implications {\rm (\ref{qeq})$\Rightarrow$(\ref{m1})$\Rightarrow$(\ref{cda})} hold.
\end{Thm}
Note that we use the base field $k$ is perfect only for the implication (\ref{cda})$\Rightarrow$(\ref{qeq}).

We do not know whether the above condition (\ref{cda}) can be replaced by a weaker one ``there is a triangle equivalence $\C_d(A)\simeq\C(\Pi)$''.

In the rest, we prove Theorem \ref{hereditary}. The following observation shows (d)$\Rightarrow$(a).

\begin{Prop}\label{Pi of H is mild}
Let $H$ be a hereditary algebra and $d\geq2$. Then  the $(d+1)$-Calabi-Yau completion $\Pi=\Pi_{d+1}(H)$ of $H$ is $\F$-liftable.
%, that is, the canonical functor $\per\Pi\to\C(\Pi)$ gives a bijection
%\[\silt\Pi\cap\F\simeq\dctilt\C(\Pi).\]
\end{Prop}

To prove this, we use a triangle equivalence $\C(\Pi)\simeq\C_d(H)$ making the following diagram commutative.
\begin{equation}\label{H Pi C}
	\xymatrix{\D^{\bb}(H)\ar[r]\ar[d]_{-\Lotimes_H\Pi}&\C_d(H)\ar[d]^\rsimeq\\
		\per\Pi\ar[r]&\C(\Pi).}
\end{equation}
A key role is played by the following analogue of $\F$ for $H$:
\begin{align*}
	\F_H&:=(\add H)*(\add H)[1]*\cdots*(\add H)[d-1]\subset\D^{\bb}(H).
\end{align*}
Then $-\Lotimes_H\Pi$ restricts to a functor $\F_H\to\F$. Moreover, $\ind\D^{\bb}(H)=\bigsqcup_{j\in\Z}\ind\nu_d^{j}(\F_H)$ holds.

%\begin{Lem}\label{vanish}
%	We have $\Hom_{\D^{\bb}(H)}(\nu_d^{j}(\F_H),\F_H[i])=0$ for each $i,j\ge1$. In particular, any element in $\silt H\cap\F_H$ is $d$-silting.
%\end{Lem}

%\begin{proof}
%	We have $\F_H\subset\D^{\le0}(\mod H)$ and $\nu_d^{j}(\F_H)\subset\D^{\ge1}(\mod H)$. Since $H$ is hereditary, we have \[\Hom_{\D^{\bb}(H)}(\nu_d^{j}(\F_H),\F_H[i])\subset\Hom_{\D^{\bb}(H)}(\D^{\ge1}(\mod H),\D^{\le-1}(\mod H))=0.\qedhere\]
%\end{proof}

\begin{Prop}\label{H to Pi}
	\begin{enumerate}
\item	We have $\Hom_{\D^{\bb}(H)}(\nu_d^{j}(\F_H),\F_H[i])=0$ for each $i,j\ge1$.
\item Any element in $\silt H\cap\F_H$ is $d$-silting. Thus $-\Lotimes_H\Pi:\D^{\bb}(H)\to\per\Pi$ gives a map $\silt H\cap\F_H\to\silt\Pi\cap\F$.
	\end{enumerate}
\end{Prop}

\begin{proof}
(1)	We have $\F_H\subset\D^{\le0}(H)$ and $\nu_d^{j}(\F_H)\subset\D^{\ge1}(H)$. Since $H$ is hereditary, we have
	\[\Hom_{\D^{\bb}(H)}(\nu_d^{j}(\F_H),\F_H[i])\subset\Hom_{\D^{\bb}(H)}(\D^{\ge1}(\mod H),\D^{\le-1}(\mod H))=0.\]

(2) Immediate from (1) and Theorem \ref{from d-silting to ctilt}(2).
%	Fix $T\in\silt H\cap\F_H$. Since $\thick T=\D^{\bb}(H)\ni H$, we have $\Pi\in\thick(T\Lotimes_H\Pi)$ and hence $\thick(T\Lotimes_H\Pi)=\per\Pi$. It suffices to show that $\Hom_{\per\Pi}(T\Lotimes_H\Pi,T\Lotimes_H\Pi[i])=0$ for each $i\ge1$. Since
%	\[\RHom_{\Pi}(T\Lotimes_H\Pi,T\Lotimes_H\Pi)=\RHom_H(T,T\Lotimes_H\Pi)=\bigoplus_{j\ge0}\RHom_H(T,\nu_d^{-j}(T)),\]
%	we have $\Hom_{\per\Pi}(T\Lotimes_H\Pi,T\Lotimes_H\Pi[i])=\bigoplus_{j\ge0}\Hom_{\D^{\bb}(H)}(T,\nu_d^{-j}(T)[i])$. Since $T\in\silt H$, we have $\RHom_H(T,T[i])=0$ for each $i\ge1$. By Lemma \ref{vanish}, we have $\RHom_H(T,\nu_d^{-j}(T)[i])=0$ for each $i,j\ge1$. Thus the assertion follows.
\end{proof}

\begin{proof}[Proof of Proposition \ref{Pi of H is mild}]
	It suffices to prove that the map $\silt\Pi\cap\F\simeq\dctilt\C(\Pi)$ is surjective.
	By \cite[Proposition 2.4]{BRT}, the canonical functor gives a bijection
	\[\silt H\cap\F_H\simeq\dctilt\C_d(H).\]
	Since we have a map $-\Lotimes_H\Pi:\silt H\cap\F_H\to\silt\Pi\cap\F$ by Proposition \ref{H to Pi}, the claim follows from the commutative diagram \eqref{H Pi C}.
\end{proof}

The following observation shows \eqref{m1}$\Rightarrow$\eqref{van}.
\begin{Lem}\label{ext=0}
Let $\Pi$ be an $\F$-liftable $(d+1)$-CY dg algebra such that $H^0\Pi$ is $1$-Iwanaga-Gorenstein. Then we have $H^{-i}\Pi=0$ for all $1\leq i\leq d-2$.
%$\Hom_{H^0\Pi}(H^{-i}\Pi,H^0\Pi)=0$ and $\Ext^1_{H^0\Pi}(H^{-i}\Pi,H^0\Pi)=0$ for all $1\leq i\leq d-2$.
\end{Lem}
\begin{proof}
	We denote by $\D$ the derived category $\D(\Pi)$, and $\Ext^i_\D(-,-)$ for $\Hom_\D(-,-[i])$. We will prove $\Hom_{\D}(H^{-i}\Pi,H^0\Pi)=0$ and $\Ext^1_{\D}(H^{-i}\Pi,H^0\Pi)=0$ by induction on $i$. Since $\Ext^1_\D(-,-)$ computed in $\D$ coincides with $\Ext^1_{H^0\Pi}(-,-)$ computed in the heart $\md H^0\Pi$, these vanishing will yield $\RHom_{H^0\Pi}(H^{-i}\Pi,H^0\Pi)=0$ by $\id H^0\Pi\leq1$, and therefore $H^{-i}\Pi=0$.
	
	Let $1\leq i\leq d-2$ and suppose we have proved the assertions for $\leq i-1$. Then we have a triangle in $\D$:
	\[ \xymatrix{ \Pi^{\leq -i}\ar[r]&\Pi\ar[r]&H^0\Pi\ar[r]&\Pi^{\leq-i}[1] }. \]
	Since $\Pi$ is $\F$-liftable, Proposition \ref{lift to F} implies that $\Pi^{\leq -i}$ is a silting object in $\D$. Applying $\Hom_{\D}(-,\Pi^{\leq-i})$ gives an exact sequence below.
	\[ \xymatrix@R=3mm{
		\Ext^{d-i}_{\D}(\Pi^{\leq-i},\Pi^{\leq-i})\ar[r]\ar@{=}[d]&\Ext^{d-i+1}_\D(H^0\Pi,\Pi^{\leq-i})\ar[r]\ar@{}[d]|\rsimeq&\Ext^{d-i+1}_\D(\Pi,\Pi^{\leq-i})\ar@{=}[d]\\
		0&D\Hom_\D(H^{-i}\Pi,H^0\Pi)&0 } \]
	%We know that the left end term is $0$ by \ref{}, and so is the right end term.
	Thus the middle term is also $0$. Now by relative Serre duality we have $\Ext^{d-i+1}_\D(H^0\Pi,\Pi^{\leq-i})=D\Ext^i_\D(\Pi^{\leq-i},H^0\Pi)=D\Hom_{\D}(H^{-i}\Pi,H^0\Pi)$, and we get that this is $0$.
	
	Also, applying the same functor to the same triangle gives the following exact sequence.
	\[ \xymatrix@R=3mm{
		\Ext^{d-i-1}_{\D}(\Pi^{\leq-i},\Pi^{\leq-i})\ar[r]\ar@{=}[d]&\Ext^{d-i}_\D(H^0\Pi,\Pi^{\leq-i})\ar[r]&\Ext^{d-i}_\D(\Pi,\Pi^{\leq-i})\ar@{=}[d]\\
		0&&0 } \]
	Similarly to above the end terms are $0$ (by $i\leq d-2$), thus so is the middle term. Then relative Serre duality shows $0=D\Ext^{d-i}_\D(H^0\Pi,\Pi^{\leq-i})=\Ext^{i+1}_\D(\Pi^{\leq-i},H^0\Pi)=\Hom_\D(\Pi^{[-i-1,-i]},H^0\Pi[i+1])$. Applying $\Hom_\D(-,H^0\Pi)$ to the triangle $H^{-i-1}\Pi[i+1]\to\Pi^{[-i-1,-i]}\to H^{-i}\Pi[i]\to$ yields an exact sequence
	\[ \xymatrix@R=3mm{ \Hom_\D(H^{-i-1}\Pi[i+1],H^0\Pi[i])\ar[r]\ar@{=}[d]&\Hom_\D(H^{-i}\Pi[i],H^0\Pi[i+1])\ar[r]&\Hom_\D(\Pi^{[-i-1,-i]},H^0\Pi[i+1])\ar@{=}[d]\\
	0&&0 }, \]
	in which the left end term is $0$ by the standard $t$-structure and the right end term is $0$ by the above discussion. We conclude that $\Ext^1_\D(H^{-i}\Pi,H^0\Pi)=0$, as desired.
\end{proof}

Let us prepare some necessary observations to prove the implication (\ref{cda})$\Rightarrow$(\ref{qeq}).
The first one is about lifting an autoequivalence of the cluster category to that of the derived category.
\begin{Lem}\label{lift}
Let $H$ be a finite dimensional hereditary algebra.
Let $F\colon\C_d(H)\to\C_d(H)$ be a triangle autoequivalence. Then there exists a triangle autoequivalence $G\colon\D^b(H)\to\D^b(H)$ such that $F\pi H \simeq\pi GH$.
\end{Lem} 

Let $H$ be a ring-indecomposable hereditary algebra, and $Q$ the valued quiver of $H$. It is well-known \cite{Hap88} that the Auslander-Reiten quiver ${\rm AR}(\per H)$ of $\per H$ has a connected component $C^{\rm t}$ which satisfies $H\in\add C^{\rm t}$ and has the form $\Z Q$. We call $C^{\rm t}$ the \emph{transjective component}. A full subquiver $S$ of $C^{\rm t}$ is called a \emph{section} (see e.g.\ \cite[Definition VIII.1.2]{ASS}) if
\begin{enumerate}
\item[$\bullet$] For each $X\in C^{\rm t}_0$, there exists a unique $i\in\!Z$ satisfying $\tau^iX\in S_0$.
\item[$\bullet$] If there exists a path $X_1\to X_2\to\cdots\to X_\ell$ in $C^{\rm t}$ satisfying $X_1,X_\ell\in S_0$, then $X_i\in S_0$ for each $i$.
\end{enumerate}
We also call $\bigoplus_{X\in S_0}X$ a \emph{section} of $C^{\rm t}$. It is known that each section of $C^{\rm t}$ is a tilting complex of $H$ since it is an iterated reflection of $Q$ at sink or source.

\begin{proof}[Proof of Lemma \ref{lift}]
Without loss of generality, assume that $H$ is ring-indecomposable. We denote by $Q$ the valued quiver of $H$.
We will show the following claim.
\begin{enumerate}
\item[$\bullet$] There exist $i\in\Z$ and a section $T$ of $C^{\rm t}[i]$ such that $\pi T=F\pi H$.
\end{enumerate}
Then $T\in\per H$ is a tilting complex of $H$ such that $\End_{\per H}(T)\simeq H$. By \cite{Ke98}, $T$ lifts to a two-sided tilting complex $T'$ of $H$, and we have an autoequivalence $-\Lotimes_HT':\per H\to\per H$, which gives the desired functor $G$.

Let ${\rm AR}(\C_d(H))$ be the Auslander-Reiten quivers of $\C_d(H)$. Then we have a covering
\begin{equation}\label{pi AR 0}
\pi:{\rm AR}(\per H)\to{\rm AR}(\C_d(H))
\end{equation}
of valued translation quivers, which gives an isomorphism of valued translation quivers
\begin{equation}\label{pi AR}
\pi:{\rm AR}(\per H)/\nu_d\simeq{\rm AR}(\C_d(H)).
\end{equation}
We divide into two cases.

(i) Assume that $H$ is representation-infinite. Then the connected components of ${\rm AR}(\per H)$ are
\begin{enumerate}
\item[$\bullet$] $C^{\rm t}[i]$ ($i\in\Z$), which has the form $\Z Q$,
\item[$\bullet$] the shifts of the regular components, which have the form $\Z A_\infty/\langle\tau^\ell\rangle$ or $\Z A_\infty$.
\end{enumerate}
Since $d\ge2$, $\nu_d$ permutes these connected components, and \eqref{pi AR 0} restricts to an isomorphism $\pi:C\simeq\pi C$ for each connected component $C$ of ${\rm AR}(\per H)$. By \eqref{pi AR}, the connected components of ${\rm AR}(\C_d(H))$ are
\begin{enumerate}
\item[$\bullet$] $\pi C^{\rm t}[i]$ ($0\le i\le d-2$), which has the form $\Z Q$,
\item[$\bullet$] the shifts of the regular components, which have the form $\Z A_\infty/\langle\tau^\ell\rangle$ or $\Z A_\infty$.
\end{enumerate}
In particular, $\pi C^{\rm t}[i]$ is not isomorphic to the shifts of the regular components. Thus there exists $0\le i\le d-2$ satisfying $F\pi C^{\rm t}=\pi C^{\rm t}[i]$. Since $\pi H$ is a section of $\pi C^{\rm t}$, then  $F\pi H$ is a section of $\pi C^{\rm t}[i]$. Thus the preimage $T$ of $F\pi H$ in $C^{\rm t}[i]$ is a section and satisfies $\pi T=F\pi H$, as desired.

%Now the same argument as in the case (i) gives the desired section $T$ of $C^{\rm t}[i]$.
%Thus the automorphism $F$ of $\C_d(A)$ gives an automorphism of the valued translation quiver ${\rm AR}(\C_d(A))$, which restricts to an isomorphism $F:\pi C\simeq\pi C[i]$ for some $0\le i\le d-2$. Thus $F\pi A$ is a section of the valued translation quiver $\pi C[i]$. Take $T\in\add C[i]$ satisfying $\pi T\simeq F\pi A$. Then $T$ is also a section of $C[i]$, and therefore the tilting complex of $A$. Since $\End_{\per A}(T)\simeq A$ holds clearly, we obtain the desired claim.

(ii) Assume that $H$ is representation-finite. Then ${\rm AR}(\per H)=C^{\rm t}\simeq\Z Q$, and we denote the covering map \eqref{pi AR 0} by $\pi\colon\Z Q\to \Z Q/\nu_d$. Then by \cite[p.207, Satz(b)]{Rie} we can lift the isomorphism $F$ on $\Z Q/\nu_d$ to $G$ as in the diagram below.
\[ \xymatrix@!R=5mm{
	\Z Q \ar@{-->}[r]^-G\ar[d]_-\pi&\Z Q\ar[d]^-\pi\\
	\Z Q/\nu_d\ar[r]^-F& \Z Q/\nu_d } \]
Now, $GH\in\Z Q$ is a desired section of $\Z Q$.
%Note that \eqref{pi AR} is a local isomorphism, that is, for each $x\in\Z Q$, the map $\pi$ induces a bijection $\{\al\in(\Z Q)_1\mid s(\al)=x\}\xsimeq\{ \be\in(\Z Q)_1\mid s(\be)=\pi(x)\}$. It follows that once we pick a preimage of the source of $F\pi A\in\Z Q/\nu_d$, one can inductively construct a lift of $\pi A$, which forms a section of $\Z Q$ since $Q$ is a tree. Thus the assertion follows. %\comment{Write more seriously}
%Since $\pi A$ is a section of $\pi C^{\rm t}$, then  $F\pi A$ is a section of $\pi C^{\rm t}[i]$. Using Thus the preimage $T$ of $F\pi A$ in $C^{\rm t}[i]$ is a section and satisfies $\pi T=F\pi A$, as desired.
\end{proof}

%\begin{Hope}\label{hope}
%Let $T\in\C_d(A)$ be a $d$-cluster tilting object such that $\End_{\C_d(A)}(T)=A$.
%\begin{enumerate}
%	\item We have $\Hom_{\C_d(A)}(T,T[-i])=0$ for $1\leq i\leq d-2$.
%	\item There exists a tilting complex $M\in\D^b(A)$ with $\End_{\D^b(A)}(M)=A$ such that its image under the canonical functor $\D^b(A)\to\C_d(A)$ is $T$.
%\end{enumerate}
%%By \cite{KRac}, (1) implies that there is an autoequivalence $\C_d(A)\simeq\C_d(A)$ sending $A$ to $T$.
%\end{Hope}

%\new{
%Note that (\ref{qeq}) is a statement about dg categories while (\ref{cda}) is a one about triangulated categories. This indicates, as it is, that the proof of the implication (\ref{cda})$\Rightarrow$(\ref{qeq}) depends on the uniqueness of enhancements of $\C_d(H)$. Let us recall the relevant notion.
%}

The next one is about lifting an autoequivalence of the cluster category to their enhancements.
We denote by $\Ga_d(H)$ the dg algebra $\Ga$ in \eqref{induction} for $\Pi=\Pi_{d+1}(H)$, so that we have an equivalence $\per\Ga_d(H)\simeq\C_d(H)$.
\begin{Lem}\label{lift2}
Let $H$ be a finite dimensional hereditary algebra and $d\geq2$. Let $F\colon\C_d(H)\to\C_d(H)$ be a triangle autoequivalence.
Then there exists an invertible bimodule $N$ over $\Ga_d(H)$ such that $-\lotimes_{\Ga_d(H)}N\colon\per\Ga_d(H)\to\per\Ga_d(H)$ identifies with $F$.
%Therefore, we have a quasi-isomorphism $\Ga_d(H)\xsimeq\REnd_{\C}(FH)$.
\end{Lem}
\begin{proof}
By Lemma \ref{lift} we may lift an autoequivalence $F$ of $\C_d(H)$ to a one $G$ on $\D^b(H)$. It is given by a two-sided tilting complex $M$ over $H$, which we view as an invertible dg bimodule over $H$. By \cite[Proposition 4.2]{Ke11} (see also Lemma \ref{theta theta'}), there exists an invertible bimodule, denoted $\Pi_{d+1}(M)$, over $\Pi_{d+1}(H)$. Since any autoequivalence on $\per\Pi_{d+1}(H)$ restricts to a one on $\pvd\Pi_{d+1}(H)$, the bimodule $\Pi_{d+1}(M)$ induces an automorphism of the dg quotient $\Pi_{d+1}(H)/\pvd_\dg\!\Pi_{d+1}(H)\simeq\Ga_d(H)$, which is given by the bimodule $N:=\Ga_d(H)\lotimes_{\Pi_{d+1}(H)}\Pi_{d+1}(M)\lotimes_{\Pi_{d+1}(H)}\Ga_d(H)$. Now, the commutative diagram
	\[ \xymatrix@C=25mm{
		\D^b(H)\ar[r]^-{-\lotimes_H\Pi_{d+1}(H)}\ar[d]_-{G=-\lotimes_HM}&\per\Pi_{d+1}(H)\ar[d]_-{-\lotimes_{\Pi_{d+1}(H)}\Pi_{d+1}(M)}\ar[r]^-{-\lotimes_{\Pi_{d+1}(H)}\Ga_d(H)}&\per\Ga_d(H)\ar[d]^-{-\lotimes_{\Ga_d(H)}N}\\
		\D^b(H)\ar[r]^-{-\lotimes_H\Pi_{d+1}(H)}&\per\Pi_{d+1}(H)\ar[r]^-{-\lotimes_{\Pi_{d+1}(H)}\Ga_d(H)}&\per\Ga_d(H) } \]	
shows that the rightmost vertical map is $F$.
\end{proof}

%Notice that the condition (\ref{cda}) is a one for triangulated categories, while the condition (\ref{qeq}) is a one at the dg level. 
We also need the following uniqueness result of enhancements of cluster categories.
%Recall that a triangulated category $\T$ has a unique enhancement if for any pretriangulated dg categories $\A$ and $\B$ such that $H^0\A\simeq\T\simeq H^0\B$ as triangulated categories, there is an invertible 
\begin{Prop}[{\cite{M,ha4}}]\label{uniqueness}
Let $H$ be a hereditary algebra over a perfect field $k$. Then for each $d\geq2$, the triangulated category $\C_d(H)$ has a unique enhancement, that is, for any dg categories $\A$ and $\B$ with $\per\A\simeq\C_d(H)\simeq\per\B$, there is an invertible $(\A,\B)$-bimodule giving a Morita equivalence $\A\simeq\B$.
\end{Prop}

\begin{proof}
    If $H$ is representation finite, then the result follows from \cite{M}, and if $H$ is representation infinite, it follows from \cite[Theorem 4.19]{ha4}.
\end{proof}

\begin{proof}[Proof of Theorem \ref{hereditary}]
	(\ref{qeq})$\Rightarrow$(\ref{m})  This is Proposition \ref{Pi of H is mild}.\\
	(\ref{m})$\Rightarrow$(\ref{van})  This follows from Lemma \ref{ext=0}.\\
	(\ref{van})$\Rightarrow$(\ref{cda}) This is \cite{KRac}.\\
	(\ref{cda})$\Rightarrow$(\ref{qeq})
	Let $\Ga$ (resp. $\Ga_d(A)$) be the canonical localization of $\Pi$ (resp. $\Pi_{d+1}(A)$) as in \eqref{induction}, so that we have a morphism $\Pi\to\Ga$ (resp. $\Pi_{d+1}(A)\to\Ga_d(A)$) such that the induction functor identifies with $\per\Pi\to\C(\Pi)=\per\Ga$ (resp. $\per\Pi_{d+1}(A)\to\C_d(A)=\per\Ga_d(A)$).
	
%\old{If $\Pi$ is mild, then by Lemma \ref{ext=0} and \cite{KRac} we have a triangle equivalence $\C(\Pi)\xsimeq\C_d(A)$ taking $\Pi$ to $A$. Then}
Since $A$ is a hereditary algebra over a perfect field, Proposition \ref{uniqueness} shows that there is an invertible $(\Ga,\Ga_d(A))$-bimodule $X$ giving an equivalence $-\lotimes_\Ga X\colon\per\Ga\to\per\G_d(A)$.
%We obtain a diagram consisting of triangle equivalences
%	\[ \xymatrix{
%		\C(\Pi)\ar@{=}[r]^-{\eqref{induction}}\ar@{-}[d]^-\rsimeq&\per\Ga\ar[d]_-{-\lotimes_\Ga X}^-\rsimeq\\
%		\C_d(A)\ar@{=}[r]^-{\eqref{induction}}&\per\Ga_d(A),} \]
%	Note that we do not know that this diagram is commutative (since we do not know if one can choose $X$ so as to lift the given triangle equivalence), but 
We get triangle equivalences
    \[ \xymatrix@R=2mm{
    \C_d(A)\ar[r]&\C(\Pi)\ar[r]^{\eqref{induction}}&\per\Ga\ar[r]^-{-\lotimes_\Ga X}&\per\Ga_d(A)\ar[r]^-{\eqref{induction}}& \C_d(A) \\
    \quad A\quad \ar@{|->}[r]& \quad \Pi \quad \ar@{|->}[r]& \quad \Ga \quad \ar@{|->}[r]& \quad X\quad \ar@{|->}[r]& \quad X\quad } \]
    where the first functor is given by the assumption. %and we denote by $T$ the image of $A$ under this functor which equals the preimage of $X\in\per\Ga_d(A)$.    
    Now by Lemma \ref{lift2}, this triangle autoequivalence of $\C_d(A)$ can be lifted to an invertible dg bimodule over $\Ga_d(A)$, and we denote its inverse by $N$. Consequently we obtain equivalences
    \[ \xymatrix{ \per\G\ar[rr]^-{-\lotimes_\G X}&&\per\G_d(A)\ar[rr]^-{-\lotimes_{\G_d(A)}N}&&\per\G_d(A) } \]
    taking $\G$ to $\G_d(A)$. Thus we obtain a quasi-equivalence $-\lotimes_\G Y:\per \Ga\to\per\G_d(A)$ where $Y:=X\lotimes_{\G_d(A)}N$. Truncating quasi-isomorphisms $\G\to\REnd_{\G_d(A)}(Y)\leftarrow\G_d(A)$, we obtain quasi-isomorphisms $\Pi\to\REnd_{\G_d(A)}(Y)^{\le0}\leftarrow\Pi_{d+1}(A)$ by Lemma \ref{leq0}, and therefore a quasi-equivalence $\Pi\to\Pi_{d+1}(A)$.
\end{proof}

%We end this subsection with the following question.
%\begin{Qs}
%Can we replace (\ref{cda}) in Theorem \ref{hereditary} by the following weaker version?
%\begin{itemize}
%\item[(\ref{cda}$^\prime$)]  There is a triangle equivalence $\C(\Pi)\simeq\C_d(A)$.
%\end{itemize}
%\end{Qs}

\subsection{Calabi-Yau completions}
We have seen in the previous subsection that the Calabi-Yau completions of finite dimensional hereditary algebras are $\F$-liftable. We pose the following converse of Theorem \ref{Pi of H is mild} which shows the `if' part.
\begin{Cj}
Let $d\geq3$ and $A$ be a finite dimensional algebra with $\gldim A\leq d$ and which is $\nu_d$-finite. Then the $(d+1)$-Calabi-Yau completion $\Pi_{d+1}(A)$ is $\F$-liftable if and only if $A$ is hereditary.
\end{Cj}
We give a partial answer to this problem. 
Part (2) of the following result shows that if $\gldim A=2$ then $\Pi_{d+1}(A)$ is not $\F$-liftable.
\begin{Thm}\label{gl2}
Let $d\geq3$, let $A$ be a finite dimensional algebra with $\gldim A<d$, and let $\Pi$ be the $(d+1)$-Calabi-Yau completion of $A$. Assume one of the following holds.
\begin{enumerate}
	\item $\gd A=2$.
	\item There is a simple $A$ module $S$ such that $\Ext^2_A(S,S)\neq0$.
\end{enumerate}
Then $\Pi$ is not $\F$-liftable.
\end{Thm}

Let $d\geq0$ and $\Pi$ be a $(d+1)$-CY dg algebra.
Let us compute the self-extensions of the truncation $\Pi^{\leq-l}$.
\begin{Lem}
Let $l>0$ and suppose $H^{-i}\Pi=0$ for $1\leq i\leq l-1$. Then there is an isomorphism $\Hom_{\D(\Pi)}(\Pi^{\leq -l},\Pi^{\leq -l}[d-l])=D\Hom_{H^0\Pi}(H^{-l}\Pi,H^0\Pi)$.
\end{Lem}
\begin{proof}
	By assumption we have a triangle
	\[ \xymatrix{ \Pi^{\leq -l}\ar[r]&\Pi\ar[r]&H^0\Pi\ar[r]& \Pi^{\leq-l}[1] } \]
	in $\D(\Pi)$. Applying $\Hom_{\D(\Pi)}(-,\Pi^{\leq-l})$ gives an exact sequence
	\[ \xymatrix@R=0mm{
		\Hom_{\D(\Pi)}(\Pi,\Pi^{\leq -l}[d-l])\ar[r]&\Hom_{\D(\Pi)}(\Pi^{\leq-l},\Pi^{\leq-l}[d-l])\ar[r]&\Hom_{\D(\Pi)}(H^0\Pi,\Pi^{\leq-l}[d-l+1])\ar[r]&\\
		\Hom_{\D(\Pi)}(\Pi,\Pi^{\leq-l}[d-l+1]), }\]
	in which the two end terms are $0$ (by $d>0$). Also by relative Serre duality we have $\Hom_{\D(\Pi)}(H^0\Pi,\Pi^{\leq-l}[d-l+1])=D\Hom_{\D(\Pi)}(\Pi^{\leq-l},H^0\Pi[l])$, which is isomorphic to $D\Hom_{\D(\Pi)}(H^{-l}\Pi,H^0\Pi)=\Hom_{H^0\Pi}(H^{-l}\Pi,H^0\Pi)$.
\end{proof}

\begin{Prop}\label{nec}
Let $d\geq0$, $\Pi$ a $(d+1)$-CY dg algebra, and $l>0$ the minimum positive integer such that $H^{-l}\Pi\neq0$. If $\Pi$ is $\F$-liftable and $l<d$, then we have $\Hom_{H^0\Pi}(H^{-l}\Pi,H^0\Pi)=0$.
\end{Prop}
\begin{proof}
	Since $\Pi$ is $\F$-liftable, the lift $\Pi^{\leq-l}[-l]$ of the $d$-cluster tilting object $\Pi[-l]\in\C(\Pi)$ has to be a silting object in $\per\Pi$. We therefore obtain the conlusion by the preceding lemma.
\end{proof}
We note that one can also prove $\Hom_{H^0\Pi}(H^{-2}\Pi,H^{-1}\Pi)=0$. Indeed, use $\Hom_{\D(\Pi)}(\Pi^{\leq-2},\Pi^{\leq-2}[2])=0$ and compute this by the triangle $\Pi^{\leq-2}\to\Pi\to\Pi^{[-1,0]}$.

We prepare the following general fact for arbitrary rings.
\begin{Lem}\label{ext01}
	Let $R$ be a ring, $e\in R$ an idempotent, and $M\in\Mod R/(e)$. Then we have $\Ext^i_R(Re\otimes_{eRe}eR,M)=0$ for $i=0,1$.
\end{Lem}
\begin{proof}
	Let $0\to M\to I^0\to I^1\to I^2\to\cdots$ be an injective resolution of $M$ in $\Mod R$. Applying $\Hom_R(Re\otimes_{eRe}eR,-)$ gives a complex
	\[ \xymatrix{ 0\ar[r]&\Hom_R(Re\otimes_{eRe}eR,I^0)\ar[r]&\Hom_R(Re\otimes_{eRe}eR,I^1)\ar[r]&\Hom_R(Re\otimes_{eRe}eR,I^2)\ar[r]&\cdots} \]
	which computes $\Ext^i_R(Re\otimes_{eRe}eR,M)$. By adjunction, it is isomorphic to
	\[ \xymatrix{ 0\ar[r]&\Hom_{eRe}(Re,I^0e)\ar[r]&\Hom_{eRe}(Re,I^1e)\ar[r]&\Hom_{eRe}(Re,I^2e)\ar[r]&\cdots}. \]
	Now, since $M$ is annihilated by $e$ the sequence $0\to I^0e\to I^1e\to I^2e$ is exact, thus the above complex is acyclic at degree $\leq1$, which gives our assertion.
\end{proof}
	
\begin{Lem}\label{ext2}
Let $d\geq1$ and let $A$ be a finite dimensional algebra with $\gldim A<d$, and let $\Pi$ be the $(d+1)$-CY completion of $A$. Suppose that $\Pi$ is $\F$-liftable. Then for any idempotent $e\in A$, we have $\Ext^2_A(A/(e),A/(e))=0$.
\end{Lem}
\begin{proof}
	%We consider the dg quotient of $\Pi$ by $e$. By \cite[Theorem 4.6(b)]{Ke11}, this is quasi-isomorphic the $(d+1)$-CY completion of the dg quotient $\sA$ of $A$ by $e$. 
	Consider the dg quotient of $\sA$ of $A$ by $e$. By the triangle $Ae\lotimes_{eAe}eA\to A\to \sA$, there is an exact sequence
	\[ \xymatrix{ 0\ar[r]&H^{-1}\sA\ar[r]&Ae\otimes_{eAe}eA\ar[r]&A\ar[r]&H^0\sA\ar[r]&0 }, \]
	in which the image of the middle map is $AeA$. Applying $\Hom_A(-,H^0\sA)$, the connecting morphisms give isomorphisms
	\[ \xymatrix{\Hom_A(H^{-1}\sA,H^0\sA)\ar[r]^-{\simeq}&\Ext^1_A(AeA,H^0\sA)\ar[r]^-{\simeq}&\Ext^2_A(H^0\sA,H^0\sA) }, \]
	where the left isomorphism follows from Lemma \ref{ext01}. We claim that $\Hom_{H^0\sA}(H^{-1}\sA,H^0\sA)=0$, which gives the desired result by $H^0\sA=A/(e)$.
	
	Let $\sPi$ be the dg quotient of $\Pi$ by $e$. By \cite[Theorem 4.6(b)]{Ke11}, this is quasi-isomorphic the $(d+1)$-CY completion of $\sA$.	If $\Pi$ is $\F$-liftable, then so is $\sPi$ by Theorem \ref{redm}, which forces $\Hom_{H^0\sPi}(H^{-1}\sPi,H^0\sPi)=0$ by Proposition \ref{nec}.
	Now observe that $H^0\sPi=H^0\sA$. Indeed, by our assumption that $\gldim A<d$, the bimodule $\RHom_A(DA,A)[d]$ is concentrated in strictly negative degrees. Then the same holds for $\sA$ since $A\to\sA$ is a localization of dg algebras, and so we have $\RHom_{\sA}(D\sA,\sA)=\sA\lotimes_A\RHom_A(DA,A)\lotimes_A\sA$ by \cite[Proposition 3.10(c)]{Ke11}. It follows that the $0$-th cohomology of $\sPi=\Pi_{d+1}(\sA)$ is $H^0\sA$.
	Finally, since the $H^0\sA$-module $H^{-1}\sPi$ has $H^{-1}\sA$ as a direct summand, we see from $\Hom_{H^0\sPi}(H^{-1}\sPi,H^0\sPi)=0$ that $\Hom_{H^0\sA}(H^{-1}\sA,H^0\sA)=0$. This proves the claim and therefore the result.
\end{proof}

We note another general lemma on finite dimensional algebras of global dimension $2$.
\begin{Lem}\label{2}
	Let $A$ be a finite dimensional algebra with $\gd A=2$. Then there exists an idempotent $e\in A$ such that $\Ext^2_A(A/(e),A/(e))\neq0$ and $1-e$ is a sum of at most two orthogonal primitive idempotents.
\end{Lem}
\begin{proof}
	Without loss of generality we assume that $A$ is basic. Since $\gd A=2$ there exist simple $A$-modules $S_i$ and $S_j$ such that $\Ext^2_A(S_i,S_j)\neq0$. We denote by $e_i$ and $e_j$ the corresponding idempotents of $A$.
	
	If $S_i\simeq S_j$ then we may take $e=1-e_i$. Indeed, we have $A/(e)\simeq S_i$ by the no loop theorem \cite{L,Ig}.
	
	In the rest we assume that $S_i\not\simeq S_j$, and claim that $e=1-e_i-e_j$ is a desired idempotent. Put $B=A/(e)$ and suppose that $\Ext^2_A(B,B)=0$. Note that $B$ is an algebra with two simple module $S_i$ and $S_j$. Applying $\Hom_A(B,-)$ the surjection $B\twoheadrightarrow S_j$ yields a surjection $\Ext^2_A(B,B)\twoheadrightarrow\Ext^2_A(B,S_j)$ by $\gd A\leq2$, so we get $\Ext^2_A(B,S_j)=0$.
	Consider the exact sequence
	\[ \xymatrix@R=4mm{	0\ar[r]& T_j\ar[r]& B\ar[r]& U_j\ar[r]& 0 }, \]
	where $T_j$ is the maximal among the submodules of $B$ whose composition factors are $S_j$.
	Applying $\Hom_A(-,S_j)$ we get an exact sequence $\Ext^1_A(T_j,S_j)\to\Ext^2_A(U_j,S_j)\to\Ext^2_A(B,S_j)$, in which the left end term is $0$ by the no loop theorem, and the right end term is $0$ by the preceeding claim. Thus we get $\Ext^2_A(U_j,S_j)=0$. 
	On the other hand, we have $0\neq\soc U_j\in\add S_i$ by the choice of $T_j$, thus a surjection $\Ext^2_A(U_j,S_j)\twoheadrightarrow\Ext^2_A(S_i^\oplus,S_j)$. This yields $\Ext^2_A(S_i,S_j)=0$, a contradiction. Therefore we must have $\Ext^2_A(B,B)\neq0$.
\end{proof}

\begin{proof}[Proof of Theorem \ref{gl2}]
	Without loss of generality we assume that $A$ is basic. %We have to show that if $\gd A=2$, then $A$ is not $\F$-liftable. There exist simple $A$-modules $S_i$ and $S_j$ such that $\Ext^2_A(S_i,S_j)\neq0$.
	
	(1)  We show that if $\gd A=2$ then $\Pi=\Pi_{d+1}(A)$ is not $\F$-liftable.If $\Pi$ is $\F$-liftable, then Lemma \ref{ext2} shows $\Ext^2_A(A/(e),A/(e))=0$ for any idempotent $e\in A$. But this contradicts Lemma \ref{2}, so $\Pi$ connot be $\F$-liftable.

	(2)  Let $f$ be the idempotent of $A$ corresponding to $S$ and put $e=1-f$. If $\Pi$ is $\F$-liftable, then by Lemma \ref{ext2} we have $\Ext^2_A(A/(e),A/(e))=0$. Now note that $A/(e)\simeq S$ as $A$-modules. Indeed, the algebra $A/(e)$ is local with unique simple module $S$, and $\Ext^1_{A/(e)}(S,S)=\Ext^1_A(S,S)=0$ by the no loop theorem \cite{L,Ig}, thus $A/(e)$ has to be a simple algebra. We therefore obtain $\Ext^2_A(S,S)=0$, a contradiction, and we conclude that $\Pi$ is not $\F$-liftable.
\end{proof}

\begin{Ex}\label{tilted}
Let $A$ be the tilted algebra of type $A_3$ which is presented by the following quiver with relations.
\[ \xymatrix{ 1\ar[r]^-a&2\ar[r]^-b&3, \quad ba=0 } \]
Let $d\geq3$ and let $\Pi=\Pi_{d+1}(A)$ be the $(d+1)$-Calabi-Yau completion of $A$. The algebra $A$ has global dimension $2$, so by Theorem \ref{gl2}(2), we see that {\it $\Pi$ is not $\F$-liftable}.

On the other hand, we claim that {\it $\Pi$ is liftable}. Indeed, our algebra $A$ is derived equivalent to the path algebra $kQ$ of type $A_3$, and hence the Calabi-Yau completions $\Pi$ and $\Pi_{d+1}(kQ)$ are dg Morita equivalent by \cite[Proposition 4.2]{Ke11}. By Theorem \ref{hereditary}, we have that $\Pi$ is $\F$-liftable, in particular liftable. We conclude by Proposition \ref{BO} that $\Pi=\Pi_{d+1}(A)$ is also liftable.
\end{Ex}
\section{Examples of silting-CT correspondence}\label{k}
We illustrate the correspondences we gave in Theorem \ref{H0=k} by examples.
\subsection{Polynomial dg algebras}
Let $\Pi=k[x_0]\otimes_k\cdots\otimes_kk[x_l]$ be the tensor product of dg polynomial algebras with trivial differentials and the grading $\deg x_i=-a_i$. Since each $k[x_i]$ is an $(a_i+1)$-Calabi-Yau dg algebra, the tensor product $\Pi$ is an $(l+a+1)$-Calabi-Yau dg algebra for $a=\sum_{i=0}^la_i$. Note that $\Pi$ is a skew polynomial ring, for example, $x_ix_j=(-1)^{a_ia_j}x_jx_i$, compare \cite{MGYC,ha3}.

We study the correspondence between the silting objects in $\per\Pi$ (or in the fundamental domain $\F=(\add\Pi)\ast\cdots\ast(\add\Pi[l+a-1])$) and cluster tilting objects in $\C(\Pi)$, that is, the horizontal maps below in \eqref{sct} and \eqref{fsct}.
\[ \xymatrix@R=3mm{
	\silt\Pi\ar[r]&(l+a)\text{-}\ctilt\C(\Pi)\\
	\silt\Pi\cap\F\ar[r]\ar@{}[u]|-\vsubset&(l+a)\text{-}\ctilt\C(\Pi)\ar@{=}[u] } \]
Since $\Pi$ is local, we have
\[ \silt\Pi=\{\Pi[i]\mid i\in\Z\}\simeq\Z \]
by \cite[Theorem 2.26]{AI}. It follows that
\[ \silt\Pi\cap\F=\{\Pi[i]\mid 0\leq i\leq l+a-1\}\simeq\{0,1,\ldots,l+a-1\}. \]

\begin{Thm}\label{poly}
Let $\Pi=k[x_0]\otimes_k\cdots\otimes_kk[x_l]$ with $\deg x_i=-a_i<0$ be the dg algebra with trivial differentials. We put $a=\sum_{i=0}^la_i$.
\begin{enumerate}
	\item\label{Fliftable} $l=0$ if and only if the map $\silt\Pi\cap\F\to (l+a)\text{-}\ctilt\C(\Pi)$ is bijective, that is, $\Pi$ is $\F$-liftable.
	%\item If $n=0$, then $\Pi$ is $\F$-liftable.
	\item\label{bij} If $l=1$, then the map $\silt\Pi\to (l+a)\text{-}\ctilt\C(\Pi)$ is bijective. Thus $\Pi$ is liftable. %$\Pi$ is not $\F$-liftable but liftable.
	\item\label{inj} If $l>0$, then the map $\silt\Pi\to (l+a)\text{-}\ctilt\C(\Pi)$ is injective.
	\end{enumerate}
\end{Thm}
We first prove rather easy statements (\ref{Fliftable}) and (\ref{inj}).
\begin{proof}[Proof of (\ref{Fliftable})(\ref{inj})]
	(\ref{Fliftable}) is Proposition \ref{H0=k}(\ref{qis})$\Rightarrow$(\ref{m}).
	
	(\ref{inj})  We have to show that the $\Pi[i]$'s are mutually non-isomorphic in $\C(\Pi)$. It is enough to show that $\Pi$ and $\Pi[-i]$ for $i>0$ are non-isomorphic. Suppose to the contraty that $\Pi\simeq\Pi[-i]$ for some $i>0$. Recall from \eqref{negative H} that we have $\Hom_{\D(\Pi)}(X,Y)\xsimeq\Hom_{\C(\Pi)}(X,Y)$ for $X\in\Pi\ast\Pi[1]\ast\cdots$ and $Y\in\cdots\ast\Pi[l+a-2]\ast\Pi[l+a-1]$. It follows that we have $\Hom_{\C(\Pi)}(\Pi,\Pi)\ysimeq\Hom_{\D(\Pi)}(\Pi,\Pi)=k$, while $\Hom_{\C(\Pi)}(\Pi,\Pi[-j])\ysimeq\Hom_{\D(\Pi)}(\Pi,\Pi[-j])=H^{-j}\Pi$ for any $j>0$. Therefore, it remains to find $i$ such that $\dim_kH^{-i}\Pi>1$.
	Now, if $\Pi\simeq\Pi[-i]$ then we have $\Pi\simeq\Pi[-ni]$ for any $n\in\Z$, and in particular for $n=a_0a_1$. We then get two linearly independent elements $x_0^{ia_1}$ and $x_1^{ia_0}$ in $H^{-ia_0a_1}\Pi$.
\end{proof}

The rest of this subsection is devoted to the proof of (\ref{bij}).
It depends essentially on a description of the cluster category $\C(\Pi)$ as an orbit category of the derived category of a finite dimensional hereditary algebra given by the following result.
\begin{Thm}[\cite{ha3}]\label{orbit}
Let $\Pi=k[x_0]\otimes_k\cdots\otimes_kk[x_l]$ with $\deg x_i=-a_i<0$ be a dg algebra with trivial differentials, and put $a=\sum_{i=0}^na_i$.
\begin{enumerate}
%\item The dg algebra $\Pi$ is $(l+a+1)$-Calabi-Yau.
\item Define the finite dimensional algebra $A$ and the $(A,A)$-bimodule $U$ by
\[ A=\begin{pmatrix}\Pi^0&0&\cdots&0\\ \Pi^{-1}&\Pi^0&\cdots&0\\ \vdots&\vdots&\ddots&\vdots\\ \Pi^{-a+1}&\Pi^{-a+2}&\cdots&\Pi^0\end{pmatrix}, \qquad U=\begin{pmatrix}\Pi^{-1}&\Pi^0&\cdots&0\\ \vdots& \vdots&\ddots&\vdots\\	\Pi^{-a+1}&\Pi^{-a+2}&\cdots&\Pi^0 \\
	\Pi^{-a}&\Pi^{-a+1}&\cdots&\Pi^{-1}\end{pmatrix}. \]
Then $A$ is $l$-representation infinite and we have $U^{\lotimes_A a}\simeq\RHom_A(DA,A)[l]$ in $\D(A^e)$.
\item\cite[Corollary 6.2]{ha3} There exists a triangle equivalence
\[ \C(\Pi)\simeq\C^{(1/a)}_{a+l}(A), \]
where the right-hand-side is defined by the $a$-th root $U[1]$ of $\RHom_A(DA,A)[a+l]$.
\item Let $P\in \proj A$ be the projective module given by the first row of $A$. Then the equivalenece in (2) restricts to an equivalence of subcategories 
\[ \add\{\Pi[i]\mid i\in\Z\}\simeq\add\{P\lotimes_AU^{\lotimes_Ai}\mid i\in\Z\}. \]
\end{enumerate}
\end{Thm}
\begin{proof}
We include the proof of (3). Since the equivalence $\C(\Pi)\simeq\C^{(1/a)}_{l+a}(A)$ in (2) takes $\Pi$ to $P$, and the suspensions of these categories are given by $[1]$ and $U^{-1}$, we get the desired assertion.
\end{proof}

Although we have no control of cluster tilting objects in the cluster category $\C(\Pi)$ in general, the description in terms of the finite dimensional algebra $A$ in Theorem \ref{orbit}(2) allows us to give a classification when $l=1$, that is, $A$ is hereditary.

In what follows we set $l=1$, so that $\Pi=k[x_0]\otimes_kk[x_1]$.
To avoid indices we rewrite $\Pi=k[x]\otimes_kk[y]$ with $\deg x=-b$ and $\deg y=-c$ (so $a=b+c$). Let $n$ be the greatest common divisor of $b$ and $c$, and put $p=b/n$, $q=c/n$, and $r=p+q$. In this case, the algebra $A$ in Theorem \ref{orbit} is the direct product of $n$ copies of the algebra $B$ on the left below, and the bimodule $U$ is described as a matrix on the right below along the decomposition $A=B\times\cdots\times B$, with $V$ the bimodule over $B$ defined by the matrix in the middle.
\[ B=\begin{pmatrix}\Pi^0&0&\cdots&0\\ \Pi^{-n}&\Pi^0&\cdots&0\\ \vdots&\vdots&\ddots&\vdots\\ \Pi^{-a+n}&\Pi^{-a+2n}&\cdots&\Pi^0\end{pmatrix}, \quad
V=\begin{pmatrix}\Pi^{-n}&\Pi^0&\cdots&0\\ \vdots& \vdots&\ddots&\vdots\\	\Pi^{-a+n}&\Pi^{-a+2n}&\cdots&\Pi^0 \\ \Pi^{-a}&\Pi^{-a+n}&\cdots&\Pi^{-n}\end{pmatrix}, \quad
U=\begin{pmatrix}
	0 & B&\cdots& 0 \\
	\vdots& \vdots&\ddots&\vdots \\
	0&0&\cdots&B\\
	V&0&\cdots&0
\end{pmatrix}, \]
%\[ V=\begin{pmatrix}\Pi^{-n}&\Pi^0&\cdots&0\\ \vdots& \vdots&\ddots&\vdots\\	\Pi^{-a+n}&\Pi^{-a+2n}&\cdots&\Pi^0 \\ \Pi^{-a}&\Pi^{-a+n}&\cdots&\Pi^{-n}\end{pmatrix}. \]

By Theorem \ref{orbit} we have an equivalence $\C(\Pi)\simeq\C_{b+c+1}^{(1/(b+c))}(A)=\D^b(\mod A)/-\lotimes_AU[1]$, which can be written as
\begin{equation}\label{B}
\C(\Pi)\simeq\C_{a+1}^{(1/r)}(B)=\D^b(\mod B)/-\lotimes_BV[n]
\end{equation}
by \cite[Appendix 1]{ha3}, also \cite[Section 5]{ha4}.

Now, the algebra $B$ is the path algebra of the quiver $Q$ of type $\widetilde{A}_{p+q-1}$, where $Q$ has vertices $\{0,1,\ldots,p+q-1\}$ with arrows $i\to i+p$ for $0\leq i\leq q-1$ (labelled by $x$) and $i\to i+q$ for $0\leq i\leq p-1$ (labelled by $y$) as below (for $b=2n$ and $c=3n$).
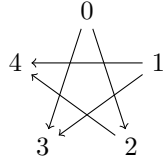
\[	\begin{tikzpicture}
	\def\radius{1cm} 
	\node (0) at (90:\radius)   {$0$};
	\node (-1) at (18:\radius)    {$1$};
	\node (-2) at (-54:\radius)  {$2$};
	\node (-3) at (-126:\radius) {$3$};
	\node (-4) at (162:\radius)  {$4$};
	
	\begin{scope}[arrows={->[scale=4]}]
	\path[->]
	(0) edge (-2)
	(-2) edge (-4);
	\path[->,font=\small]
	(-1) edge (-3);
	\path[->,font=\small]
	(0) edge (-3);
	\path[->,font=\small]
	(-1) edge (-4);
	\end{scope}
\end{tikzpicture}
\]
We will use the following information on the shape of the quiver $Q$.
\begin{Lem}\label{clock}
	The quiver $Q$ has $p$ clockwise oriented arrows and $q$ counter-clockwise oriented arrows.
\end{Lem}
\begin{proof}
	Since $p$ and $q$ are relatively prime, the edges $\{ i\text{ --- }i+p \mid i\in\Z/(p+q)\Z\}$ with suitable orientations are precisely the arrows in $Q$, under the identification $Q_0=\Z/(p+q)\Z$. Moreover, the edges are oriented in the way that the source is smaller than the target if we take the representative $\{0,1,\ldots,p+q-1\}$ of $\Z/(p+q)\Z$. It follows that the clockwise arrows are precisely the ones starting at $0,1,\ldots, p-1$.
\end{proof}

The description \eqref{B} of $\C(\Pi)$ as the orbit category of a hereditary algebra gives the following classification of objects in $\C(\Pi)$. We denote by $\ind\A$ the set of isomorphism classes of indecomposable objects in a Krull-Schmidt category $\A$.
\begin{Prop}\label{F}
There is a bijection $\ind\C(\Pi)\simeq(\bigsqcup_{i=0}^{n-1}\ind(\mod B)[i])\sqcup \{e_0B[n]\}$.
\end{Prop}
In order to classify cluster tilting objects in $\C(\Pi)$, we are thus lead to classify rigid objects in $\mod B$. Over a hereditary algebra, any indecomposable object is either preprojective, preinjective, or regular. Certainly, any preprojective or preinjective indecomposable object is rigid. To see which regular module is rigid, let us recall the desrciption of the Auslander-Reiten quiver of the extended Dynkin quiver of type $\widetilde{A}$.

Recall that a {\it tube of rank $t$} is a component of the Auslander-Reiten quiver of an algebra which is isomorphic to $\Z A_{\infty}/(\tau^t)$. We say that a tube is {\it homogeneous} if $t=1$. 
\begin{Prop}[{see \cite[XIII. 2.1]{SS2}}]
Let $Q$ be an extended Dynkin quiver of type $\widetilde{A}$ with $p$ clockwise oriented arrows and $q$ counter-clockwise oriented arrows. Then the regular component of $\mod kQ$ has two non-homogeneous tubes, which have rank $p$ and $q$.
\end{Prop}

The following computations are crucial. The first one is easy.
\begin{Lem}\label{Ext^1}
If $X\in\mod B$ is an indecomposable object in a homogeneous tube, then $\Hom_{\C(\Pi)}(X,X[1])\neq0$.
\end{Lem}
\begin{proof}
	Since $\tau X\simeq X$ in this case, we have $0\neq\Ext^1_B(X,\tau X)=\Ext^1_B(X,X)\hookrightarrow\Hom_{\C(\Pi)}(X,X[1])$.
\end{proof}

\begin{Lem}\label{Ext^a}
Let $X\in\mod B$ be an indecomposable object in the tube $C$ of rank $p>1$.
\begin{enumerate}
\item The $r$-th root $F:=-\lotimes_BV$ of $\tau^{-1}$ preserves $C$.
\item $F$ acts on objects in $C$ as $\tau^{-r^\prime}$, where $r^\prime$ is the inverse of $r$ in $\Z/p\Z$.
\item We have $\Hom_{\C(\Pi)}(X,X[nq+1])\neq0$.
\end{enumerate}
\end{Lem}
\begin{proof}
	(1) Since $p$ and $q$ are relatively prime, $C$ is the only rank $p$ tube in the Auslander-Reiten quiver of $\mod B$, and hence the only one up to shift in the Auslander-Reiten quiver of $\D^b(\mod B)$.
	On the other hand, since $F$ is an autoequivalence of the derived category $\D^b(\mod B)$, it acts on its Auslander-Reiten components. Moreover, since $F$ takes each indecomposable module except the simple injective module $D(Be_{r-1})$ to a module, we conclude that $F$ must preserve $C$.
	
    (2) Note that the possible action of $F$ on the rank $p$ tube is a power of $\tau$. It follows from $F^r=\tau^{-1}$ that $F$ must act on objects in $C$ as $\tau^{-r^\prime}$.
    
    %is the inverse of $r$ in the ring $\Z/p\Z$. 
    (3) In $\Z/p\Z$, we have $qr'=rr'=1$. Thus $F^q=\tau^{-qr'}=\tau^{-1}$ on objects in $C$. Then we deduce that $0\neq\Ext^1_B(X,\tau X)=\Ext^1_B(X,F^{-q}X)\hookrightarrow\Hom_{\C(\Pi)}(X,F^{-q}X[1])\simeq\Hom_{\C(\Pi)}(X,X[nq+1])$ by \eqref{B}.
\end{proof}

As a final preparation, we record the following reformulation of Theorem \ref{orbit}(3) in terms of the Auslander-Reiten quivers. %Recall that a {\it transjective component} of $\D^b(\mod H)$ for a hereditary algebra $H$ is a connected component of its Auslander-Reiten quiver containing a shift of a projective module. We 
\begin{Lem}\label{transjective}
The objects in the transjective components of $\C(\Pi)$ are precisely the shifts of $\Pi$.
\end{Lem}
\begin{proof}
By Proposition \ref{F}, the Auslander-Reiten quiver of $\C(\Pi)$ has $n$ transjective components, which we denote by $C_0,\ldots,C_{n-1}$. Each $C_i$ contains the image of $\Pi[-i], \Pi[-i-n], \ldots,\Pi[-i-a+n]$ which corresponds to the image of the summands of $B[i]$ under the equivalence $\C(\Pi)\simeq\D^b(\mod B)/-\lotimes_BV[n]$. Noting that the Auslander-Reiten translation is $[n+a-1]$ since $\C(\Pi)$ is $(n+a)$-Calabi-Yau, we conclude that the $\tau$-orbit of the $B[i]$'s are precisely the shifts of $\Pi$.
\end{proof}

%We also need the following result. %which relies on $B$ being hereditary.
%\begin{Lem}\label{indec}
%Any $(a+1)$-cluster tilting object in $\C(\Pi)$ is indecomposable.
%\end{Lem}
%\begin{proof}
%	Consider the orbit functors $\D^b(\mod B)\to\D^b(\mod B)/-\lotimes_BDB[-a-1]\to\D^b(\mod B)/-\lotimes_BV[n]$, where the middle term is the ordinary cluster category $\C_{a+1}(B)$, and the last term is the folded cluster category $\C_{a+1}^{(1/r)}(B)=\C(\Pi)$. The first functor is a $\Z$-covering functor and the second one is $\Z/r\Z$-covering, and both functors are dense. It follows that an $(a+1)$-cluster tilting object $X\in\C(\Pi)$ with $|X|=t$ can be lifted to an $(a+1)$-cluster tilting object $Y\in\C_{a+1}(B)$ with $|Y|=rt$. On the other hand, by Theorem \ref{Pi of H is mild}, every $(a+1)$-cluster tilting object in $\C_{a+1}(B)$ has the same number of summands, which is $|B|=r$. We conclude that one must have $t=1$.
%\end{proof}

We are now ready to prove the remaining assertion of the main theorem \ref{poly} in this section.
\begin{proof}[Proof of Theorem \ref{poly}(\ref{bij})]
	We have to show that the cluster tilting objects in $\C(\Pi)$ are precisely the shifts of $\Pi$. By \eqref{B}, we have a triangle equivalence $\C(\Pi)\simeq\C^{(1/r)}_{a+1}(B)$ for the path algebra $B$ of type $\widetilde{A}_{r-1}$.
	
	We first limit the possibilities of $(a+1)$-rigid indecomposable objects in $\C(\Pi)$. In view of Proposition \ref{F}, it is enough to consider such objects in $\mod B$, but by Lemma \ref{Ext^1} and Lemma \ref{Ext^a} we find that no regular indecomposable module is $(a+1)$-rigid. Therefore an object in $\C(\Pi)$ is $(a+1)$-rigid only if it is in the image of the transjective components of $\D^b(\mod B)$.
    On the other hand, such objects in the image are in fact $(a+1)$-cluster tilting objects by Lemma \ref{transjective}.
    We therefore conclude that the shifts of $\Pi$ are all the $(a+1)$-cluster tilting objects, and hence we obtain the assertion.
	%Notice next by Lemma \ref{indec} that we only have to check if each indecomposable object in the tranjective components is $(a+1)$-cluster tilting. But it is easy to see that these objects are precisely the image of the shifts of $\Pi\in\per\Pi$ in $\C(\Pi)$, thus they are indeed such objects. We therefore obtain the desired result.
	%In this case the algebra $A$ in Theorem \ref{orbit}(2) is a direct product of $p$ copies of the algebra $\begin{pmatrix}\Pi^0&0\\\Pi^{-p}&\Pi^0\end{pmatrix}$ which is nothing but the path algebra of the Kronecker quiver $Q\colon\xymatrix{\circ\ar@2[r]&\circ}$. Also by \cite[Section 5]{ha4} the folded cluster category $\C(\Pi)=\C^{(1/2p)}_{2p+1}(A)$ is equivalent to $\C^{(1/2)}_{2p+1}(kQ)=\D^b(\mod kQ)/-\lotimes_{kQ}\tau^{-1/2}[p]$ for $\tau^{-1/2}=-\lotimes_{kQ}\begin{pmatrix}\Pi^{-p}&\Pi^0\\\Pi^{-2p}&\Pi^{-p}\end{pmatrix}$. By \cite[Corollary 7.5(5)]{IYo}, we see that $(2p+1)\text{-}\ctilt\C(\Pi)=\{\Pi[i]\mid i\in\Z\}$. This shows that the map $\silt\Pi\to(2p+1)\text{-}\ctilt\C(\Pi)$ is surjective, if fact, bijective, and that the map $\silt\Pi\cap\F\to(2p+1)\text{-}\ctilt\C(\Pi)$ is not surjective.
\end{proof}

%\begin{Ex}
%	We regard the polynomial ring $\Pi=k[x_1,\ldots,x_n]$ as a DG algebra with $\deg x_i=-a_i<0$ and zero differential. We assume that $a=\sum_{i=1}^na_i$ is odd. Then $\Pi$ is an $(n+a)$-CY DG algebra by \cite{MGYC}\cite[Theorem 5.2]{ha3}. By Proposition \ref{H0=k}(\ref{m})$\Leftrightarrow$(\ref{qis}), we immediately deduce that $\Pi$ is $\F$-liftable if and only if $n=1$.
%\end{Ex}

\subsection{Non-$\F$-liftable example: Folded cluster categories of type $A_2$}
Consider the path algebra $A$ of the type $A_2$ quiver $\xymatrix{ 1\ar[r]&2}$. The Auslander-Reiten quiver of its derived category is as follows.
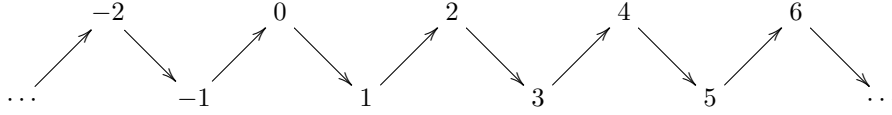
\[ \xymatrix@!R=3mm@!C=3mm{
	&-2\ar[dr]&&0\ar[dr]&&2\ar[dr]&&4\ar[dr]&&6\ar[dr]&\\
	\cdots\ar[ur]&&-1\ar[ur]&&1\ar[ur]&&3\ar[ur]&&5\ar[ur]&&\cdots } \]
It has a square root of $\tau$ given by $i\mapsto i-1$ in the above Auslander-Reiten quiver.
For each $d\geq1$, let $\C_{2d+1}^{(1/2)}(A)=\D^b(\mod A)/\tau^{-1/2}[d]$ be the $2$-folded $(2d+1)$-cluster category. By \cite[Theorem 5.2]{Ha7}, it is equivalent to the cluster category of $\Pi_{2d+2}^{(1/2)}(A)$, which is presented by the following dg path algebra below (see \cite[Proposition 8.6]{Ha7}). By \cite[Proposition 8.4]{Ha7} this dg algebra is $(2d+2)$-Calabi-Yau if $d$ is odd.
\[
\xymatrix@R=1.5mm@C=1mm{
	\\
	\circ\ar@(ur,dr)[]^-a\ar@(dl,ul)[]^-b\\
	}
\qquad
\xymatrix@R=2mm{
	|a|=-d, \, |b|=-2d-1\\
	da=0,\, db=a^2 }
\]
The Auslander-Reiten quiver of $\C_{2d+1}^{(1/2)}(A)$ is obtained by replacing each vertex of that of $\D^b(\mod A)$ modulo $3d+1$, and each indecomposable object in $\C_{2d+1}^{(1/2)}(A)$ is $(2d+1)$-cluster tilting. 
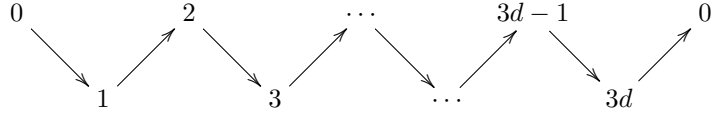
\[ \xymatrix@!R=3mm@!C=3mm{
	&0\ar[dr]&&2\ar[dr]&&\cdots\ar[dr]&&3d-1\ar[dr]&&0&\\
&&1\ar[ur]&&3\ar[ur]&&\cdots\ar[ur]&&3d\ar[ur]&& } \]
We therefore obtain the following.
\begin{Prop}\label{A_2}
Suppose that $d$ is odd and put $\D=\per\Pi$ and $\C=\C_{2d+1}^{(1/2)}(A)$.
\begin{enumerate}
\item $\silt\D=\{\Pi[i]\mid i\in\Z\}$, and $\silt\D\cap\F=\{\Pi[i]\mid 0\leq i\leq 2d\}$.
\item $d\text{-}\ctilt\C=\{\Pi[i]\mid i\in\Z/(3d+1)\Z\}$.
\end{enumerate}
Therefore, $\Pi$ is liftable, but not $\F$-liftable.
\end{Prop}

\section{Examples of silting-silting correspondence}

\subsection{Dg path algebras}
Let us recall the notion of {\it homotopically finitely presented} dg algebras, which are a class of dg algebras for which one can give an explicit computation of inverse dualizing complexes.

Let $Q$ be a finite quiver, on which we give a (cohomological) grading on each arrow. We assume that each arrow has a non-positive degree. Consider the dg algebra of the form $A=(kQ,d)$, thus the underlying graded algebra of $A$ is the path algebra $kQ=T_{kQ_0}(kQ_1)$, and it has a differential $d$. By finiteness and connectivity of $Q$, its set of arrows $Q_1$ has a finite filteration
\[ \emptyset=F_{-1}\subset F_0\subset F_1\subset\cdots\subset F_N=Q_1 \]
such that the differential takes $F_p$ to the subalgebra $T_{kQ_0}(kF_{p-1})$ spanned by $F_{p-1}$. Indeed, taking $F_i=\{\al\in Q_1\mid \deg \al\geq-i\}$, one must have $d(F_p)\subset T_{kQ_0}(kF_{p-1})$ since each arrow is of non-positive degree. Also, finiteness of $Q$ implies that the filteration is finite.
Such a dg algebra is called {\it homotopically finitely presented}, see \cite[Section 3.6]{Ke11}, \cite[Section 4.7]{Ke06}.

One has an explicit description of a bimodule cofibrant resolution for such dg algebras. We define a differential $d$ on the graded $A^e$-module $A\otimes_{kQ_0}kQ_1\otimes_{kQ_0}A$ as follows. Let $\be\colon A\to A\otimes_{kQ_0}kQ_1\otimes_{kQ_0}A$ be the unique $kQ_0^e$-derivation such that $\be(v)=1\otimes v\otimes 1$ for every $v\in Q_1$. Extend it to an $A^e$-linear map $A\otimes_{kQ_0}A\otimes_{kQ_0}A\to A \otimes_{kQ_0}kQ_1\otimes_{kQ_0}A$, and restrict it to $A\otimes_{kQ_0}kQ_1\otimes_{kQ_0}A$, which defines a differential by \cite[3.7]{Ke11}.
\begin{Prop}[{\cite[3.7]{Ke11}}]\label{resol}
	Let $A=(kQ,d)$ be a connective dg path algebra over a finite quiver $Q$. We endow the graded $A^e$-module $A\otimes_{kQ_0}kQ_1\otimes_{kQ_0}A$ with the differential defined above. Then the mapping cone of the morphism
	\[ \xymatrix{ A\otimes_{kQ_0}kQ_1\otimes_{kQ_0}A\ar[r]&A\otimes_{kQ_0}A, &1\otimes v\otimes1\ar@{|->}[r]&v\otimes1-1\otimes v } \]
	gives a cofibrant resolution of $A$ over $A^e$. In particular, $A$ is smooth.
\end{Prop}

This bimodule resolution gives an easy criterion for $A\in\per A$ to be $d$-silting. We denote by $kQ_{\geq n}$ the ideal of $kQ$ generated by paths of length $\geq n$.
\begin{Prop}\label{gldim}
	Let $A=(kQ,d)$ be a connective dg path algebra of a finite quiver $Q$ such that $d$ takes each arrow to $kQ_{\geq2}$. Then $A\in\per A$ is $d$-silting if and only if every arrow has degree at least $-d+1$.
\end{Prop}
\begin{proof}
	By Proposition \ref{resol} we see that the inverse dualizing bimodule $\RHom_{A^e}(A,A^e)$ is quasi-isomorphic to the mapping cocone of the morphism
	\[ \xymatrix{ A\otimes_{kQ_0}kQ_0^\vee\otimes_{kQ_0}A\ar[r]&A\otimes_{kQ_0}kQ_1^\vee\otimes_{kQ_0}A }, \]
	where $(-)^\vee$ is the bimodule dual $\Hom_{kQ_0^e}(-,kQ_0^e)$. Note that $kQ_1^\vee$ is concentrated in degree $\leq n$ for $n:=\max\{-\deg\al \mid \al\in Q_1\}$. It follows that $\RHom_{A^e}(A,A^e)$ is concentrated in degree $\leq n+1$, and hence we get the ``if'' part.
	
	Suppose conversely that $\min\{\deg \al\mid \al\in Q_1\}=-n$, and we show that $H^{n+1}\RHom_{A^e}(A,A^e)\neq0$. (We understand this as a trivial statement when $Q_1$ is empty.) Let $S=kQ_0$, which we regard as a semisimple $A$-module. It is enough to show that $\Hom_{\D(A)}(S\lotimes_A\RHom_{A^e}(A,A^e)[n+1],S)\neq0$, which is equivalent to $\Hom_{\D(A)}(S,S[n+1])\neq0$ by Serre duality \cite[Lemma 4.1]{Ke08}.	Appplying $S\lotimes_A-$ to the resolution in Proposition \ref{resol}, we get a triangle
	\[ \xymatrix{ P\ar[r]&A\ar[r]& S } \]
	with $P= kQ_1\otimes_{kQ_0}A$ as a graded $A$-module with the differential $d(v\otimes 1)=\sum \la u_1\otimes u_2\cdots u_m$ for $v\in Q_1$ whenever $dv=\sum\la u_1\cdots u_m$ for some $u_1,\ldots, u_m\in Q_1$ and $\la\in k$. 
	Applying $\RHom_A(-,S)$ yields a triangle
	\[ \xymatrix{ \RHom_A(S,S)\ar[r]&\cHom_A(A,S)\ar[r]&\cHom_A(P,S) }. \]
	Note that by assumption we have $m\geq2$ for each term in $dv=\sum\la u_1\cdots u_m$. It follows that in the above triangle the second map is $0$. Since $Q$ contains an arrow of degree $-n$, we deduce that $H^{n+1}\RHom_A(S,S)\hookleftarrow H^n\cHom_A(P,S)\neq0$.
\end{proof}

\subsection{$2$-Calabi-Yau completions of Dynkin type}
Let $Q$ be a Dynkin quiver. We discuss the relationship between $1$-silting objects over $kQ$ and silting objects over the $2$-Calabi-Yau completion $\Pi=\Pi_2(kQ)$, that is, the image of the embedding given in Theorem \ref{from d-silting to ctilt}:
\begin{equation}\label{1-silt}
	 \xymatrix{-\Lotimes_{kQ}\Pi\colon \silt^1kQ\ar[r]&\silt\Pi }.
\end{equation} %We shall observe that these behave in a different way from higher dimensions. 
By Proposition \ref{gldim}, the elements of $\silt^1kQ$ are precisely tilting objects in $\per kQ$ whose endomorphisms algebras are hereditary.

Note that the set $\silt\Pi$ can be described as the braid group of $Q$ by the following recent result by Mizuno-Yang \cite{MY}.
Recall that the {\it braid group} associated to a quiver $Q$ is the group $B=B_Q$ whose generators are $b_i$, $i\in Q_0$, with relations $b_ib_jb_i=b_jb_ib_j$ whenever the vertices $i$ and $j$ are connected by precisely one arrow, and $b_ib_j=b_jb_i$ if there is no arrow between $i$ and $j$. The set of {\it positive braids} is the submonoid of $B^+=B^+_Q$ of $B$ generated by the $b_i$'s. We regard $B$ as a poset by setting $a\geq b$ if and only if $ab^{-1}\in B^+$.

For the $2$-Calabi-Yau completion $\Pi=\Pi_2(kQ)$ of the path algebra of a quiver $Q$, let $S_i$ be the simple $\Pi$-module associated to the vertex $i$ of $Q$. We denote by $I_i$ the mapping cocone of $\Pi\to S_i$, which is a silting object in $\per\Pi$.
\begin{Thm}[\cite{MY}]\label{MY}
Let $Q$ be a Dynkin quiver, and let $\Pi=\Pi_2(kQ)$ be the $2$-Calabi-Yau completion of $kQ$. Then the assignment $b_i\mapsto I_i$ extends to an order-reversing isomorphism of posets
\begin{equation}\label{B silt}
B\xsimeq\silt\Pi.
\end{equation}
\end{Thm}
The aim of this subsection is to describe the image of the embedding \eqref{1-silt} as a subset of $B$.

%First, notice that in view of Proposition \ref{gldim}, a $1$-silting object over $kQ$ is nothing but a tilting complex whose endomorphism ring is hereditary. Therefore we get the following bijection.

The first step is to give a combinatorial description of the set $\silt^1kQ$. 
%For a finite connected acyclic quiver $Q$ and a sink $i\in Q_0$, we denote by $\mu_i^+(Q)$ the reflection of $Q$ at $i$. Similarly, if $i$ is a source, we denote the reflection at $i$ by $\mu_i^-(Q)$.
%If $i$ is either a sink or a source, we set
%\[ \mu_i=\begin{cases} \rho_i^+ &i \text{ is a sink}\\ \rho_i^- & i \text{ is a source} \end{cases}. \]
Let $\sect\Z Q$ be the set of sections of the translation quiver $\Z Q$. Then the sink/source reflections, or mutations, act also on $\sect\Z Q$ as follows.
If $T\in\sect\Z Q$ and $i$ is either a sink or a source in $T$, we define $\si_i(T)$ to be the reflection at $i$, precisely, if $i$ is a sink then $\si_i(T)$ is obtained by replacing the vertex $i$ of $S$ by $\tau^{}(i)$, and dually for the source.
%Then the sink/source reflections act also on $\S$ as follows. %If $T\in\S$ and $i$ is a sink in $T$, then we define $\mu_i^+(T)\in\S$ as the full subquiver of $\Z Q$ obtained by replacing the vertex $i$ of $S$ by $\tau^{}(i)$. Similarly, if $i$ is a source, then $\mu_i^-(T)$ is defined as the full subquiver of $\Z Q$ obtained by replacing $i$ by $\tau^{-1}(i)$. Certainly, the reflected quivers $\mu_i^\pm(T)$ are again sections of $\Z Q$.

\begin{Lem}\label{section}
Let $Q$ be a Dynkin quiver.
There exists a bijection $\silt^1kQ\simeq\sect\Z Q$ which is compatible with the mutations.
\end{Lem}
\begin{proof}
    Recall that a $1$-silting object in $\per kQ$ is nothing but a tilting object whose endomorphism ring is hereditary. 
    This implies that there is a bijection $\silt^1kQ\simeq\sect\Z Q$. Clearly, the refection at sink/source in $\sect\Z Q$ corresponds to silting mutation. %in such a way that exchange triangles are AR triangles. 
Therefore the assertion follows.
\end{proof}

We prepare some more terminologies to state the result. 

\begin{Def}
A sequence $I=(i_1,i_2,\ldots,i_N)$ of vertices of a quiver $Q$ is called {\it admissble} if for every $1\leq j\leq N$, the vertex $i_j$ is either a sink or a source in the quiver $Q^{(j-1)}$, where $Q^{(j)}$ is defined inductively by $Q^{(0)}=Q$, and $Q^{(j)}=\si_j(Q^{(j-1)})$.
In this case, we define $(\e_1,\ldots,\e_N)\in\{\pm1\}^N$, by setting $\e_j=+1$ when $i_j$ is a sink in $Q^{(j-1)}$ and $\e_j=-1$ when $i_j$ is a source in $Q^{(j-1)}$. We put
\[b_I:=b_{i_N}^{-\e_N}\cdots b_{i_1}^{-\e_1}\in B\ \mbox{ and }\ \si_I(Q):=\si_{i_N}\cdots\si_{i_1}(Q)\in \sect\Z Q.\]
%for the elements of the braid group.
%For brevity, we will write $\si_I:=\si_{i_N}\cdots\si_{i_1}$ for an admissible sequence $I=(i_1,\ldots,i_N)$ of vertices of $Q$. Similarly, for any admissible sequence $I=(i_1,\ldots,i_N)$, 
\end{Def}

We are ready to state our results.
%We denote by $Q$ a fixed element in $\sect\Z Q$ which is isomorphic to $Q$.

\begin{Thm}\label{braid}
Let $Q$ be a Dynkin quiver.
%We fix $Q\in\S$, an initial section (isomorphic to) $Q$.
\begin{enumerate}
\item There exists a map $F\colon\sect\Z Q\to B$ such that for every admissble sequence $I$ of vertices of $Q$, we have $F(\si_I(Q))=b_I$.
\item The image of the map $-\Lotimes_{kQ}\Pi\colon\silt^1kQ\to\silt\Pi\simeq B$ given in \eqref{1-silt} and \eqref{B silt} coincides with the image of the map $F\colon\sect\Z Q\to B$.
\end{enumerate}
\end{Thm}
\begin{proof}
	We define the map $F\colon\sect\Z Q\to B$ as the composite
	\[ \xymatrix{ \sect\Z Q\ar@{-}[r]^-\simeq& \silt^1kQ\ar[rr]^-{-\lotimes_{kQ}\Pi}&&\silt\Pi&B\ar[l]_-\simeq }, \]
	where the first map is the one in Lemma \ref{section}, the second one is \eqref{1-silt}, and the last one is Theorem \ref{MY}. The first map is compatible with sink/source mutations by Lemma \ref{section}, and so is the second one by Proposition \ref{mutation compatible}. Furthermore, the last map is also compatible with the mutations by \cite[Theorem 1.1(2)]{MY}. The assertions follow immediately from these observations.
\end{proof}	

\begin{Ex}
Let $Q\colon\xymatrix{ 1\ar[r]&2}$ be the quiver of type $A_2$. As above, the objects of its derived category are described as its Auslander-Reiten quiver below.
\begin{equation}\label{D(A_2)}
	\xymatrix@!R=3mm@!C=3mm{
	&-2\ar[dr]&&0\ar[dr]&&2\ar[dr]&&4\ar[dr]&&6\ar[dr]&\\
	\cdots\ar[ur]&&-1\ar[ur]&&1\ar[ur]&&3\ar[ur]&&5\ar[ur]&&\cdots }
\end{equation}
We identify the $1$-silting objects and the of sections of $\Z Q$ as $\{ i\oplus (i+1)\mid i\in\Z\}=\sect\Z Q$, where we fix the initial section as $Q:=1\oplus 2\in\sect\Z Q$. Then every $1$-silting object can be written either as $\mu_*^-\cdots\mu_2^-\mu_1^-(kQ)$ or as $\mu_*^+\cdots\mu_1^+\mu_2^+(kQ)$. By Theorem \ref{braid}, the corresponding elements in the braid group is
\[ \{(b_2b_1)^i,\ b_1(b_2b_1)^i\mid i\in\Z\}.\]
%\cup\{ b_\ast^{-1}\cdots b_1^{-1}b_2^{-1}\}. \]
\end{Ex}

\subsection{Calabi-Yau completions of type $A_2$}\label{section: A_2}
Let $A$ be the path algebra of type $A_2$. The objects of its derived category are described as its Auslander-Reiten quiver as in \eqref{D(A_2)} above. 
%\[ \xymatrix@!R=3mm@!C=3mm{
%	&-2\ar[dr]&&0\ar[dr]&&2\ar[dr]&&4\ar[dr]&&6\ar[dr]&\\
%	\cdots\ar[ur]&&-1\ar[ur]&&1\ar[ur]&&3\ar[ur]&&5\ar[ur]&&\cdots } \]
The silting objects of $\per A$ are $\{ i\oplus j\mid j-i=3n+1 \text{ for some } n\geq0\}$, and the Hasse quiver of $\silt A$ is as follows.
\begin{equation}\label{siltA}
	\xymatrix@!C=1mm{
		&\cdots\ar[drrr]\ar[rr]&&-1\oplus0\ar[drrr]\ar[rr]&&0\oplus1\ar[drrr]\ar[rr]&&1\oplus2\ar[drrr]\ar[rr]&&2\oplus3\ar[drrr]\ar[rr]&&3\oplus4\ar[drrr]\ar[rr]&&4\oplus5\ar[rr]&&\cdots\\
		\cdots\ar[urrr]\ar[drrr]&&-3\oplus1\ar[urrr]\ar[drrr]&&-2\oplus2\ar[urrr]\ar[drrr]&&-1\oplus3\ar[urrr]\ar[drrr]&&0\oplus4\ar[urrr]\ar[drrr]&&1\oplus5\ar[urrr]\ar[drrr]&&2\oplus6\ar[urrr]\ar[drrr]&&\cdots&\\
		&-5\oplus2\ar[urrr]\ar[drrr]&&-4\oplus3\ar[urrr]\ar[drrr]&&-3\oplus4\ar[urrr]\ar[drrr]&&-2\oplus5\ar[urrr]\ar[drrr]&&-1\oplus6\ar[urrr]\ar[drrr]&&0\oplus7\ar[urrr]\ar[drrr]&&1\oplus8&&\cdots\\
		\cdots\ar[urrr]&&\cdots\ar[urrr]&&\cdots\ar[urrr]&&\cdots\ar[urrr]&&\cdots\ar[urrr]&&\cdots\ar[urrr]&&\cdots\ar[urrr]&&\cdots&
	}
\end{equation}
It is easy to see that when $j-i=3n+1$, the derived endomorphism algebra $\REnd_A(i\oplus j)$ is the dg path algebra $\xymatrix{ i\ar[r]^-{-n}&j}$ with the trivial differential. It follows from \ref{gldim} that for each $d\geq1$, the $d$-silting objects are precisely the silting objects in the first $d$ rows of \eqref{siltA}.

Next, consider the $(d+1)$-Calabi-Yau completion $\Pi=\Pi_{d+1}(A)$ of $A$ which is presented by the dg path algebra
\[ \xymatrix{ 1\ar@<2pt>[r]^-0\ar@(ul,dl)[]_-{-d}&2\ar@<2pt>[l]^-{-d+1}\ar@(ur,dr)[]^-{-d} } \]
with certain differential.
We will denote by $P_i$ the image of $i\in\per A$ under the induction functor $-\lotimes_A\Pi$. Thus, if we let $A=1\oplus 2$ in $\per A$, then $\Pi=P_1\oplus P_2$, and we also have the formulas $P_3=\cone(P_1\to P_2)$ and $P_{n+3}=P_n[1]$.

Let us describe the Hasse quiver of $\silt\Pi$ below. We refer to \cite[Section 10]{KQ} for a description of simple-minded collections using the same pictures.
For this we introduce some notations. We denote by $\TT_3$ the $3$-regular tree, and for each $n\geq0$ we let $\TT_3^{(n)}$ the graph obtained from $\TT_3$ by inserting $n$ vertices to each of its edge. For example, $\TT_3$ is the graph of the left below, and on the right we have $\TT_3^{(1)}$.
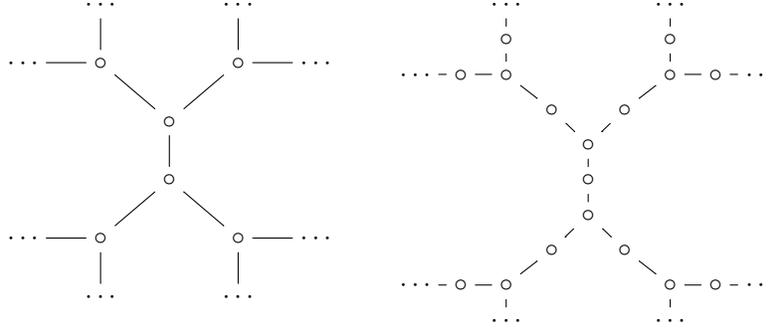
\[
\xymatrix@R=1mm@C=1mm{
	&&\cdots&&&&\cdots&&\\
	\\
	\cdots&&\ar@{-}[ll]\circ\ar@{-}[uu]&&&&\circ\ar@{-}[uu]\ar@{-}[rr]&&\cdots\\
	&&&&&&&&\\
	&&&&\circ\ar@{-}[uurr]\ar@{-}[uull]\ar@{-}[dd]&&&&\\
	&&&&&&&&\\
	&&&&\circ\ar@{-}[ddrr]\ar@{-}[ddll]&&&&\\
	&&&&&&&&\\
	\cdots&&\circ\ar@{-}[ll]\ar@{-}[dd]&&&&\circ\ar@{-}[dd]\ar@{-}[rr]&&\cdots\\
	\\
	&&\cdots&&&&\cdots&&}
\qquad
\xymatrix@R=1mm@C=1mm{
	&&\cdots&&&&\cdots&&\\
	&&\circ\ar@{-}[u]&&&&\circ\ar@{-}[u]&&\\
	\cdots&\circ\ar@{-}[l]&\ar@{-}[l]\circ\ar@{-}[u]&&&&\circ\ar@{-}[u]\ar@{-}[r]&\circ\ar@{-}[r]&\cdots\\
	&&&\circ\ar@{-}[ul]&&\circ\ar@{-}[ur]&&&\\
	&&&&\circ\ar@{-}[ur]\ar@{-}[ul]\ar@{-}[d]&&&&\\
	&&&&\circ\ar@{-}[d]&&&&\\
	&&&&\circ\ar@{-}[dr]\ar@{-}[dl]&&&&\\
	&&&\circ\ar@{-}[dl]&&\circ\ar@{-}[dr]&&&\\
	\cdots&\circ\ar@{-}[l]&\circ\ar@{-}[l]\ar@{-}[d]&&&&\circ\ar@{-}[d]\ar@{-}[r]&\circ\ar@{-}[r]&\cdots\\
	&&\cdots&&&&\cdots&&}
\]
We have the following description of the global structure of the exchange graph of $\silt\Pi_{d+1}(A)$.
\begin{Prop}\label{map}
Let $d\geq2$. There is a surjection $\silt\Pi_{d+1}(A)\to\TT_3^{(d-2)}$ of graphs such that the fiber of each vertex in $\TT_3^{(d-2)}$ is $\Z$.
\end{Prop}
\begin{proof}
	We write $\Pi=\Pi_{d+1}(A)$ for brevity.
	One can draw a part of the Hasse quiver of $\silt\Pi$ as follows. Here, we omit the direct sum symbol, we put $Q=\cone(P_2[d-1]\to P_1)=\cone(P_{3d-1}\to P_1)$, and there are $d-1$ down-right going arrows, say, from $P_1P_2$ to $P_{3d-1}P_1$. 
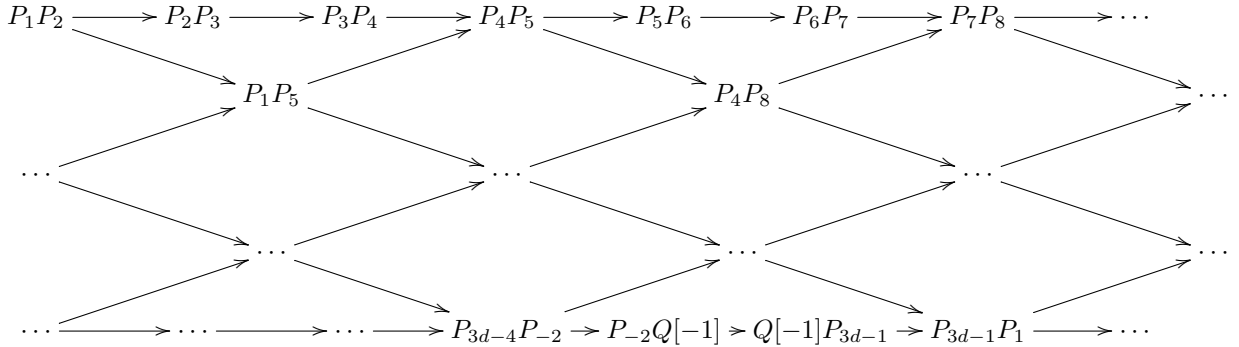
\begin{equation}\label{siltPi_d+1}
	\xymatrix@!C=2mm@!R=2mm{
		P_1P_2\ar[drrr]\ar[rr]&&P_2P_3\ar[rr]&&P_3P_4\ar[rr]&&P_4 P_5\ar[drrr]\ar[rr]&& P_5 P_6\ar[rr]&&P_6 P_7\ar[rr]&&P_7P_8\ar[rr]\ar[drrr]&&\cdots\\
		&&&P_1P_5\ar[drrr]\ar[urrr]&&&&&&P_4P_8\ar[drrr]\ar[urrr]&&&&&&\cdots\\
		\cdots\ar[drrr]\ar[urrr]&&&&&&\cdots\ar[drrr]\ar[urrr]&&&&&&\cdots\ar[drrr]\ar[urrr]&&&\\
		&&&\cdots\ar[drrr]\ar[urrr]&&&&&&\cdots\ar[drrr]\ar[urrr]&&&&&&\cdots\\
		\cdots\ar[urrr]\ar[rr]&&\cdots\ar[rr]&&\cdots\ar[rr]&&P_{3d-4}P_{-2}\ar[urrr]\ar[rr]&& P_{-2}Q[-1]\ar[rr]&&Q[-1]P_{3d-1}\ar[rr]&&P_{3d-1}P_1\ar[rr]\ar[urrr]&&\cdots}
\end{equation}
	Similarly, there is a length $(d-1)$-path from $P_2P_3$ to $P_{3d}P_2$ (which is on a different horizontal sequence from the last row of in the above diagram), and also from $P_3P_4$ to $P_{3d+1}P_3$ (which is on a still different sequence).

	Now, define the map $\silt\Pi\to\TT_3^{(d-2)}$ as follows. We first declare the vertices $P_iP_{i+1}$ to be sent to (any of) a trivalent node in $\TT_3^{(d-2)}$. We next send $P_iP_{i+4}$ (the second line in \eqref{siltPi_d+1}) to each of the adjacent $3$ nodes in $\TT_3^{(d-2)}$ depending on the class of $i$ modulo $3$. One can inductively construct a map $\silt\Pi\to\TT_3^{(d-2)}$ with the claimed properties.
\end{proof}

\begin{Ex}
First let $d=2$, so we consider the $3$-Calabi-Yau completion $\Pi_3(A)$. In this case, the diagram \eqref{siltPi_d+1} is as follows.
%\begin{equation}\label{siltPi_3}
%	\xymatrix@!C=2mm@!R=2mm{
	%		&\cdots\ar[drrr]\ar[rr]&&P_2[-1]Q[-1]\ar[rr]&&Q[-1] P_1\ar[rr]&&P_1 P_2\ar[drrr]\ar[rr]&& P_2 Q\ar[rr]&&Q P_1[1]\ar[rr]&&\cdots&\\
	%		\cdots\ar[rr]&&R[-1]P_2\ar[rr]&&P_2 P_1[-1]\ar[rr]\ar[urrr]&&P_1[-1] R\ar[rr]&& R P_2[1]\ar[rr]&& P_2[1] P_1\ar[urrr]\ar[rr]&& P_1 R[1]\ar[rr]&&\cdots }
%\end{equation}
%\begin{equation}\label{siltPi_3}
%	\xymatrix@!C=2mm@!R=2mm{
%		\cdots\ar[drrr]\ar[rr]&&\ar[rr]&&\ar[rr]&&P_{-1}P_3\ar[drrr]\ar[rr]&&\ar[rr]&&\ar[rr]&&P_2P_6\ar[rr]&&\cdots\\
%		&\cdots\ar[drrr]\ar[rr]&&P_{-1}P_0\ar[rr]\ar[urrr]&&P_0 P_1\ar[rr]&&P_1 P_2\ar[drrr]\ar[rr]&& P_2 P_3\ar[rr]\ar[urrr]&&P_3 P_4\ar[rr]&&\cdots&\\
%		\cdots\ar[rr]&&Q[-2]P_2\ar[rr]&&P_2 P_{-2}\ar[rr]\ar[urrr]&&P_{-2} Q[-1]\ar[rr]&& Q[-1] P_5\ar[rr]&& P_5 P_1\ar[urrr]\ar[rr]&& P_1 Q\ar[rr]&&\cdots }
%\end{equation}
%\begin{equation}\label{siltPi_3}
%	\xymatrix@!C=2mm@!R=2mm{
	%		\cdots\ar[drrr]\ar[rr]&&\ar[rr]&&\ar[rr]&&P_2[-1]Q\ar[drrr]\ar[rr]&&\ar[rr]&&\ar[rr]&&P_2Q[1]\ar[rr]&&\cdots\\
	%		&\cdots\ar[drrr]\ar[rr]&&P_2[-1]Q[-1]\ar[rr]\ar[urrr]&&Q[-1] P_1\ar[rr]&&P_1 P_2\ar[drrr]\ar[rr]&& P_2 Q\ar[rr]\ar[urrr]&&Q P_1[1]\ar[rr]&&\cdots&\\
	%		\cdots\ar[rr]&&R[-1]P_2\ar[rr]&&P_2 P_1[-1]\ar[rr]\ar[urrr]&&P_1[-1] R\ar[rr]&& R P_2[1]\ar[rr]&& P_2[1] P_1\ar[urrr]\ar[rr]&& P_1 R[1]\ar[rr]&&\cdots }
%\end{equation}
\begin{equation}\label{siltPi_3}
	\xymatrix@!C=2mm@!R=2mm{
		&\cdots\ar[drrr]\ar[rr]&&P_{-1}P_0\ar[rr]&&P_0 P_1\ar[rr]&&P_1 P_2\ar[drrr]\ar[rr]&& P_2 P_3\ar[rr]&&P_3 P_4\ar[rr]&&\cdots&\\
		\cdots\ar[rr]&&Q[-2]P_2\ar[rr]&&P_2 P_{-2}\ar[rr]\ar[urrr]&&P_{-2} Q[-1]\ar[rr]&& Q[-1] P_5\ar[rr]&& P_5 P_1\ar[urrr]\ar[rr]&& P_1 Q\ar[rr]&&\cdots }
\end{equation}
Note that there are much more objects arrows, for example, a family on a horizontal line together with ``zig-zags'' connecting to the shifts of $P_0P_1$.
Similarly, each horizontal line is connected by such zig-zags to three horizontal lines.
We see that the global structure of $\silt\Pi_3(A)$ can be described as the $3$-regular tree below, where each vertex represents a horizontal line in \eqref{siltPi_3}, and each edge represents a zig-zag. This describes the map in Proposition \ref{map} above.
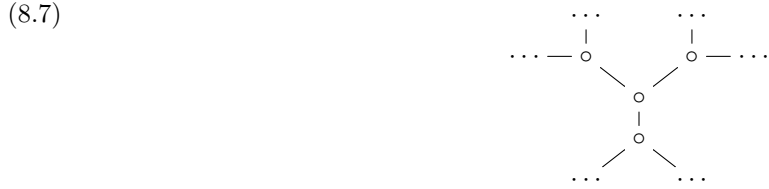
\begin{equation}\label{3bungi}
	\xymatrix@R=1.8mm@C=2mm{
		&\cdots&&\cdots&\\
		\cdots&\circ\ar@{-}[l]\ar@{-}[u]&&\circ\ar@{-}[u]\ar@{-}[r]&\cdots\\
		&&\circ\ar@{-}[ur]\ar@{-}[ul]\ar@{-}[d]&&\\
		&&\circ\ar@{-}[dl]\ar@{-}[dr]&&\\
		&\cdots&&\cdots& }
\end{equation}
We next discuss the image of $2$-silting objects of $\per A$ in $\per\Pi_3(A)$. Recall that the $2$-silting objects in $\per A$ are precisely the first two rows of \eqref{siltA}. The first row, which are the $1$-silting objects, are taken to the first row of \eqref{siltPi_3}. The second row, which are $2$-silting but not $1$-silting, are taken to the vertices adjacent to some vertices in the first row of \eqref{siltPi_3}.
Note that other silting objects over $\Pi$, for example, those containing $Q$, do not come from $\per A$. Indeed, the object $Q$ does not lie in the image of $-\lotimes_A\Pi\colon\per A\to\per\Pi$.
We see that only the `middle' vertex and the one third of the adjacent vertices in \eqref{3bungi} come from silting objects in $A$.
\end{Ex}

\begin{Ex}
Next let $d=3$. The Hasse quiver of $\Pi_4(A)$ looks as follows.
\begin{equation}\label{siltPi_4}
	\xymatrix@!C=2mm@!R=2mm{
	\cdots\ar[drrr]\ar[rr]&&P_{-1}P_0\ar[rr]&&P_0 P_1\ar[rr]&&P_1 P_2\ar[drrr]\ar[rr]&& P_2 P_3\ar[rr]&&P_3 P_4\ar[rr]&&P_4P_5\ar[rr]\ar[drrr]&&\cdots\\
	&&&P_2 P_{-2}\ar[drrr]\ar[urrr]&&&&&&P_1P_5\ar[drrr]\ar[urrr]&&&&&&\cdots\\
	\cdots\ar[urrr]\ar[rr]&&P_{-2}Q[-2]\ar[rr]&&Q[-2] P_5\ar[rr]&&P_5 P_{-2}\ar[urrr]\ar[rr]&& P_{-2}Q[-1]\ar[rr]&&Q[-1]P_8\ar[rr]&&P_8P_1\ar[rr]\ar[urrr]&&\cdots}
\end{equation}
We see again that the $d$-silting objects from \eqref{siltA} are taken the frist row and the middle `band' part in \eqref{siltPi_d+1}, as asserted in \ref{from d-silting to ctilt}(1). In the global picture below, their image consits of the black circles $\bullet$, and one third of the circled dots $\odot$.
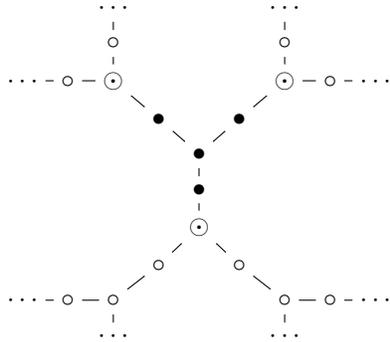
\[
\xymatrix@R=1mm@C=1mm{
	&&\cdots&&&&\cdots&&\\
	&&\circ\ar@{-}[u]&&&&\circ\ar@{-}[u]&&\\
	\cdots&\circ\ar@{-}[l]&\ar@{-}[l]\odot\ar@{-}[u]&&&&\odot\ar@{-}[u]\ar@{-}[r]&\circ\ar@{-}[r]&\cdots\\
	&&&\bullet\ar@{-}[ul]&&\bullet\ar@{-}[ur]&&&\\
	&&&&\bullet\ar@{-}[ur]\ar@{-}[ul]\ar@{-}[d]&&&&\\
	&&&&\bullet\ar@{-}[d]&&&&\\
	&&&&\odot\ar@{-}[dr]\ar@{-}[dl]&&&&\\
	&&&\circ\ar@{-}[dl]&&\circ\ar@{-}[dr]&&&\\
	\cdots&\circ\ar@{-}[l]&\circ\ar@{-}[l]\ar@{-}[d]&&&&\circ\ar@{-}[d]\ar@{-}[r]&\circ\ar@{-}[r]&\cdots\\
	&&\cdots&&&&\cdots&&}
\]
\end{Ex}

%Globally, the Hasse quiver is described as follows for $d=3$. For general $d$, we have $d-2$ nodes between trivalent nodes.
%\[ \xymatrix@R=1mm@C=1mm{
	%	&&\cdots&&&&\cdots&&\\
	%	&&\circ\ar@{-}[u]&&&&\circ\ar@{-}[u]&&\\
	%	\cdots&\circ\ar@{-}[l]&\ar@{-}[l]\circ\ar@{-}[u]&&&&\circ\ar@{-}[u]\ar@{-}[r]&\circ\ar@{-}[r]&\cdots\\
	%	&&&\circ\ar@{-}[ul]&&\circ\ar@{-}[ur]&&&\\
	%	&&&&\circ\ar@{-}[ur]\ar@{-}[ul]\ar@{-}[d]&&&&\\
	%	&&&&\circ\ar@{-}[d]&&&&\\
	%	&&&&\circ\ar@{-}[dr]\ar@{-}[dl]&&&&\\
	%	&&&\circ\ar@{-}[dl]&&\circ\ar@{-}[dr]&&&\\
	%	\cdots&\circ\ar@{-}[l]&\circ\ar@{-}[l]\ar@{-}[d]&&&&\circ\ar@{-}[d]\ar@{-}[r]&\circ\ar@{-}[r]&\cdots\\
	%	&&\cdots&&&&\cdots&&} \]


\thebibliography{99}
\bibitem[AI]{AI} T. Aihara and O. Iyama, {Silting mutation in triangulated categories}, J. London Math. Soc. 85 (2012) no.3, 633--668.
\bibitem[Am]{Am09} C. Amiot, {Cluster categories for algebras of global dimension 2 and quivers with potentional}, Ann. Inst. Fourier, Grenoble 59, no.6 (2009) 2525--2590.
\bibitem[AIR]{AIR} T. Adachi, O. Iyama, I. Reiten, $\tau$-tilting theory, Compos. Math. 150 (2014), no. 3, 415--452.
\bibitem[AO]{AOce} C. Amiot and S. Oppermann, {Cluster equivalence and graded derived equivalence}, Doc. Math. 19 (2014), 1155--1206.
\bibitem[ASS]{ASS} I. Assem, D. Simson, A. Skowronski, {Elements of the representation theory of associative algebras. Vol. 1. Techniques of representation theory}, London Mathematical Society Student Texts, 65. Cambridge University Press, Cambridge, 2006.
\bibitem[BD]{BD19} C. Brav and T. Dyckerhoff, {Relative Calabi-Yau structures}, Compos. Math. 155 (2019) 372--412.
\bibitem[BS]{BS} T. Bridgeland, I. Smith, {Quadratic differentials as stability conditions}, Publ. Math. Inst. Hautes Études Sci. 121 (2015), 155–-278.
\bibitem[BMRRT]{BMRRT} A. B. Buan, R. Marsh, M. Reineke, I. Reiten, and G. Todorov, {Tilting theory and cluster combinatorics}, Adv. Math. 204 (2006) 572--618.
\bibitem[BRT]{BRT} A. B. Buan, I. Reiten, and H. Thomas, {Three kinds of mutation}, J. Algebra 339 (2011), 97--113.
\bibitem[DI]{DI} E. Darp\"o, O. Iyama, {$d$-representation-finite self-injective algebras}, Adv. Math. 362 (2020), 106932.
\bibitem[DK]{DK} R. Dehy, B. Keller, \emph{On the combinatorics of rigid objects in 2-Calabi-Yau categories}, Int. Math. Res. Not. IMRN 2008, no. 11, Art. ID rnn029.
\bibitem[Dr]{Dr04} V. Drinfeld, {DG quotients of DG categories}, J. Algebra 272, (2004) 643--691.
\bibitem[Gu]{Guo} L. Guo, {Cluster tilting objects in generalized higher cluster categories}, J. Pure Appl. Algebra 215 (2011), no. 9, 2055--2071.
\bibitem[HKK]{HKK} F. Haiden, and L. Katzarkov, and M. Kontsevich, {Flat surfaces and stability structures}, Publ. Math. Inst. Hautes Études Sci. 126 (2017), 247--318.
\bibitem[Han1]{ha3} N. Hanihara, {Cluster categories of formal DG algebras and singularity categories}, Forum of Mathematics, Sigma (2022), Vol. 10:e35 1--50.
\bibitem[Han2]{ha4} N. Hanihara, {Morita theorem for hereditary Calabi-Yau categories}, Adv. Math. 395 (2022) 108092.
\bibitem[Han3]{Ha7} N. Hanihara, {Calabi-Yau completions for roots of dualizing dg bimodules}, arXiv:2412.18753.
\bibitem[HaI]{HaI} N. Hanihara, O. Iyama, Enhanced Auslander-Reiten duality and Morita theorem for singularity categories, arXiv:2209.14090.
\bibitem[Hap]{Hap88} D. Happel, {Triangulated categories in the representation theory of finite-dimensional algebras}, London Mathematical Society Lecture Note Series, 119. Cambridge University Press, Cambridge, 1988.
\bibitem[HIMO]{HIMO} M. Herschend, O. Iyama, H. Minamoto, and S. Oppermann, {Representation theory of Geigle-Lenzing complete intersections}, Mem. Amer. Math. Soc. 285 (2023), no. 1412, vii+141 pp.
\bibitem[Ig]{Ig} K. Igusa, {Notes on the no loops conjecture}, J. Pure Appl. Algebra 69 (1990), no. 2, 161--176.
\bibitem[Iy1]{Iy} O. Iyama, {Cluster tilting for higher Auslander algebras}, Adv. Math. 226 (2011) 1--61.
\bibitem[Iy2]{Iy2} O. Iyama, {Tilting Cohen-Macaulay representations}, Proceedings of the International Congress of Mathematicians--Rio de Janeiro 2018. Vol. II. Invited lectures, 125--162, World Sci. Publ., Hackensack, NJ, 2018.
\bibitem[IO]{IO} O. Iyama, S. Oppermann, {Stable categories of higher preprojective algebras}, Adv. Math. 244 (2013), 23--68.
\bibitem[IW]{IW} O. Iyama, M. Wemyss, {Maximal modifications and Auslander-Reiten duality for non-isolated singularities}, Invent. Math. 197 (2014), no. 3, 521--586.
\bibitem[IYa1]{IYa1} O. Iyama and D. Yang, {Silting reduction and Calabi-Yau reduction of triangulated categories}, Trans. Amer. Math. Soc. 370 (2018) no.11,  7861--7898.
\bibitem[IYa2]{IYa2} O. Iyama and D. Yang, {Quotients of triangulated categories and equivalences of Buchweitz, Orlov and Amiot--Guo--Keller}, Amer. J. Math. 142 (2020), no. 5, 1641--1659.
\bibitem[IYo]{IYo} O. Iyama and Y. Yoshino, {Mutation in triangulated categories and rigid Cohen-Macaulay modules}, Invent. math. 172, 117--168 (2008).
\bibitem[KaY]{KaY} M. Kalck and D. Yang, {Relative singularity categories I: Auslander resolutions}, Adv. Math. 301 (2016) 973-1021.
\bibitem[Ke1]{Ke98} B. Keller, {On the construction of triangle equivalences}, Derived equivalences for group rings, 155--176, Lecture Notes in Math., 1685, Springer, Berlin, 1998.
%\bibitem[Ke2]{Ke98-2} B. Keller, {Invariance and localization for cyclic homology of DG algebras}, J. Pure Appl. Algebra 123 (1998), no. 1-3, 223-273.
\bibitem[Ke3]{Ke05} B. Keller, {On triangulated orbit categories}, Doc. Math. 10 (2005), 551--581.
\bibitem[Ke4]{Ke06} B. Keller, {On differential graded categories}, Proceedings of the International Congress of Mathematicians, vol. 2, Eur. Math. Soc, 2006, 151--190.
\bibitem[Ke5]{Ke08} B. Keller, {Calabi-Yau triangulated categories}, in: {Trends in representation theory of algebras and related topics}, EMS series of congress reports, European Mathematical Society, Z\"{u}rich, 2008.
\bibitem[Ke6]{Ke10} B. Keller, {Cluster algebras, quiver representations and triangulated categories}, in: {Triangulated categories}, London Math. Soc. Lecture Note Ser. 375, Cambridge Univ. Press, Cambridge, 2010.
\bibitem[Ke7]{Ke11} B. Keller, {Deformed Calabi-Yau completions}, with an appendix by M. Van den Bergh, J. Reine Angew. Math. 654 (2011) 125--180.
\bibitem[Ke8]{Ke11+} B. Keller, {Erratum to ``Deformed Calabi-Yau completions''}, available at {https://webusers.imj-prg.fr/\verb|~|bernhard.keller/publ/KellerErratumCYCompletions.pdf}.
\bibitem[KL]{KeL} B. Keller and J. Liu, {On Amiot's conjecture}, arXiv:2311.06538.
\bibitem[KN]{KN} B. Keller and P. Nicholas, {Cluster hearts and cluster tilting objects}, in preparation. 
\bibitem[KR]{KRac} B. Keller and I. Reiten, {Acyclic Calabi-Yau categories}, with an appendix by M. Van den Bergh, Compos. Math. 144 (2008) 1332--1348.
\bibitem[KeY]{KeY} B. Keller and D. Yang, {Derived equivalences from mutations of quivers with potential}, Adv. Math. 226 (2011) 2118–-2168.
\bibitem[KQ1]{KQ} A. King and Y. Qiu, {Exchange graphs and $\Ext$ quivers}, Adv. Math. 285 (2015), 1106--1154.
\bibitem[KQ2]{KQ2} A. King and Y. Qiu, {Cluster exchange groupoids and framed quadratic differentials}, Invent. Math. 220 (2020), no. 2, 479--523.
\bibitem[KoY]{KY} S. Koenig, D. Yang, {Silting objects, simple-minded collections, $t$-structures and co-$t$-structures for finite-dimensional algebras}, Doc. Math. 19 (2014), 403--438.
\bibitem[L]{L} H. Lenzing, {Nilpotente Elemente in Ringen von endlicher globaler Dimension}, (German) Math. Z. 108 (1969), 313--324.
\bibitem[MGYC]{MGYC} X. Mao, X. Gao, Y. Yang, and J. Chen, {DG polynomial algebras and their homological properties}, Sci. China Math. 62 (2019), no.4, 629--648.
\bibitem[MY]{MY} Y. Mizuno and D. Yang, {Derived preprojective algebras and spherical twist functors}, arXiv:2407.02725.
\bibitem[M]{M} F. Muro, {Enhanced finite triangulated categories}, J. Inst. Math. Jussieu 21 (2022), no. 3, 741--783.
\bibitem[Pa]{Pa} Y. Palu, {Cluster characters for 2-Calabi-Yau triangulated categories}, Ann. Inst. Fourier (Grenoble) 58 (2008), no. 6, 2221–2248.
\bibitem[Pl]{Pl} P.-G. Plamondon, {Cluster algebras via cluster categories with infinite-dimensional morphism spaces}, Compos. Math. 147 (2011), no. 6, 1921--1954.
\bibitem[Pr]{Pr} M. Pressland, {Internally Calabi-Yau algebras and cluster-tilting objects}, Math. Z. 287 (2017), no. 1-2, 555–-585.
\bibitem[Ri]{Rie} C. Riedtmann, {Algebren, Darstellungsk\"ocher, Uberlagerugen und zur\"uck}, Comment. Math. Helvetici 55 (1980) 199--224.
\bibitem[Re]{Re} I. Reiten, {Cluster categories}, Proceedings of the International Congress of Mathematicians. Volume I, 558--594, Hindustan Book Agency, New Delhi, 2010.
\bibitem[SS]{SS2} D. Simson and A. Skowronski, {Elements of the representation theory of associative algebras. Vol. 2. Tubes and concealed algebras of Euclidean type}, London Mathematical Society Student Texts, 71. Cambridge University Press, Cambridge, 2007.
\bibitem[V]{V} M. Van den Bergh, {Noncommutative crepant resolutions, an overview}, ICM—International Congress of Mathematicians. Vol. II. Plenary lectures, 1354–-1391, EMS Press, Berlin, 2023.
\bibitem[We]{We} M. Wemyss, {Flops and clusters in the homological minimal model programme}, Invent. Math. 211 (2018), no. 2, 435-–521.
\bibitem[Wu]{Wu23} Y. Wu, {Relative cluster categories and Higgs categories}, Adv. Math. 424 (2023), Paper No. 109040, 112 pp.
\bibitem[Y]{Ye16} W. K. Yeung, {Relative Calabi-Yau completions}, arXiv:1612.06352.
\end{document}